\documentclass[reqno,10pt]{amsart}
\usepackage{amscd,amssymb}
\usepackage{latexsym}
\usepackage[all]{xy}

\def\into{\hookrightarrow}

\def\la{\leftarrow}

\def\toisom{\widetilde{\to}}

\def\.{,\dots ,}
\def\wt{\widetilde}
\def\wh{\widehat}
\def\ol{\overline}

\def\wtimes{\wh{\otimes}}
\def\circcirc{{\circ\circ}}

\def\Nr{\calN r}

\def\Aut{{\rm Aut}}

\def\Gal{{\rm Gal}}

\def\Spec{{\rm Spec}}

\def\Frac{{\rm Frac}}

\def\Bl{{\rm Bl}}

\def\dim{{\rm dim}}

\def\sing{{\rm sing}}

\def\Nor{{\rm Nor}}
\def\RZ{{\rm RZ}}

\def\st{{\rm st}}
\def\sm{{\rm sm}}

\def\rk{{\rm rk}}
\def\Hom{{\rm Hom}}

\def\calExt{{\mathcal Ext}}

\def\Res{{\rm Res}}

\def\gr{{\rm gr}}

\def\sh{{\rm sh}}
\def\mr{{\rm mr}}
\def\nr{{\rm nr}}
\def\deg{{\rm deg}}

\def\cha{{\rm char}}
\def\expcha{{\rm exp.char}}

\def\Span{{\rm Span}}

\def\trdeg{{\rm tr.deg.}}

\def\prof{{\rm prof}}

\def\bfA{{\bf A}}

\def\bfF{{\bf F}}

\def\bfN{{\bf N}}

\def\bfP{{\bf P}}
\def\bfQ{{\bf Q}}
\def\bfR{{\bf R}}

\def\bfZ{{\bf Z}}

\def\gtC{{\mathfrak C}}

\def\gtS{{\mathfrak S}}

\def\gtU{{\mathfrak U}}

\def\gtX{{\mathfrak X}}

\def\gtx{{\mathfrak x}}

\def\calA{{\mathcal A}}

\def\calF{{\mathcal F}}

\def\calH{{\mathcal H}}

\def\calL{{\mathcal L}}
\def\calM{{\mathcal M}}
\def\calN{{\mathcal N}}
\def\calO{{\mathcal O}}

\def\oA{{\ol A}}

\def\oC{{\ol C}}
\def\oD{{\ol D}}
\def\oE{{\ol E}}

\def\oK{{\ol K}}

\def\oP{{\ol P}}

\def\oS{{\ol S}}

\def\oX{{\ol X}}
\def\oY{{\ol Y}}
\def\oZ{{\ol Z}}

\def\ob{{\ol b}}

\def\of{{\ol f}}

\def\oh{{\ol h}}

\def\oll{{\ol l}}

\def\os{{\ol s}}

\def\ox{{\ol x}}

\def\oz{{\ol z}}

\def\tilA{{\wt A}}
\def\tilB{{\wt B}}
\def\tilC{{\wt C}}
\def\tilD{{\wt D}}
\def\tilE{{\wt E}}
\def\tilF{{\wt F}}

\def\tilK{{\wt K}}
\def\tilL{{\wt L}}

\def\tilR{{\wt R}}
\def\tilS{{\wt S}}
\def\tilT{{\wt T}}
\def\tilU{{\wt U}}

\def\tilX{{\wt X}}

\def\tila{{\wt a}}
\def\tilb{{\wt b}}
\def\tilc{{\wt c}}

\def\tilf{{\wt f}}
\def\tilg{{\wt g}}

\def\tili{{\wt i}}

\def\tilk{{\wt k}}
\def\till{{\wt l}}

\def\tilu{{\wt u}}
\def\tilv{{\wt v}}
\def\tilw{{\wt w}}
\def\tilx{{\wt x}}

\def\tilz{{\wt z}}

\def\hatA{{\wh A}}
\def\hatB{{\wh B}}
\def\hatC{{\wh C}}

\def\hatF{{\wh F}}

\def\hatK{{\wh K}}
\def\hatL{{\wh L}}

\def\hatR{{\wh R}}

\def\hatX{{\wh X}}

\def\hatk{{\wh k}}

\def\hatx{{\wh x}}

\def\ocalO{{\ol\calO}}

\def\ogtC{{\ol\gtC}}

\def\ogtx{{\ol\gtx}}

\def\hatcalA{{\wh\calA}}

\def\hatcalO{{\wh\calO}}

\def\kcirc{k^\circ}
\def\lcirc{l^\circ}

\def\Acirc{A^\circ}

\def\Fcirc{F^\circ}

\def\Kcirc{K^\circ}
\def\Lcirc{L^\circ}

\def\tilkcirc{\tilk^\circ}
\def\tillcirc{\till^\circ}

\def\tilAcirc{\tilA^\circ}

\def\hatFcirc{\hatF^\circ}

\def\hatKcirc{\hatK^\circ}
\def\hatLcirc{\hatL^\circ}

\def\calAcirc{\calA^\circ}

\def\Kcirccirc{K^{\circ\circ}}
\def\Lcirccirc{L^{\circ\circ}}

\def\kcirccirc{k^{\circ\circ}}
\def\lcirccirc{l^{\circ\circ}}

\def\hatFcirccirc{\hatF^\circcirc}

\def\hatKcirccirc{\hatK^\circcirc}

\def\oeta{{\ol\eta}}

\def\oomega{{\ol\omega}}

\def\tilpi{{\wt\pi}}

\def\alp{{\alpha}}

\def\veps{\varepsilon}
\def\ve{\veps}

\def\whka{{\wh{k^a}}}

\def\R+*{{\bf R^*_+}}

\newtheorem{theorsect}{Theorem}[section]
\newtheorem{propsect}[theorsect]{Proposition}
\newtheorem{lemsect}[theorsect]{Lemma}
\newtheorem{corsect}[theorsect]{Corollary}

\newtheorem{theor}{Theorem}[subsection]
\newtheorem{prop}[theor]{Proposition}
\newtheorem{lem}[theor]{Lemma}
\newtheorem{cor}[theor]{Corollary}

\theoremstyle{definition}

\newtheorem{remsect}[theorsect]{Remark}

\newtheorem{rem}[theor]{Remark}
\newtheorem{exam}[theor]{Example}

\newtheorem{assum}[theor]{Assumption}

\renewcommand*{\Nor}{{\calN r}}

\begin{document}

\title{Stable modification of relative curves}
\author{Michael Temkin}
\address{\tiny{School of Mathematics, Institute for Advanced Study, Princeton, NJ 08540, USA}}
\email{\scriptsize{temkin@math.ias.edu}}
\thanks{This article is based on a part of my Ph.D. thesis, and I want to thank my advisor Prof. V. Berkovich.
I am absolutely indebted to B. Conrad and I owe a lot to A. Ducros for pointing out various mistakes and
inaccuracies, and for many suggestions that led to two revisions of the paper and improved the exposition
drastically. I express my deep gratitude to Israel Clore Foundation for supporting my doctoral studies and to
Max Planck Institute for Mathematics where a part of this paper was written. A final revision was made when the
author was staying at IAS and supported by NFS grant DMS-0635607.}

\begin{abstract}
We generalize theorems of Deligne-Mumford and de Jong on semi-stable modifications of families of proper curves.
The main result states that after a generically \'etale alteration of the base any (not necessarily proper)
family of multipointed curves with semi-stable generic fiber admits a minimal semi-stable modification. The
latter can also be characterized by the property that its geometric fibers have no certain exceptional
components. The main step of our proof is uniformization of one-dimensional extensions of valued fields.
Riemann-Zariski spaces are then used to obtain the result over any integral base.
\end{abstract}


\maketitle

\section{Introduction}
\subsection{The motivation}
The stable reduction theorem of Deligne-Mumford \cite[2.7]{DM} states that for any smooth projective curve $C$
over the fraction field $K$ of a discrete valuation ring $R$ there exists a finite separable extension $L$ of
$K$ such that $C\otimes_K L$ can be extended to a stable curve over the integral closure of $R$ in $L$. This
theorem plays a key role in the proof of properness of the moduli space of stable $n$-pointed curves of genus
$g$. In its turn, the latter implies the following generalization of the Deligne-Mumford theorem (stable
extension theorem): for any proper stable curve $C$ over an open dense subscheme of a quasi-compact
quasi-separated integral scheme $S$ there exists an alteration $S'\to S$ such that $C\times_S S'$ can be
extended to a proper stable curve over $S'$ (see \cite[1.6]{Del}).

A stronger semi-stable modification theorem was proved by de Jong in \cite{dJ}: for any proper curve $C$ over an
integral quasi-compact excellent scheme $S$ there exist an alteration $S'\to S$ and a modification $C'\to
C\times_S S'$ such that $C'$ is a proper semi-stable curve over $S'$. De Jong's proof is also based on existence
and properness of the moduli spaces. Naturally, de Jong's theorem leads to the following two questions. Is it
true that the same result takes place for not necessarily proper curves $C$ over $S$? (Of course, in that case
$C'$ is not required to be proper over $S'$.) And, is it true that there is a minimal semi-stable modification?

The main result of the current paper is stable modification theorem formulated in \S\ref{mainsubsec}. This
theorem strengthens de Jong's theorem in a few aspects; in particular, it answers affirmatively both above
questions. In addition, our work is not based on \cite{DM} and \cite{dJ}, and we thereby reprove their results
including the stable reduction theorem. The main ingredient in proving the stable modification theorem are
Theorems \ref{valunif} and \ref{fieldunif} on uniformization of valued fields. These Theorems are of their own
interest; in particular, Theorem \ref{fieldunif} is used in a subsequent work \cite{Tem3} to establish
inseparable local uniformization of varieties of positive characteristic. We refer to Remark \ref{unifrem} for
more comments on the connections between stable reduction and uniformization of valued fields.

\subsection{The main results}\label{mainsubsec}
To formulate our main results we have to first introduce some terminology. Let $S$ be a scheme. A {\em
multipointed $S$-curve} $(C,D)$ consists of flat finitely presented morphisms $C\to S$ and $D\to S$ of pure
relative dimensions one and zero, respectively, and of a closed immersion $D\to C$ over $S$ (the subscheme $D$
may be empty). Note that $C$ is not assumed to be $S$-proper or even $S$-separated. A {\em morphism}
$f:(C',D')\to (C,D)$ is a compatible pair of $S$-morphisms $f_C:C'\to C$ and $f_D:D'\to D$. A {\em modification
of} $(C,D)$ is a morphism $f$ in which both $f_C$ and $f_D$ are {\em modifications}, i.e. proper morphisms
inducing isomorphisms between schematically dense open subschemes. Furthermore, a multipointed $S$-curve $(C,D)$
is said to be {\em semi-stable} if $\phi:C\to S$ is {\em semi-stable} (i.e. $\phi$ is flat and its geometric
fibers have at most ordinary double points as singularities), $D\to S$ is \'etale, and $D$ is disjoint from the
non-smoothness locus of $C\to S$. Note that $(C,D)$ is semi-stable if and only if all its geometric fibers
$(C_\os,D_\os)$ are semi-stable multipointed $\os$-curves. Indeed, obviously $C\to S$ is semi-stable if and only
if its geometric fibers are so, and by the fiber criterion of \'etaleness the flat morphism $D\to S$ is \'etale
if and only if its geometric fibers are \'etale.

A {\em semi-stable modification of} $(C,D)$ is a modification $f:(C',D') \to (C,D)$ in which $(C',D')$ is
semi-stable. Finally, such a semi-stable modification is said to be {\em stable} if for any geometric point
$\os\to S$ the fiber $C'_\os$ has no {\em exceptional components}, i.e. irreducible components $Z$ which are
isomorphic to the projective line, have at most two points of intersection with $D'_\os\cup (C'_\os)_{sing}$,
and are contracted to a point in $C_\os$.

Now we are going to introduce a sheaf $\omega_{C/S}$. Similarly to \cite{DM} we will use Grothendieck's duality
theory for the sake of speeding things up, though one can define $\omega_{C/S}$ in a lengthier but much more
elementary way similarly to \cite[III.7.11]{Har} (after a Zariski localization, realize $C$ as a complete
intersection in $X=\bfA^n_S$ via $i:C\into X$, then set
$\omega_{C/S}=\calExt^{n-1}_X(i_*(\calO_C),\Omega^n_{X/S})$ and compute it explicitly by use of the Koszul
complex via the ideal of $C$ in $\calO_X$). Any semi-stable curve $\phi:C\to S$ is a relative locally complete
intersection, so if $\phi$ is separated then the complex $\phi^!(\calO_S)$ has a unique non-zero cohomology
sheaf which is invertible and is the dualizing sheaf. We denote the latter sheaf as $\omega_{C/S}$ and note that
its definition is local on $C$, since the non-zero sheaf of $\phi^!(\calO_S)$ is local on the source for any CM
morphism (see \cite[p.157]{conres}). Therefore, the definition of $\omega_{C/S}$ extends to non-separated
semi-stable curves as well. The following well known properties of $\omega_{C/S}$ (mentioned, for example, in
\cite[\S1]{DM} in the case of proper $\phi$) allow to compute the geometric fibers: (a) the sheaves
$\omega_{C/S}$ are compatible with the base changes $S'\to S$, (b) if $S=\Spec(k)$ for an algebraically closed
field $k$, $\pi:\tilC\to C$ is the normalization and $E=\pi^{-1}(C_\sing)$ then $\omega_{C/S}$ is the subsheaf
of $\pi_*(\omega_{\tilC/S}(E))=\pi_*(\Omega^1_{\tilC/S}(E))$ given by the conditions
$\Res_{x_i}(\nu)=\Res_{y_i}(\nu)$ for all pairs of different points $x_i,y_i\in E$ with common image in $C$. If
$(C,D)$ is semi-stable then $D$ is a Cartier divisor and we set $\omega_{(C,D)/S}:=\omega_{C/S}(D)$. It is well
known that for a proper semi-stable $n$-pointed $S$-curve $(C,D)$ with geometrically connected fibers stability
(over $S$) is equivalent to $S$-ampleness of the sheaf $\omega_{(C,D)/S}$. (This fact is used for constructing
moduli spaces of $n$-pointed nodal curves.) We now prove an analog of this ampleness result for stable
modifications by a slight adjustment of the classical proof.

\begin{theorsect}[Stable Modification Theorem: projectivity]\label{projtheor}
For any multipointed $S$-curve $(C,D)$ with a semi-stable modification $(C',D')$, the sheaf $\omega_{(C',D')/S}$
is $C$-ample if and only if the modification is stable. In particular, for any stable modification
$(C_\st,D_\st)\to(C,D)$ the modification $C_\st\to C$ is projective.
\end{theorsect}
\begin{proof}
Assume that the modification is stable. Since relative ampleness can be checked separately on each fiber by
\cite[$\rm IV_3$, 9.6.5]{ega}, it suffices to establish ampleness of the restriction of
$\omega=\omega_{(C',D')/S}$ onto the fiber $Z:=C'_x$ over a point $x\in C$. It is enough to consider the case
when $x$ is closed in its $S$-fiber because otherwise $Z$ is zero-dimensional and so $\omega|_Z$ is ample. If
$\os$ is a geometric point over $s=\phi(x)$ then $(C'_\os,D'_\os)\to(C_\os,D_\os)$ is a proper morphism with
semi-stable source and such that its fibers do not contain exceptional irreducible components. Also,
$\oomega:=\omega_{(C_\os,D_\os)/\os}$ is the pullback of $\omega$ by compatibility with base changes (property
(a) above). Ampleness of the restriction of $\omega$ onto $Z\into C'_s$ is equivalent to ampleness of the
restriction of $\oomega$ onto $\oZ\into C'_\os$ where $\oZ$ is the preimage of $Z$ viewed as a reduced scheme.

Now, we have a reduced $\os$-proper subscheme $\oZ\subset C'_\os$ not containing exceptional components of
$(C'_\os,D'_\os)$ and it suffices to show that $\oomega|_\oZ$ is ample. Obviously, we can deal with connected
components independently, so we can assume that $\oZ$ is connected. Only the case when $\oZ$ is a curve should
be treated and then $\oZ$ is a proper connected semi-stable curve. Let $\oD=\oZ\cap D'_\os$ and let $\oE$ be the
union of points that are smooth in $\oZ$ but are singular in $C'_\os$. The assumption that an irreducible
component $V\subset\oZ$ is not an exceptional component of $(C'_\os,D'_\os)$ is equivalent to the fact that $V$
is not an exceptional component (in the classical sense) of the $n$-pointed $\os$-curve $(\oZ,\oD\sqcup\oE)$.
Thus, $(\oZ,\oD\sqcup\oE)$ is stable and hence $\omega_{(\oZ,\oD\sqcup\oE)/\os}$ is ample by the classical
computation in the theory of stable $n$-pointed curves. It remains to note that by the property (b) above
$\oomega|_\oZ$ is isomorphic to $\omega_{(\oZ,\oD)/\os}(\oE)=\omega_{(\oZ,\oD\sqcup\oE)/\os}$.

If the modification is not stable then there exists a geometric fiber $(C'_\os,D'_\os)$ with an exceptional
irreducible component $\oZ$ contracted to a point $\ox\in C_\os$. Let $x\in C$ be the image of $\ox$. Since
$\oomega|_\oZ$ is not ample in this case (by the same classical computation), one can use the above reasoning to
show that already the restriction of $\omega$ onto $C'_x$ is not ample. Thus, $\omega$ is not $C$-ample.
\end{proof}

Other parts of our main result are given below and they will be proved in \S\ref{mainsec}. Starting with this
point we assume that $S$ is integral quasi-compact and quasi-separated with generic point $\eta$. By an {\em
$\eta$-modification} $(C',D')\to(C,D)$ of multipointed $S$-curves we mean a modification whose $\eta$-fiber
$(C'_\eta,D'_\eta)\toisom(C_\eta,D_\eta)$ is an isomorphism.

\begin{theorsect}[Stable Modification Theorem: uniqueness]\label{minstabtheor}
Assume that $S$ is normal and a multipointed $S$-curve $(C,D)$ admits stable $\eta$-modification
$f:(C_\st,D_\st)\to (C,D)$. Then this modification is minimal in the sense that any semi-stable modification
$(C',D') \to (C,D)$ goes through a unique $S$-morphism $(C',D')\to(C_\st,D_\st)$.
\end{theorsect}

\begin{corsect}\label{ccc}
In the situation of Theorem \ref{minstabtheor}, the stable $\eta$-modification $f$ is

(i) unique up to unique isomorphism;

(ii) an isomorphism over the semi-stable locus of $(C,D)\to S$.
\end{corsect}

\begin{corsect} \label{minstabcor}
Assume that $(C,D)$ admits a stable $\eta$-modification $(C_\st,D_\st)$ and a semi-stable modification
$(\oC,\oD)$. Then there exists a finite modification $S'\to S$ such that $(\oC,\oD)\times_S S'\to(C,D)\times_S
S'$ goes through a unique $S'$-morphism $(\oC,\oD)\times_S S'\to (C_\st,D_\st)\times_S S'$.
\end{corsect}

\begin{theorsect}[Stable Modification Theorem: existence]\label{stabmodtheor}
For any multipointed $S$-curve $(C,D)$ with semi-stable generic fiber $(C_\eta,D_\eta)$ there exist a
generically \'etale alteration $S'\to S$ and a stable $\eta$-modification $(C'_\st,D'_\st) \to (C,D)\times_S
S'$.
\end{theorsect}

\begin{corsect} \label{equivcor}
Assume that $S'$ is normal in Theorem \ref{stabmodtheor}. Then the natural actions of the groups $\Aut(S)$ and
$\Aut_S((C,D))\times\Aut_S(S')$ on $(C,D)\times_S S'$ lift equivariantly to $(C'_\st,D'_\st)$.
\end{corsect}

\subsection{Overview}
Our proof of the stable modification theorem originates in non-Archimedean analytic geometry. Namely, the stable
reduction theorem of Deligne-Mumford (in the form of Bosch-L\"utkebohmert \cite{BL1}) was used in
\cite[\S4.3]{Ber1} and \cite[\S3.6]{Ber2} to give a local description of smooth analytic curves. On the other
hand, it was clear to experts that the stable reduction theorem is a consequence of such a description and that
the latter is actually of a very elementary nature and would follow from a description of certain extensions of
analytic fields (i.e. complete valued fields of height one).

In general outline, our proof goes as follows. The core is Theorem \ref{fieldunif} which provides a description
of extensions of analytic fields $K/k$ with $k$ algebraically closed (or, more generally, deeply ramified in the
sense of Assumption \ref{assum}) and $K$ the completion of a finitely generated field of transcendence degree
one over $k$. Since some readers may be interested in following the algebra-geometric arguments with taking the
non-archimedean analytic results as a black box, we build our exposition accordingly. In \S\ref{fields} we
formulate the only analytic black box we need. This is Theorem \ref{fieldunif1}, which is a light version of
Theorem \ref{fieldunif}. Then we postpone until \S\ref{ansec} any work involving non-archimedean analytic
geometry. As a next step in the paper, we deduce Theorem \ref{valunif} from (temporarily black boxed) Theorem
\ref{fieldunif1}. Theorem \ref{valunif} is the main result of \S\ref{fields} and it describes general extensions
of valued fields of transcendence degree one.

\begin{rem}\label{unifrem}
(i) In this paper we only use the case when the ground valued field $k$ is algebraically closed. It was known to
experts that in this case the uniformization of valued fields is equivalent (up to some work) to the stable
reduction theorem. We prove uniformization directly, thus obtaining a new proof of the stable reduction theorem.
On the other side, one could use the stable reduction (over general valuation rings!) to get an indirect proof
of Theorem \ref{valunif}. This would shorten our argument, nearly by eliminating long but elementary
\S\ref{ansec}. Note for the sake of comparison that the only known alternative method for proving stable
reduction theorem in such generality is by use of moduli spaces of curves (so this involves stable reduction
over a DVR, Hilbert schemes, GIT and/or DM stacks, etc.).

(ii) The main advantage of uniformizing valued fields directly (in addition to making the proof much more
elementary) is that we are able to treat some cases of deeply ramified analytic ground fields, including the
perfect fields in the equicharacteristic case, in Theorem \ref{fieldunif}. It does not seem to be probable that
the latter case could be deduced from the stable reduction theorem. In particular, we will prove local
desingularization of varieties up to purely inseparable alteration in \cite{Tem3}, and the proof is ultimately
based on Theorem \ref{fieldunif} for perfect equicharacteristic ground fields.

(iii) In Appendix \ref{A3}, we briefly discuss other results on uniformization of one-dimensional valued fields
due to Grauert-Remmert, Matignon and Kuhlmann. It seems that none of those approaches applies when the ground
field is not algebraically closed.
\end{rem}

In \S\ref{valcurves} we use Theorem \ref{valunif} to prove the stable modification theorem in the case when the
base scheme $S$ is the spectrum of a valuation ring of finite height and with separably closed fraction field.
This case of the stable modification theorem is rather similar to the theory of minimal desingularization of
surfaces (see the Appendix). Curiously enough, we essentially use uniqueness of stable modification in order to
prove its existence. The situation is similar to the desingularization theory, where local desingularizations
are glued together using functoriality of the construction. Moreover, even when the initial $C$ is $S$-proper we
glue the stable modification from pieces that are not $S$-proper, and at this stage we exploit the fact that our
theorem treats all $S$-curves, including the non-proper ones (see the proof of Proposition \ref{sepclosprop}).

The general stable modification theorem (with an arbitrary base $S$) is deduced in \S\ref{mainsec} rather
easily. The main idea is that if a multipointed $S$-curve $(C,D)$ is chosen then for any valuation ring $\calO$
of the separable closure of the field of rational functions on $S$, there exist a generically \'etale alteration
$S'\to S$ and an open subscheme $U\subset S'$ such that $\calO$ is centered on $U$ and the main results for
$(C,D)$ hold already after base change to $U$. Then, using quasi-compactness of the Riemann-Zariski space of $S$
introduced in \S\ref{RZsec} and the uniqueness of the constructed stable modifications over $U$'s, we glue them
together into a stable modification over sufficiently large generically \'etale alteration of $S$.

It seems that our method of proving the stable modification theorem can be applied to many other birational
problems. So, we describe the method in an abstract general form in \S\ref{Pmodsec}, and then use it to prove in
Theorem \ref{redfibth} a particular case of the reduced fiber theorem \cite[2.1']{BLR} of
Bosch-L\"utkebohmert-Raynaud (in our case the base $S$ is integral). Also, we show in \S\ref{relRZsec} how to
define relative Riemann-Zariski spaces generalizing the classical ones. In \cite{Tem2}, these spaces will play a
critical role in re-proving the general reduced fiber theorem and generalizing the stable modification theorem
to the case of not integral $S$. It should be noted also that our method of applying relative and absolute
Riemann-Zariski spaces to birational geometry is very close to the method of K. Fujiwara and F. Kato, as
explained in their survey \cite{FK}.

The paper contains two appendixes. In Appendix \ref{A} we describe a similarity between our method and Zariski's
desingularization of surfaces, and compare the new proof of the stable reduction theorem with other proofs (I
know six published proofs). In Appendix \ref{B} we collect some known results on curves over separably closed
fields which are used in the paper.

We conclude the introduction with an interesting question which is not studied in the paper. Is stable
modification of $S$-curves functorial with respect to all $S$-morphisms? For the sake of comparison, we remark
that minimal desingularization of surfaces is functorial with respect to regular morphisms, but is not
functorial with respect to all morphisms. Nevertheless, I expect that the answer to the above question is
affirmative. Such result would indicate that despite large similarity between stable modification and minimal
desingularization of surfaces (see Appendix \ref{A}), the former is a "tighter" (and less subtle) construction.

\tableofcontents

\section{Uniformization of one-dimensional valued fields}
\label{fields}

\subsection{Basic theory of valued fields}\label{basicsec}
Unfortunately, it is not so easy to give a good reference to the theory of valued fields, so we are going to
recall in \S\ref{basicsec} some basic (and often well known) facts about valued fields. By a {\em valued field}
we mean a field $k$ provided with a non-Archimedean multiplicative valuation $|\ |:k^\times\to\Gamma_k$. Recall
that this means that $\Gamma_k$ is a totally ordered multiplicative group and $|\ |$ is a homomorphism
satisfying the strong triangle inequality $|x+y|\le\max(|x|,|y|)$. We write also $|0|=0$ and postulate that $0$
is smaller than any element of $\Gamma_k$. The map $|\ |:k\to\Gamma_k\cup\{0\}$ is often called a
non-Archimedean absolute value, but we will simply say {\em valuation} in the sequel. Two valuations $|\
|_i:k^\times\to\Gamma_i$, $i=1,2$ are called {\em equivalent} if there exists an ordered isomorphism between
their images $|k^\times|_i$ compatible with the valuations.

The set $\kcirc:=\{x\in k|\ |x|\le 1\}$ is a {\em valuation ring} (i.e. either $f\in\kcirc$ or $f^{-1}\in\kcirc$
for any $f\in k^\times$) called the ring of integers of $k$, its maximal ideal is $\kcirccirc:=\{x\in k|\
|x|<1\}$ and its residue field is $\tilk:=\kcirc/\kcirccirc$. Note that $(\kcirc)^\times$ is the kernel of $|\
|$, $k^\times/(\kcirc)^\times\toisom|k^\times|$ and the order on $|k^\times|$ is induced from the divisibility
in $k^\times$ with respect to $\kcirc$, i.e. $|a|\le |b|$ if and only if $a/b\in\kcirc$. In particular, $\kcirc$
defines the valuation up to an equivalence, and an alternative way to provide $k$ with a structure of a valued
field is to fix a valuation ring $\kcirc\subset k$ with $\Frac(\kcirc)=k$.

Any ring $A\subseteq k$ that contains $\kcirc$ is a valuation ring. Note that $A=\kcirc_m$ for the prime ideal
$m=\{x\in\kcirc|\ x^{-1}\notin A\}$. This gives a one-to-one correspondence between the points of
$\Spec(\kcirc)$ and the overrings $A\supseteq\kcirc$ in $k$. Also, there is a one-to-one correspondence
$A\mapsto|A^\times|$ between the overrings and the convex subgroups $\Gamma\subseteq|k^\times|$. (Thus, $A$
induces a valuation on $k$ with the group of values $|k^\times|/\Gamma$.) Recall that the {\em height} of a
valued field $k$ is defined as the cardinality of the set $\Spec(\kcirc)$ decreased by $1$ (following
\cite{Bou}, we prefer the notion of height, though one usually says the rank or the convex rank of the
valuation). We automatically provide any valued field of height one with the $\pi$-adic topology, where
$\pi\in\kcirccirc\setminus\{0\}$. (The topology is independent of the choice of $\pi$.) By an {\em analytic
field} $k$ we mean a complete valued field with a non-trivial valuation $|\ |:k\to\bfR_+$. Any analytic field is
{\em henselian} in the sense that its ring of integers $\kcirc$ is henselian.

Note that the topological space of the scheme $S=\Spec(\kcirc)$ is totally ordered with respect to
generalizations since the set of all prime ideals of $\kcirc$ (as well as the sets of all its ideals and
fractional ideals) is totally ordered with respect to inclusion. So, $k$ is of a finite height $h$ if and only
if $|S|$ is homeomorphic to the set $\{0\. h\}$ with the topology given by open subsets $\{0\. n\}$ for $-1\le
n\le h$.

\begin{lem}\label{finhtlem}
Let $R$ be a valuation ring of finite height and $S=\Spec(R)$. Then for any quasi-finite $S$-scheme $X$ with a
point $x\in X$, $X_x:=\Spec(\calO_{X,x})$ is an open subscheme of $X$. In particular, $\calO_{X,x}$ is finitely
generated over $R$. Likewise, if $X$ is finitely presented over $S$ then $\calO_{X,x}$ is finitely presented
over $R$.
\end{lem}
\begin{proof}
Since the set $|S|$ is finite, so is the set $|X|$. Replacing $X$ with a neighborhood of $x$ we can achieve that
$X=\Spec(A)$ is affine. The ring $\calO_{X,x}$ is a localization of $A$ along an infinite set, but since $X_x$
is obtained from $X$ by removing finitely many points, $\calO_{X,x}$ can be obtained by localizing $A$ by a
single element. The assertion of the Lemma follows.
\end{proof}

Given a valued field $k$, by a {\em valued $k$-field} we mean a field $l$ containing $k$ and provided with a
valuation $|\ |_l:l^\times\to\Gamma_l$ such that $\Gamma_k\subset\Gamma_l$ and $|\ |_l$ extends $|\ |_k$. In the
above situation we say that $l/k$ is an {\em extension of valued fields}. We use the standard notation
$e_{l/k}=\#(|l^\times|/|k^\times|)$ and $f_{l/k}=[\till:\tilk]$ (these cardinals can be infinite) and say that
$l/k$ is {\em immediate} if $e=f=1$. If $L/k$ is a finite extension of abstract fields, $\Lcirc$ is the integral
closure of $\kcirc$ in $L$ and $m_1\. m_n$ are its maximal ideals, then $\Lcirc_i:=\Lcirc_{m_i}$ for $1\le i\le
n$ are the extensions of $\kcirc$ to a valuation ring of $L$. A classical argument shows that $\sum_{i=1}^n
e_{L_i/k}f_{L_i/k}\le[L:k]$. It is obvious from the above that the valuation of $k$ uniquely extends to all
finite extensions $L/k$ if and only if each $\Lcirc$ is a local ring. It is easy to see that the latter happens
if and only if any connected finite $\kcirc$-scheme is local, that is, $\kcirc$ is henselian. In this case we
will also say that the valued field $k$ is {\em henselian}.  For an extension $l/k$ of valued fields we set also
$F_{l/k}=\trdeg(\till/\tilk)$ and $E_{l/k}=\dim_\bfQ\left((|l^\times|/|k^\times|)\otimes_\bfZ\bfQ\right)$. The
following Lemma is called Abhyankar's inequality.

\begin{lem}\label{Abhlem}
Any extension $l/k$ of valued fields satisfies $\trdeg(l/k)\ge E_{l/k}+F_{l/k}$.
\end{lem}
\begin{proof}
Choose a transcendence basis of $\till/\tilk$ and let $B_F\subset\lcirc$ be any lifting. Also, choose a subset
of $B_E\subset|l^\times|$ which is mapped bijectively onto a $\bfQ$-basis of
$|l^\times|/|k^\times|\otimes_\bfZ\bfQ$. Note that $B_E\cap B_F=\emptyset$ and set $B=B_E\sqcup B_F$. Then
$|B_E|=E_{l/k}$ and $|B_F|=F_{l/k}$, and it is easy to see that the monomial basis of $k[B]$ is orthogonal with
respect to the valuation on $l$. In particular, $k[B]\into l$ and hence $|B|\le\trdeg(l/k)$.
\end{proof}

\begin{cor}\label{approxcor}
(i) If a valued field $k$ is of finite absolute transcendence degree $d$ then the height of $k$ is bounded by
$d+1$.

(ii) Any valuation ring is a filtered union of its valuation subrings of finite height.
\end{cor}
\begin{proof}
(i) holds for the prime fields $\bfQ$ and $\bfF_p$ since all their valuation rings are $\bfZ_{(p)}$, $\bfQ$ and
$\bfF_p$. The general case of (i) follows by applying Abhyankar's inequality to the extension $k/\bfF$ where
$\bfF$ is the valued prime subfield of $k$. Using (i) we see that any valuation ring $\kcirc$ is the union of
its valuation subrings of finite height which are cut off from $\kcirc$ by finitely generated subfields of $k$.
\end{proof}

Many statements about schemes over arbitrary valuation rings can be reduced to the case of finite height using
approximation from Corollary \ref{approxcor}(ii). Next, we are going to introduce a technique of induction on
height that often makes it possible to reduce a problem on valuations to the height one case. Assume that the
field $\tilk$ is provided with a valuation $|\ |_\tilk$ and let $\tilkcirc$ be its valuation ring. Then the
preimage of $\tilkcirc$ in $\kcirc$ is a valuation ring $R$, and we say that the corresponding valuation of $k$
(unique up to equivalence) is {\em composed} from $|\ |_k$ and $|\ |_\tilk$. Topologically $\Spec(R)$ is
obtained by gluing $\Spec(\kcirc)$ and $\Spec(\tilkcirc)$ so that the closed point $\Spec(\tilk)$ of the first
space is pasted to the generic point of the second one. (We will not need it, but one can easily show that this
is a scheme-theoretic gluing in the sense that the four morphisms $\Spec(\tilk)\to\Spec(\kcirc)\to\Spec(R)$,
$\Spec(\tilk)\to\Spec(\tilkcirc)\to\Spec(R)$ are monomorphisms that form a {\em universally bi-Cartesian
square}, i.e. the square is both Cartesian and universally co-Cartesian; see \cite[\S2.3]{Tem2} for more
details.)

\begin{lem}\label{composlem}
Let $k$ be a valued field.

(i) For any overring $A$ with $\kcirc\subseteq A\subseteq k$ we have that $A$ is a valuation ring,
$m_A\subset\kcirc$, $A':=\kcirc/m_A$ is a valuation ring in $A/m_A$ and $\kcirc$ is composed from $A$ and $A'$.

(ii) If $k$ is of height larger than one then there exists $A$ as above so that both $A$ and $A'$ are of
positive height.

(iii) If $k$ is of finite height then it can be obtained by iterative composing valuations of height one.
\end{lem}
\begin{proof}
(i) is obvious. In (ii) we choose a prime ideal $p$ with $0\subsetneq p\subsetneq\kcirc$ and set $A=\kcirc_p$.
To prove (iii) we choose $A$ as in (ii) and note that the heights of $A$ and $A'$ are smaller than that of $k$,
hence we can apply induction on height.
\end{proof}

Next, we discuss basic ramification theory of algebraic extensions of valued fields. Such an extension $l/k$ is
called {\em unramified} if there exists a local injective $\kcirc$-homomorphism $\lcirc\into(\kcirc)^\sh$ with
target being the strict henselization of $\kcirc$. In particular, $k^\nr:=\Frac((\kcirc)^\sh)$ is the maximal
unramified extension of $k$ and a finite extension $l/k$ is unramified if and only if the embedding
$\kcirc\to\lcirc$ is local-\'etale (i.e. $\lcirc$ is a localization of a $\kcirc$-\'etale algebra). It follows
that if $k$ is henselian (e.g. analytic) then $\Gal(k^\nr/k)\toisom\Gal(\tilk^s/\tilk)$ and any subfield
$\tilk\subseteq\till\subseteq\tilk^s$ is the residue field of a valued field $l$ that is unramified over $k$ and
unique up to unique isomorphism . An extension $l/k$ is called {\em totally ramified} if for any tower
$k\subseteq k'\subsetneq l'\subseteq l$, the extension $l'/k'$ is not unramified. An extension $l/k$ is called
{\em moderately ramified} (or {\em tamely ramified}) if it is a composition of an unramified extension $l^n/k$
followed by an extension $l/l^n$ such that any its finite subextension is of degree prime to $p$. Any other
$l/k$ is called {\em wildly ramified}, including the case when $l/k$ is inseparable, as opposed to the usual
convention.

Let $k$ be henselian. By \cite[\S6.2]{GR}, $k$ possesses unique maximal moderately ramified extension $k^\mr$,
$k^\mr/k^\nr$ is an abelian extension and $\Gal(k^s/k^\mr)$ is a pro-$p$-group, where $p=\expcha(\tilk)$ is the
{\em exponential characteristic} (i.e. $p=1$ or $p$ is the usual prime characteristic). In particular, any
finite extension of $k^\mr$ is a $p$-extension, and $k^\mr=k^a$ when $p=1$. We say that a finite extension $l/k$
of henselian valued fields {\em defectless} if its {\em defect} $d_{l/k}:=[l:k]/(e_{l/k}f_{l/k})$ equals to one
(in other words, the defect is trivial). In this case one also says that $l/k$ is {\em Cartesian} because this
happens if and only if the $k$-vector space $l$ has an {\em orthogonal} basis, i.e. a basis $v_1\. v_n$ such
that the non-archimedean inequality $|\sum_{i=1}^n a_iv_i|\le\max_{1\le i\le n}|a_iv_i|$ is an equality for any
choice of $a_i\in k$.

\begin{lem}
Let $l/k$ be a finite extension of henselian valued fields, $e=e_{l/k}$, $f=f_{l/k}$, $d=d_{l/k}$,  $n=[l:k]$
and $p=\expcha(\tilk)$.

(i) $l/k$ splits uniquely into a tower $l/l^m/l^n/k$, where $l^n/k$ is unramified, $l^m/l^n$ is totally
moderately ramified, and $l/l^m$ is totally wildly ramified;

(ii) $l/k$ is unramified if and only if $f=n$ and $\till/\tilk$ is separable;

(iii) $l/k$ is moderately ramified if and only if $\till/\tilk$ is separable, $(e,p)=1$ and $d=1$ (i.e. $l/k$ is
Cartesian);

(iv) $d$ is a power of $p$.
\end{lem}
\begin{proof}
Find an unramified extension $l^n/k$ with an isomorphism $\phi:\till^n\toisom F$, where $F$ is the separable
closure of $\tilk$ in $\till$. Then $\phi$ lifts to a (necessarily local) homomorphism $(l^n)^\circ\into\lcirc$
by \cite[$\rm IV_4$, 18.8.4]{ega}, so we identify $(l^n)^\circ$ and $l^n$ with their images in $l$. The
extension $l/l^n$ is totally ramified because all intermediate fields are henselian and have purely inseparable
residue field extension, so there is no non-trivial unramified subextension. By the maximality condition
$(l^n)^\mr$ contains all its conjugates over $l^n$, hence $l^m=l\cap (l^n)^\mr$ is a well defined moderately
ramified extension of $k$. Finally, $l/l^m$ is a $p$-extension because $\Gal(k^s/(l^n)^\mr)$ is a pro-$p$-group
and so $k^a/(l^n)^\mr$ is a union of finite $p$-extensions. This proves existence in (i), and the method of
proof gives uniqueness, so the remaining claims follow.
\end{proof}

\begin{lem}\label{etalelem}
A finite extension $l/k$ of valued fields of finite height is unramified if and only if $\lcirc/\kcirc$ is
\'etale.
\end{lem}
\begin{proof}
By the very definition, $l/k$ is unramified if and only if $\lcirc/\kcirc$ is local-\'etale (or essentially
\'etale), hence we have only to show that the local-\'etale extension $\lcirc/\kcirc$ is \'etale. But this is a
consequence of Lemma \ref{finhtlem}.
\end{proof}

\begin{rem}\label{etalerem}
By definition, $l/k$ is unramified if and only if $\lcirc$ is local-\'etale over $\kcirc$. However, if the
heights are infinite the latter does not necessarily imply that $\lcirc/\kcirc$ is \'etale.
\end{rem}

An extension $K/k$ of valued fields will be called {\em bounded} if any non-zero element of $\Kcirc$ divides a
non-zero element of $\kcirc$. This condition is equivalent to requiring that $k\Kcirc$, which is the
localization of $\Kcirc$ at $k^\times\cap\Kcirc$ and hence a valuation ring of $K$, coincides with $K$. The main
result of \S\ref{fields} is the following statement, which will be called uniformization of one-dimensional
valued fields.

\begin{theor}
\label{valunif} Let $K/k$ be a finitely generated extension of valued fields of transcendence degree one. Assume
that $k$ is separably closed and the valuation ring $k\Kcirc$ is centered on a smooth point of the normal
projective $k$-model of $K$. Then there exists a transcendence basis $\{x\}$ of $K/k$ such that the finite
extension $K/k(x)$ is unramified.
\end{theor}

\begin{rem}
I expect that the Theorem is true for any $n=\trdeg(K/k)$. We will see that already in our case the proof is
difficult, and the case of $n>1$ is absolutely open. Establishing the case of $n>1$ would be a major
breakthrough that (almost surely) would enable one to prove a local version of the higher dimensional
semi-stable reduction theorem. Even when $k$ is trivially valued, this is open for $n>3$. This particular case
follows from (conjectural) local desingularization of varieties along valuations (so called, Zariski local
uniformization), and I expect it is not essentially easier than the Zariski local uniformization itself.
\end{rem}

Theorem \ref{valunif} will be proved in the next section. It admits the following analytic version, which we
call uniformization of one-dimensional analytic fields.

\begin{theor}\label{fieldunif1}
Let $k$ be an analytic algebraically closed field and $K$ be an analytic $k$-field which is finite over a
subfield $\ol{k(y)}$ topologically generated by an element. Then

(i) $K$ is finite and unramified over a subfield $\ol{k(x)}$ for some choice of $x\in K$,

(ii) moreover, there exists a positive $\veps$ such that for any $x'\in K$ with $|x-x'|<\veps$ the extension
$K/\ol{k(x')}$ is finite and unramified,

(iii) if $\tilK\neq\tilk$ and $x\in\Kcirc$ is any element such that $\tilx$ is transcendental over $\tilk$ and
$\tilK/\tilk(\tilx)$ is separable then $K/\ol{k(x)}$ is finite and unramified.
\end{theor}

Our proof of Theorem \ref{fieldunif1} involves methods from non-archimedean analytic geometry, and for
expository reasons we postpone it until \S\ref{ansec}. For our current purposes the reader can view this Theorem
as the only analytic black box we use.

\subsection{Reduction to uniformization of analytic fields} \label{gencasesec}
The goal of \S\ref{gencasesec} is to prove Theorem \ref{valunif} modulo the black box given by Theorem
\ref{fieldunif1}. The proof contains two main steps: decompletion, i.e. proving the theorem for non-complete
fields of height one, and composition of valuations which allows induction on height. We start with proving two
criteria for an extension of valued fields to be unramified. These are a decompletion criterion and a
composition criterion. Note that in the decompletion criterion given below, the separability assumption is
necessary only when $\hatK$ is not separable over $K$.

\begin{prop}
\label{unramcrit1} A finite extension $L/K$ of height one valued fields is unramified if and only if it is
separable and the extension $\hatL/\hatK$ is unramified.
\end{prop}
\begin{proof}
Assume that the extension $L/K$ is unramified. In particular, it is separable and the homomorphism
$\Kcirc\to\Lcirc$ is \'etale by Lemma \ref{etalelem}. It follows that $A:=\Lcirc\otimes_{\Kcirc}\hatKcirc$ is
\'etale over $\hatKcirc$ and hence $A$ is normal. Note also that $L\otimes_K\hatK$ is a separable
$\hatK$-algebra with a direct factor isomorphic to $\hatL$. By normality of $A$ it has a direct factor $A'$ with
fraction field isomorphic to $\hatL$. Since $A'$ is normal and $\hatLcirc$ is the integral closure of
$\hatKcirc$ in $\hatL$ because $\hatK$ is henselian, we obtain that $A'\toisom\hatLcirc$. So, $\hatLcirc$ is
\'etale over $\hatKcirc$.

The converse implication is more involved. Let $F$ denote the field $L$ without the valued field structure (so,
$F/K$ is separable). Let $\Fcirc$ be the integral closure of $\Kcirc$ in $F$; it is a semi-local ring with
maximal ideals $m_1\. m_n$. The localizations $\Fcirc_i=F_{m_i}$ are the valuation rings of $F$ lying over
$\Kcirc$, and without loss of generality $\Lcirc=\Fcirc_1$. Let $F_i$ denote the valued field corresponding to
$\Fcirc_i$.

Let $y$ be a primitive element of the separable extension $F/K$ and $x_0$ be an element in the ideal $m_2\dots
m_n\subset\Fcirc$ such that its image in $\Fcirc/m_1\toisom\tilL$ is non-zero and generates $\tilL$ over
$\tilK$. Since all but finitely many linear combinations $x_0+zy$ with $z\in K$ are primitive for $F/K$, we can
find $z\in\Kcirccirc$ such that $x=x_0+zy$ is primitive and $zy\in m_1\dots m_n$, and so $x\in m_2\dots m_n$ and
the image of $x$ is a non-zero primitive element of $\tilL/\tilK$. Let $f(T)\in\Kcirc[T]$ be the minimal
polynomial of $x$ and $f=f_1\dots f_n$ be its decomposition in $\hatK[T]$. Since $\hatK\otimes_K
F\toisom\prod_{i=1}^n\hatF_i$, we can renumber $f_i$'s so that $\hatF_i\toisom\hatK[T]/(f_i(T))$. Now, we
consider the situation in $\hatLcirc$. Let $\hatx$ be the image of $x$ in $\hatLcirc$. Since $f_1(\hatx)=0$, we
have that $f'(\hatx)=f'_1(\hatx)f_2(\hatx)\dots f_n(\hatx)$. The element $f'_1(\hatx)$ is invertible because
$\hatLcirc$ is finite \'etale over $\hatKcirc$ with degree $[\hatL:\hatK]$ (so $\tilf_1\in\tilK[T]$ is the
minimal polynomial of $\tilx$). For any $2\le i\le n$, $f_i(T)$ is an irreducible polynomial of degree
$d_i=[\hatF_i:\hatK]$ whose roots are in $\hatFcirccirc_i$ (since $\hatKcirc[T]/(f_i)$ in $\hatF_i$ has $T$
corresponding to $x$, so $f_i$ has a root in $m\hatFcirc_i=\hatFcirccirc_i$). It follows that $f_i(T)\in
T^{d_i}+\hatKcirccirc[T]$. The residue of $\hatx$ in $\tilL$ is not zero, hence $\hatx$ is invertible and
therefore the elements $f_i(\hatx)$ are invertible in $\Lcirc$ for $i>1$.

We have proved that $f'(\hatx)$ is invertible in $\hatLcirc$, and therefore $f'(x)$ is invertible in $\Lcirc$.
We conclude that the ring $A=\Kcirc[x]_{f'(x)}$ is contained in $\Lcirc$. On the other hand, $A$ is \'etale over
$\Kcirc$ and hence is integrally closed. Thus, $A$ contains $\Fcirc$, so the localization of $A$ at its
(necessarily maximal) prime ideals over $\Kcirccirc$ are valuation rings, so one of these must be $\Lcirc$. It
follows that $\Lcirc$ is local-\'etale over $\Kcirc$ (and it is even \'etale by Lemma \ref{finhtlem}).
\end{proof}

Next we prove a composition criterion. We formulate just the statement that will be used. Note, however, that
much stronger results for "composed" rings can be found in \cite[\S2.3]{Tem2} and further generalizations were
obtained by D. Rydh.

\begin{prop} \label{unramcrit2}
Assume that $l/k$ is a finite extension of valued fields of finite height. Suppose that $\till$ and $\tilk$ are
provided with compatible structures of valued fields of finite height, and let $L$ and $K$ be the fields $l$ and
$k$ provided with the composed valuations. If the extensions $l/k$ and $\till/\tilk$ are unramified then $L/K$
is unramified.
\end{prop}
\begin{proof}
To simplify the notation we rename the valuation rings and their maximal ideals as $A=\kcirc$, $B=\tilkcirc$ and
$C=\Kcirc$ (so $C$ is the composed valuation ring), $m=\kcirccirc$ and $n=\Kcirccirc$. In the same manner, we
set $A'=\lcirc$, $B'=\tillcirc$, $C'=\Lcirc$, $m'=\lcirccirc$, $n'=\Lcirccirc$. Finally, let $m_B$ and $m'_B$ be
the maximal ideals of $B$ and $B'$. By Lemma \ref{etalelem} the embeddings $f_A:A\to A'$ and $f_B:B\to B'$ are
\'etale and our aim is to prove that $f_C:C\to C'$ is \'etale. For this it is enough to show that $C'$ is
$C$-flat, $n'=nC'$ and $f_C$ is finitely presented. The first claim is obvious because any $C$-module without
torsion is flat.

By Lemma \ref{finhtlem} there exists an element $s\in C$ such that $A=C_s$. Clearly, we also have that
$A'=C'_s$. By the obvious bijection between $C$-submodules $m\subseteq D\subseteq A$ and $B$-submodules
$\tilD\subseteq K$ (given by $D\mapsto\tilD=D/m$) we see that $m=s^{-1}m\subset n$ and similarly
$m'=s^{-1}m'\subset n'$. Using that $m'=mA'$ by \'etaleness of $f_A$ we obtain that $m'=mC'_s=mC'\subset nC'$.
Now, to prove that the inclusion $nC'\subseteq n'$ is an equality it is enough to show that it becomes an
equality after quotient by $m'$. So, it remains to note that $n'/m'=m'_B$ and $nC'/m'$ contains $(n/m)B'=m_BB'$,
which is $m'_B$ by \'etaleness of $f_B$.

The last (and the most subtle) check is that $f_C$ is finitely presented. Let $a$ be a finite subset of $A'$
that generates it over $A$. Multiplying $a$ by an appropriate $s^n$, we can achieve that $a\subset C'$. Note
that $m'=mA'=mA[a]=m[a]\subset C[a]$, where $m[a]$ denotes the set of polynomials in $a$ with coefficients in
$m$. Pick up a finite subset $b\subset C'$ whose image generates $B'$ over $B$, then one easily sees $C'$
coincides with its subalgebra $C''=C[a,b]$ because $C''_s=A'=C'_s$ and $C''/m'=B'=C'/m'$. This proves that $f_C$
is of finite type. Since $f_C$ is flat, \cite[3.4.7]{RG} implies that $f_C$ is also finitely presented.
\end{proof}

\begin{proof}[Proof of Theorem \ref{valunif}]
The extension $K/k$ is induced from an extension of absolutely finitely generated valued fields, that is there
exist subfields $L\subset K$ and $l\subset k$ such that $l$ and $L$ are finitely generated over their prime
subfields with $L/l$ separable of transcendence degree $1$ such that $L\otimes_{l}k$ is integral and
$K=\Frac(L\otimes_{l}k)$. The valued field $L$ and $l$ are of finite height by Corollary \ref{approxcor}(i). We
claim that it suffices to prove the Theorem for the extension $L/l$. Indeed, assume that there exists $x\in L$
transcendental over $k$ and such that $L/l(x)$ is unramified. Then $\Lcirc$ is \'etale over $l(x)^\circ$ by
Lemma \ref{etalelem}, and since $\Kcirc$ is a localization of its subring $\Lcirc\otimes_{l(x)^\circ}k(x)^\circ$
(because both are localizations of the integral closure of $k(x)^\circ$ in $K$), we obtain that $K$ is
unramified over $k(x)$. Thus, we assume in the sequel that $k$ is of finite height $h$.

Let $C$ be the normal projective $k$-model of $K$, and consider the valuation ring $\calO=k\Kcirc$. If
$\calO\neq K$ (i.e. the extension $K/k$ is not bounded) then $\calO$ dominates the local ring $\calO_z$ of a
closed point $z$, which must be $k$-smooth by the assumption of the Theorem. In particular, $\calO$ is the DVR
$\calO_z$, and hence for any uniformizer $x$ at $z$, $\calO$ is \'etale over $\calO\cap k(x)$ (since $x$ induces
an \'etale morphism from a neighborhood of $z$ to $\bfA^1_k$). By Lemma \ref{composlem}(i), the valuation of $K$
is composed from the valuation induced by $\calO$ and from the valuation on its residue field $k(z)$ which
extends the valuation of $k$ (the latter is uniquely defined because $k(z)/k$ is finite and purely inseparable).
Notice that the residue field of the valuation ring $\calO\cap k(x)$ is also $k(z)$. Hence applying Proposition
\ref{unramcrit2} we obtain that $K$ is unramified over $k(x)$.

In the sequel we assume that $\calO=K$, and so it is centered on the generic point of $C$. Note that $K/k$ is
separable because the generic point of $C$ is $k$-smooth by our assumptions. Our proof will run by induction on
$h$. If $h=0$, then $K$ is necessarily of height $0$ too, hence $K/k(x)$ is unramified if and only if $x$ is a
separable transcendence basis of $K$ over $k$. It is well known that such a basis exists.

In the sequel we assume that $h>0$ (i.e the valuations are non-trivial). We will need the following well known
fact which hold for any extension $K/k$ which is separable, finitely generated and of transcendence degree one:
$K/k$ possesses a separable transcendence basis, and $\{x\}\subset K$ is a such a basis if and only if $x\notin
kK^p$, where $p=\cha(k)$ and we agree on notation $K^p=1$ when $p=0$. Let $k_1$ and $K_1$ denote the fields $k$
and $K$ provided with the induced valuations of height one. We provide the residue fields $\tilk_1$ and
$\tilK_1$ with the valuations induced from $k$ and $K$. Since $k_1$ is separably closed, its completion
$\hatk_1$ is algebraically closed, and we obtain in particular that $\tilk_1$ is algebraically closed. We will
also use the simple fact that $k_1K_1^p$ is a $k_1$-subspace of $K_1$ with empty interior. Indeed, it suffices
to check that $0$ is not in the interior, but for any $y\in K_1\setminus k_1K_1^p$ the set $yk_1^\times$ is
disjoint from $k_1K_1^p$ but contains $0$ in its closure.

Assume first that $\tilK_1=\tilk_1$. We claim that there exists $x\in K_1$ such that firstly $\hatK_1$ is
unramified over $\ol{k_1(x)}$, and secondly $x\notin k_1K_1^p$. If $\hatK_1=\hatk_1$ (i.e. $K_1$ is a valued
subfield in $\hatk_1$) then this is obvious since the first condition is empty. Otherwise, we apply Theorem
\ref{fieldunif1}(i) to find $x\in\hatK_1$ such that $\hatK_1/\ol{k_1(X)}$ is finite and unramified. Moreover,
part (ii) of the same Theorem implies that we can move $x$ slightly, and since $k_1K_1^p$ is nowhere dense we
can achieve that $x\in K_1$ and $x\notin k_1K_1^p$. This proves the above claim. Note that $K_1$ is a finite
separable extension of $k_1(x)$, and hence $K_1$ is unramified over $k_1(x)$ by Proposition \ref{unramcrit1}.
Now, to prove that $K/k(x)$ is unramified it remains to use Proposition \ref{unramcrit2} and the assumption that
$\tilK_1=\tilk_1$.

Finally we consider the case when $\tilK_1\neq\tilk_1$. Since $\tilk_1$ is algebraically closed, Abhyankar
inequality (Lemma \ref{Abhlem}) implies that $\trdeg(\tilK_1/\tilk_1)=1$. Thus, the extension $\tilK_1/\tilk_1$
satisfies the conditions of the Theorem, and we can use the induction assumption due to the fact that the height
of $\tilk_1$ equals to the height of $k$ decreased by one. Find $\tilx\in\tilK_1$ transcendental over $\tilk_1$
and such that the extension $\tilK_1/\tilk_1(\tilx)$ is unramified (in particular, it is separable), and lift it
to an element $x\in\Kcirc_1$ not contained in $k_1K_1^p$. Then $\hatK_1/\ol{k_1(x)}$ is finite and unramified by
Theorem \ref{fieldunif1}(iii). Now it remains to use Propositions \ref{unramcrit1} and \ref{unramcrit2}, as
earlier.
\end{proof}

\section{Riemann-Zariski spaces}\label{RZsec}

We describe the classical Riemann-Zariski spaces in \S\ref{absRZsec} and indicate in \S\ref{relRZsec} how they
can be generalized to a relative case. In \S\ref{Pmodsec} we describe a method of proving modification theorems,
which we illustrate in \S\ref{redfibsec} by proving a particular case of the reduced fiber theorem that will be
used later in the paper.

If $X$ is a scheme then by $X^0$ we denote the set of generic points of $X$. We refer the reader to \cite{egaI},
6.10.1 and 6.10.2, for the definitions of schematical image and density. Let us fix the notions of modification
and alteration (which can differ from paper to paper). By a {\em modification} (resp. {\em quasi-modification})
we mean a proper (resp. separated finite type) morphism which induces an isomorphism of open schematically dense
subschemes. By an {\em alteration} (resp. {\em quasi-alteration}) of an integral scheme $X$ we mean a proper
(resp. separated finite type) morphism $f:Y\to X$ with integral $Y$ and finite field extension $k(Y)/k(X)$. If
the latter extension is separable then we say that the (quasi-) alteration is {\em generically \'etale}. To
justify this terminology, we note such morphism $f:Y\to X$ is of finite presentation over an open subscheme of
$X$, and hence an easy approximation argument (see \S\ref{approxsec}) implies that $f$ is \'etale over a dense
open in $X$ when $f$ is generically \'etale in the preceding sense.

The combination of words "quasi-compact quasi-separated" will appear very often in the paper, so we will
abbreviate it with a single "word" {\em qcqs}. Note that Grothendieck called qcqs schemes {\em coherent}, but
this might be too confusing. Until the end of this paper $S$ is a scheme. We assume that $S$ is integral, qcqs
and with the generic point $\eta=\Spec(K)$ if not said explicitly. For an $S$-scheme $X$, by $X_\eta$ and $X_s$
we denote its fibers over $\eta$ and $s$, respectively, where $s\in S$ is any point. A modification $f:X'\to X$
is called an {\em $\eta$-modification} if its $\eta$-fiber $f_\eta:X'_\eta\to X_\eta$ is an isomorphism. A
reduced $S$-scheme $X$ is called {\em $\eta$-normal} if it does not admit non-trivial finite
$\eta$-modifications. Note that $\eta$-normality is a sort of partial normality condition "semi-orthogonal" to
normality of $X_\eta$ because $X$ is normal if and only if $X_\eta$ is normal and $X$ is $\eta$-normal.

\subsection{Noetherian approximation}
\label{approxsec} In this section we recall some results from \cite[$\rm IV_3$, \S8]{ega} on projective limits
of schemes. These results will be used very often in the sequel. Let $\{S_i\}$ be a filtered projective family
of schemes with affine transition morphisms and initial object $S_0$ (the latter assumption does not really
restrict the generality), then $S=\projlim S_i$ exists by \cite[$\rm IV_3$, 8.2.3]{ega}. We assume that $S_0$ is
qcqs, then by \cite[$\rm IV_3$, 8.8.2 and 8.5.2]{ega}, any finitely presented morphism $f:X\to S$ (resp.
finitely presented quasi-coherent $\calO_S$-module) is induced from a finitely presented morphism $f:X_i\to S_i$
(resp. finitely presented quasi-coherent $\calO_{S_i}$-module), and for any pair of finitely presented morphisms
$X_i\to S_i$, $Y_i\to S_i$ we have that
$$\injlim_{j\ge i}\Hom_{S_j}(Y_i\times_{S_i}S_j,X_i\times_{S_i}
S_j)\toisom\Hom_S(Y_i\times_{S_i}S,X_i\times_{S_i}S).$$ Moreover, by \cite[$\rm IV_3$, 11.2.6]{ega} $f$ is flat
if and only if there exists $i_0$ such that each $f_i$ with $i\ge i_0$ is flat. We will refer to all these
statements by the word "approximation". By "noetherian approximation" we mean these statements combined with
\cite[C.9]{TT}, which asserts that any qcqs scheme $S$ can be represented as the projective limit of a filtered
family of schemes of finite type over $\bfZ$ such that the transition morphisms are affine. A typical example of
an application of noetherian approximation is that any relative $S$-curve $C$ is induced from a relative curve
$C_0$ over a scheme $S_0$ of finite type over $\bfZ$ in the sense that $C\toisom C_0\times_{S_0}S$.

\begin{rem}\label{prem}
(i) Many properties of $f:X\to S$ descend to $f_i$'s with $i\ge i_0$. Let $\bfP$ be a property of schemes of
finite type over a field such that (a) for any field extension $l/k$, a finite type morphism $f:X\to\Spec(k)$
satisfies $\bfP$ if and only if $f\otimes_k l$ satisfies $\bfP$, (b) for any finitely presented morphism $X\to
S$ with qcqs $S$ the set of points $s\in S$ with the fiber $X_s$ satisfying $\bfP$ is constructible. For
example, $\bfP$ can be geometric reducedness since (a) is obvious and (b) is proved in \cite[$\rm IV_3$,
9.7.7(iii)]{ega}. Then $f$ is a $\bfP$-morphism (i.e. it is flat and its fibers satisfy $\bfP$) if and only if
each $f_i$ is a $\bfP$-morphism for $i\ge i_0$. Indeed, this is true for flatness by \cite[$\rm IV_3$,
11.2.6]{ega}, hence without restriction of generality we can assume that $f_0$ is flat. Now, if $E_i\subset S_i$
is the constructible set of points $s$ with non-$\bfP$ fiber $(X_i)_s$ then each $E_i$ is the preimage of $E_0$.
Since the preimage of $E_0$ in $S$ is empty, already some $E_i$ is empty by \cite[$\rm IV_3$, 8.3.3]{ega}.

(ii) A similar claim holds for suitable properties of diagrams of $S$-flat schemes, e.g. for modifications
$(C',D')\to(C,D)$ of multipointed $S$-curves.
\end{rem}

We will not care about it, but in all cases when noetherian approximation is used in this paper, it can be
easily seen that the case of an affine $S$ suffices. In this particular case noetherian approximation is easier
and appeared already in \cite[$\rm IV_3$, 8.9.1]{ega}.

\subsection{Absolute Riemann-Zariski spaces}
\label{absRZsec} We adopt the exposition of \cite[\S1]{Tem1} to the case of general schemes, but we will use
different notation. Fix an integral scheme $S$ (temporarily not necessarily qcqs) and a dominant morphism
$\oeta:\Spec(\oK)\to S$, where $\oK$ is a field. Let $\eta$ be the generic point of $S$ and $K=k(\eta)$. By a
{\em $\oK$-modification} (resp. {\em $\oK$-quasi-modification}) we mean a splitting of $\oeta$ into a
composition of a schematically dominant morphism $\Spec(\oK)\to S_i$ (so $S_i$ is integral) and a proper (resp.
separated finite type) morphism $S_i\to S$. By use of schematical images of $\Spec(\oK)$ in fiber products over
$S$ one shows that the projective family of all $\oK$-modifications (resp. $\oK$-quasi-modifications) is
filtered. If $K\toisom\oK$, $K^a\toisom\oK$ or $K^s\toisom\oK$ then $S_i$'s are modifications, alterations or
generically \'etale alterations of $S$, respectively (resp. quasi-versions of these notions). The topological
space $\gtS=\RZ_\oK(S)=\projlim S_i$, where $S_i$'s are the $\oK$-modifications, is called the {\em
Riemann-Zariski space} of $S$ with respect to $\oK$.

\begin{prop}
\label{quacomprop} The space $\gtS$ is qcqs if and only if the scheme $S$ is qcqs.
\end{prop}
\begin{proof}
Assume that $S$ is qcqs. Let $f_i:\gtS\to S_i$ be the projections and $f_{ji}:S_j\to S_i$ with $j\ge i$ be the
transition maps. An open subscheme $U\into S_i$ is qcqs if and only if it is compact in the constructible
topology of $U$. If this is the case then each $f_{ji}^{-1}(U)$ is compact in the constructible topology and
hence $\gtU=f_i^{-1}(U)$ is compact in the constructible projective limit topology. Since the usual (Zariski)
topology of $\gtS$ is weaker, $\gtU$ is quasi-compact. In particular, $\gtS$ is quasi-compact. Preimages of the
sets $U$ as above form a basis of the topology of $\gtS$, hence for quasi-separatedness of $\gtS$ it is enough
to show that the intersection of two such preimages $\gtU$ and $\gtU'$ is quasi-compact. Enlarging $i$ enough we
can assume that $\gtU=f_i^{-1}(U)$ and $\gtU'=f_i^{-1}(U')$, but then $\gtU\cap\gtU'=f_i^{-1}(U\cap U')$ is
quasi-compact by what we proved above.

Now, let us assume that $\gtS$ is qcqs. Let us assume for a moment that the projection $f:\gtS\to S$ is
surjective (that is not automatic even though $f_{ji}$'s are surjective). Then $S$ is obviously quasi-compact,
and quasi-separatedness of $S$ follows from the fact that the preimage in $\gtS$ of an open quasi-compact set
$U\subset S$ is open and quasi-compact (by the same argument as we used in the direct implication). It remains
for a point $x\in S$ to show that $f^{-1}(x)$ is non-empty. Find a valuation ring $\calO$ of $\oK$ which
dominates $\calO_{S,x}\subset\oK$ (it exists by Zorn's lemma). Then the morphism $\Spec(\calO)\to S$ factors
through each $S_i$ by the valuative criterion of properness, hence the images of the closed point of
$\Spec(\calO)$ in each $S_i$ give rise to a compatible family of points $x_i\in S_i$. The point $\gtx\in\gtS$
corresponding to the family $\{x_i\}$ sits over $x\in S$, and we are done.
\end{proof}

In the sequel, we assume again that $S$ is qcqs. We provide $\gtS$ with a sheaf $\calO_\gtS=\injlim
\pi_i^{-1}(\calO_{S_i})$, where $\pi_i:\gtS\to S_i$ are the projections. The following easy approximation lemma
will be very useful in the sequel.

\begin{lem}\label{approxlem}
For any point $\gtx\in\gtS$, the scheme $\Spec(\calO_{\gtS,\gtx})$ is isomorphic to a projective limit of
$\oK$-quasi-modifications of $S$.
\end{lem}
\begin{proof}
Since $\calO_{\gtS,\gtx}=\injlim\calO_{S_i,x_i}$ where $x_i=\pi_i(\mathfrak x)$, $\Spec(\calO_{\gtS,\gtx})$ is
isomorphic to the projective limit of the schemes $\Spec(\calO_{S_i,x_i})$. It remains to notice that each
$\Spec(\calO_{S_i,x_i})$ is isomorphic to the projective limit of its open neighborhoods in $S_i$, and the
latter are obviously $\oK$-quasi-modifications of $S$.
\end{proof}

This Lemma combined with an approximation argument often allows to reduce certain birational problems on $S$ to
problems over the local rings $\calO_{\gtS,\gtx}$. These rings are valuation rings, as we are going to prove.

\begin{lem}
\label{ratfunlem} For any element $f\in\oK$, there exists a $\oK$-modification $S'\to S$ such that $f$ is a
rational function on $S'$ giving rise to a morphism $S'\to\bfP^1_\bfZ$.
\end{lem}
\begin{proof}
Consider the morphism $\Spec(\oK)\stackrel{(\oeta,f)}{\to} S\times_{\Spec(\bfZ)}\bfP^1_\bfZ$, and take $S'$ to
be its schematical image.
\end{proof}

\begin{cor}\label{314cor}
For any point $\gtx\in\gtS$, the ring $\calO_{\gtS,\gtx}$ is a valuation ring with the fraction field $\oK$.
\end{cor}
\begin{proof}
Since $\calO_{\gtS,\gtx}=\injlim\calO_{S_i,x_i}$, where $\{S_i\}$ is the set of all $\oK$-modifications,
including the one in Lemma \ref{ratfunlem}, it follows that for any element $f\in\oK^\times$, either $f$ or
$f^{-1}$ is contained in $\calO_{\gtS,\gtx}$. Hence $\calO_{\gtS,\gtx}$ is a valuation ring and
$\Frac(\calO_{\gtS,\gtx})=\oK$.
\end{proof}

Let $Val_\oK(S)$ be the set of morphisms $\phi_\gtx:\Spec(\calO_\gtx)\to S$ such that $\calO_\gtx$ is a
valuation ring of $\oK$ and $\phi_\gtx$ has $\oeta$ as the generic fiber. By the above Corollary, we have a
natural map $\gtS\to Val_\oK(S)$. Conversely, any morphism $\phi_\gtx$ as above factors uniquely through any
$\oK$-modification of $S$ by the valuative criterion of properness. The images in all $S_i$'s of the closed
point of $\Spec(\calO_\gtx)$ give rise to a point $\gtx$ of the projective limit $\gtS$. We have constructed an
opposite map $Val_\oK(S)\to\gtS$ which is inverse because $\calO_{\gtS,\gtx}\toisom\calO_\gtx$. Indeed,
$\calO_\gtx$ dominates the local rings $\calO_{S_i,x_i}$, hence it dominates their union $\calO_{\gtS,\gtx}$,
but both are valuation rings with common fraction field, so they must coincide. We have thereby proved the
following statement.

\begin{cor}
\label{bijcor} The sets $\gtS$ and $Val_\oK(S)$ are naturally bijective.
\end{cor}

Any quasi-compact open subset $\gtS'\subset\gtS$ is induced from a quasi-compact open subscheme $S'$ of some
modification of $S$, and by Corollary \ref{bijcor} the natural map $\gtS'\to\RZ_\oK(S')$ is bijective. More
generally, by the valuative criterion of separatedness we obtain a natural injective map $\RZ_\oK(S')\into\gtS$
for any $\oK$-quasi-modification $S'\to S$.

\begin{prop}
\label{oetamodprop} Assume that $S_1\. S_n$ are $\oK$-quasi-modifications of $S$. Then there exists a
$\oK$-modification $S'\to S$ such that $S'$ contains open subschemes $S'_i$ which are $S$-isomorphic to
$\oK$-modifications of $S_i$'s.
\end{prop}
\begin{proof}
Since $S$ is quasi-compact, $S$ and each $S_i$ possess a finite affine covering. Note that we can replace each
$S_i$ with its open subschemes $U_1\. U_m$ which cover $S_i$. For this reason it suffices to treat the case when
each $S_i$ is affine and its image is contained in an affine open subscheme of $S$. Next, we claim that it
suffices to find a $\oK$-modification $X_i\to S$ which satisfies the assertion of the Proposition for a single
$S_i$, because then we can take $S'$ to be any $\oK$-modification of $S$ which dominates all $X_i$'s. So, we can
assume that $n=1$, $S_1=\Spec(B)$ and the image of $S_1$ is contained in $\Spec(A)=V\into S$. Let $f_1\.
f_l\in\oK$ be generators of $B$ over $A$. By Lemma \ref{ratfunlem} we can find a $\oK$-modification $\phi:S'\to
S$ such that each $f_j$ induces a morphism $F_j:S'\to\bfP^1_\bfZ$. Then $S'$ is as required because it is easy
to check that $(\cap_{j=1}^l F_j^{-1}(\bfA^1_\bfZ))\cap\phi^{-1}(V)\into S'$ is a $\oK$-modification of $S_1$.
\end{proof}

\begin{cor}
\label{openembcor} For any $\oK$-quasi-modification $S'\to S$, the injection $\RZ_\oK(S')\to\RZ_\oK(S)$ is a
homeomorphism onto an open subspace.
\end{cor}
\begin{proof}
It would suffice to know that there exists a cofinal family of $\oK$-modifications $S'_i\to S'$ such that each
$S'_i$ admits an open immersion $S'_i\into S_i$ compatible with $\oK$ into a $\oK$-modification of $S$. But the
latter is an obvious consequence of Proposition \ref{oetamodprop}.
\end{proof}

\begin{rem}
(i) Corollaries \ref{bijcor} and \ref{openembcor} imply that Riemann-Zariski spaces of affine schemes (we call
them {\em affine}) admit a usual valuation-theoretic description. Namely, if $X=\Spec(A)$ then $\gtS$ is the set
of all valuation rings of $\oK$ that contain $A$. The sets $\RZ(A[f_1\. f_n])$ with $f_i\in\oK$ form a basis of
open subsets of the topology of $\gtS$.

(ii) In particular, $\RZ_\oK=\RZ_\oK(\bfZ)$ (resp. $\RZ_\oK(k)$ for a subfield $k\subset\oK$) is the classical
Riemann-Zariski space of $\oK$ whose points are valuations on $\oK$ (resp. trivial on $k$).

(iii) By Proposition \ref{quacomprop}, a general Riemann-Zariski space is pasted from finitely many affine ones
via a finite gluing data.
\end{rem}

 We will not need the following result, but analogously to \cite[1.3]{Tem1} one can
strengthen our Corollary \ref{bijcor} as follows.

\begin{lem}
The topology on $\gtS$ is the weakest topology for which the natural maps $\phi:\gtS\to\RZ_\oK$ and $\gtS\to S$
are continuous. If $S$ is separated then $\phi$ is a homeomorphism onto its image (so, the topology is generated
only by  $\phi$).
\end{lem}

\subsection{Relative Riemann-Zariski spaces}\label{relRZsec}
This section will not be used in this paper, but the relative spaces will play important role in \cite{Tem2}.
Throughout \S\ref{relRZsec} we fix a qcqs morphism of schemes $f:Y\to X$. In particular, the sheaf
$f_*(\calO_Y)$ is quasi-coherent by \cite[6.7.1]{egaI}. Consider the family of all factorizations of $f$ into a
composition of a schematically dominant qcqs morphism $f_i:Y\to X_i$ followed by a proper morphism $\pi_i:X_i\to
X$. We call the pair $(f_i,\pi_i)$ a {\em $Y$-modification} of $X$; usually it will be denoted simply $X_i$.
Given two $Y$-modifications of $X$, we say that $X_j$ dominates $X_i$ if there exists an $X$-morphism
$\pi_{ji}:X_j\to X_i$ compatible with $f_i,f_j,\pi_i$ and $\pi_j$. Notice that if $\pi_{ji}$ exists then it is
unique. The family of $Y$-modifications of $X$ is filtered because two $Y$-modifications $X_i,X_j$ are dominated
by the schematical image of $Y$ in $X_i\times_X X_j$. Also, this family has an initial object corresponding to
the schematical image of $Y$ in $X$. The same facts are valid for the more restrictive class of finite
$Y$-modifications. The projective limit of finite $Y$-modifications of $X$ exists in the category of schemes. We
will denote it $\Nor_Y(X)$ and call the {\em $Y$-normalization of $X$}. We define the {\em Riemann-Zariski space
of $X$ with respect to $Y$} to be the projective limit of the underlying topological spaces of all
$Y$-modifications of $X$; this space will be denoted $\RZ_Y(X)$. The proof of Proposition \ref{quacomprop}
carries over verbatim to prove the following Proposition.

\begin{prop}
\label{quacomproprel} The space $\gtX=\RZ_Y(X)$ is qcqs if and only if the scheme $X$ is qcqs.
\end{prop}

Let $\pi:\gtX\to X$ and $i:Y\to\gtX$ be the natural maps. We provide $\gtX$ with the sheaf
$\calM_\gtX=i_*(\calO_Y)$ of "meromorphic functions" and the sheaf $\calO_\gtX=\injlim\pi_i^{-1}(\calO_{X_i})$
of "regular functions".

\begin{exam}
(i) The classical (absolute) version of the above notions is obtained when $X$ is integral and $Y$ is the
generic point of $X$.

(ii) Another example is obtained when $X$ is of finite presentation over a valuation ring $R$ and $Y=X_\eta$ is
the generic fiber of $X$. Then $\gtX:=\RZ_Y(X)$ is the projective limit of all $\eta$-modifications of $X$, so
this space arise naturally when one studies some problems involving $\eta$-modifications.

(iii) Assume that $R$ is of height one in (ii). One can show that if $\hatX$ is the formal completion of $X$
along the special fiber and $\hatX_\eta$ is its "generic" fiber in the category of adic spaces, then the special
$R$-fiber of $\RZ_Y(X)$ is homeomorphic to $\hatX_\eta$ (the generic $R$-fiber of $\RZ_Y(X)$ is, obviously,
$Y$).

(iv) Although $f$ is a monomorphism (probably not of finite type) in (i)--(ii), there exist other interesting
examples. In \cite[\S1]{Tem1} and in the previous section, we considered the case when $Y$ is a point and $f$ is
dominant but not necessarily a monomorphism.
\end{exam}

\subsection{A method of proving $P$-modification theorems} \label{Pmodsec}
This is the only section in the paper where we weaken our assumptions on $\eta$. We only assume that $S$ is a
qcqs scheme and $|\eta|\subset|S|$ is a quasi-compact subset which is closed under generalizations and such that
$\eta=(|\eta|,\calO_S|_{|\eta|})$ is a scheme. Then $\eta$ is isomorphic to the scheme-theoretical projective
limit of its open neighborhoods. In particular, the natural embedding morphism $i_\eta:\eta\to S$ is a
quasi-compact monomorphism and any morphism $X\to S$ with image in $|\eta|$ factors through $\eta$ uniquely.
(So, $\eta$ is in a sense a "pro-open subscheme" of $S$.) We assume in addition that $i_\eta$ is schematically
dominant, i.e. $\eta$ is not contained in a proper closed subscheme of $S$.

Let $P$ be a property of morphisms of schemes. By a {\em $P$-modification statement} over $S$ we mean a
statement that if the "generic fiber" $\phi_\eta:X_\eta=X\times_S\eta\to\eta$ of a flat finitely presented
morphism $\phi:X\to S$ satisfies $P$ then there exist a base change morphism $S'\to S$ of a certain class $Q$
and a modification $\psi:X'\to X\times_S S'$ such that $X'\to S'$ is flat and satisfies $P$. (Such kind of
statements is called a permanence principle in the introduction to \cite{BLR}.) The class $Q$ can be, for
example, the class of all $\eta$-modifications (i.e. proper morphisms $S'\to S$ such that $\eta'=\eta\times_SS'$
is schematically dense in $S'$ and is mapped isomorphically onto $\eta$), the class of all generically \'etale
alterations, of all finite flat morphisms, etc.

\begin{exam}
In the following two examples $Q$ is the class of morphisms of the form $S'\to S''\to S$, where $S''\to S$ is an
$\eta$-modification and $S'\to S''$ is a finitely presented flat surjective morphism which is \'etale over
$\eta$.

(i) The {\em semi-stable modification theorem} is obtained when $P$ is being a semi-stable curve.

(ii) The {\em reduced fiber theorem} is obtained when $P$ is having geometrically reduced fibers.

Note that any morphism $f:S'\to S$ from $Q$ is \'etale over $\eta$. This is rather restrictive and as a
drawback, one has to allow reducible $S'$'s even when $S$ is irreducible. In the particular case when $\eta$ is
a point, one can impose an extra-condition that $f^{-1}(\eta)$ is a point. Then $Q$ reduces to the class of all
generically \'etale alterations.
\end{exam}

\begin{rem}
The reduced fiber theorem was proved by Bosch, L\"utkebohmert and Raynaud in \cite[2.1']{BLR} (the main theorem
2.1 of loc.cit. deals with its formal version). Also, it was conjectured (or hoped) in loc.cit. in the end of
the introduction that a semi-stable modification exists in all relative dimensions. It follows from simple
examples with two-dimensional bases, see \cite{AK}, conjecture 0.2, that semi-stable modification does not exist
in general. One can also construct analogous examples over the "two-dimensional" base $S=\Spec(\Kcirc)$, where
$K$ is a valued field with $|K^\times|\toisom\bfZ^2$. A possible salvage of the situation is to extend the class
of semi-stable morphism. For example, one can consider a wider class of poly-stable morphisms from
\cite[1.2]{Ber3}. The author expects (or hopes) that poly-stable modification is possible over any qcqs base
scheme.
\end{rem}

We will prove the two above modification theorems only when $\eta$ is the spectrum of a field (so, $S$ is
integral) and $Q$ is the class of generically \'etale alterations. The case of an arbitrary qcqs scheme $S$ will
be deduced in a subsequent work \cite{Tem2} by use of relative Riemann-Zariski spaces $\RZ_\eta(S)$. Both
theorems are proved via a similar method which, as the author hopes, can be useful when studying other
$P$-modification problems. So, it seems plausible to describe this method briefly. Note also that our method
seems to be very close to the approach of K. Fujiwara and F. Kato, as outlined in \cite{FK}.

(i) Uniqueness: add extra conditions to your problem, so that the required modification $\psi$ becomes uniquely
defined or functorial in $S'$.

(ii) Analytic input: prove the theorem over the valuation ring of a compete algebraically closed field $K$ of
height one.

(iii) Decompletion: use approximation to deduce the theorem over valuation rings of height one.

(iv) Induction on height: deduce the theorem over valuation rings of finite height.

(v) Limit: deduce the theorem over valuation rings.

(vi) The general case: use the Riemann-Zariski space $\RZ_\eta(S)$ and the uniqueness property from (i) to
deduce the general case.

The first two steps are critical. Naturally, one can hope to incorporate some non-Archimedean analytic geometry
over $K$ into the second step (it will be so in our two cases). We stress that it is necessary to consider the
case of an arbitrary analytic $K$, including the cases when $K$ is not isomorphic to the completed algebraic
closure of a discretely valued field, e.g. the case when $\rk_\bfQ(|K^\times|)>1$.

In the case of the reduced fiber theorem, we will take the Grauert-Remmert finiteness theorem (see Step 2 in the
proof of Theorem \ref{redfibth}) as an analytic input. To achieve functoriality we will consider
$\eta$-normalizations instead of arbitrary modifications $\psi$. For the semi-stable modification theorem, we
achieve uniqueness by considering stable modifications rather then semi-stable ones. Our main analytic input
here is the uniformization of analytic fields established by Theorem \ref{fieldunif} from which we deduced in
\S\ref{gencasesec} uniformization of valued fields via steps (iii)-(iv). In \S\ref{valcurves} we will deduce
stable modification over valuation rings, and in \S\ref{mainsec} we will work out step (vi) for stable
modification (with (v) obtained as a by-product). For the sake of completeness, we note that alternatively one
could deduce from Theorem \ref{fieldunif} analytic stable modification theorem of Bosch-L\"utkebohmert (it is
easy, and was done in unpublished master thesis of the author), and then steps (iii)-(vi) of the method could be
worked out for the assertion of the stable modification theorem itself.

\subsection{Reduced fiber theorem}
\label{redfibsec} In this section we apply the method from \S\ref{Pmodsec} to the reduced fiber theorem
\cite[2.1']{BLR}, see Theorem \ref{redfibth}. Recall that in this paper we treat only the particular case when
$\eta=\Spec(K)$. So, we will show that up to a generically \'etale alteration of the base, any finitely
presented morphism $X\to S$ with geometrically reduced $\eta$-fiber can be {\em $\eta$-modified} to a flat
morphism $X'\to S$ with geometrically reduced fibers (i.e. $X'\to X$ is proper and induces an isomorphism
$X'_\eta\toisom X_\eta$ on $\eta$-fibers).

\begin{lem}
\label{valringlem} Let $R$ be a valuation ring and $A$ be an $R$-algebra. Then $A$ is $R$-flat if and only if
$A$ has no $\pi$-torsion for any non-zero element $\pi\in R$. If $A$ is $R$-flat then it is finitely presented
over $R$ if and only if it is finitely generated over $R$.
\end{lem}
\begin{proof}
The first part is easy, so we omit the proof. The second statement is much deeper. It holds more generally over
any integral ring $R$, as proved in \cite[3.4.7]{RG}.
\end{proof}

The following result is critical for the proof of the reduced fiber theorem; it ensures uniqueness of the
$\eta$-modification in the theorem. The author is indebted to \cite[2.3(v)]{BLR} and \cite[2.5(c)]{BL2} for an
elegant idea of a proof based on the theory of depth and $Z$-closures as developed in \cite[$\rm
IV_2$,\S\S5.9-5.10]{ega}.

\begin{prop}
\label{redfibprop} Assume that $S$ is normal. Let $\phi:X\to S$ be a flat finitely presented morphism with
reduced geometric fibers. Then $X$ is $\eta$-normal in the sense that any finite $\eta$-modification of $X$ is
an isomorphism.
\end{prop}
\begin{proof}
Clearly, $X$ is reduced. Assume, on the contrary, that $X$ is not $\eta$-normal and pick up a non-trivial finite
$\eta$-modification $f:X'\to X$. Let us check that it is harmless to assume that $S,X$ and $X'$ are noetherian.
By noetherian approximation, see \S\ref{approxsec}, there exists a normal noetherian scheme $S_0$ with a
morphism $S\to S_0$ such that $f$ is induced from a modification $f_0:X'_0\to X_0$ of finitely presented
$S_0$-schemes. Moreover, by Remark \ref{prem}(i) we can achieve that $X_0$ has geometrically reduced
$S_0$-fibers. It now suffices to contradict that $f_0$ is not an isomorphism, so replacing $S$ and $f$ with
$S_0$ and $f_0$, respectively, we can assume that $S$, $X$ and $X'$ are noetherian.

\begin{lem}\label{genlem}
Let $\phi:X\to S$ be a flat finitely presented morphism of schemes. Then the set $(X/S)^0$ of the generic points
of the $S$-fibers of $X$ is closed under generalizations.
\end{lem}
\begin{proof}
The claim is local on $S$, so we can assume that it is qcqs. Then by noetherian approximation we reduce the
question to the case of a universally catenary noetherian $S$ (e.g. of finite type over $\bfZ$). For any point
$x\in X$ with $s=\phi(x)$ the inequality $\dim(\calO_{X,x})\ge\dim(\calO_{S,s})$ holds, and $x$ is in $(X/S)^0$
if and only if this is an exact equality. For any generalization $x'\succ x$ with $s'=\phi(x')$ the dimension
drop $\dim(\calO_{X,x})-\dim(\calO_{X,x'})$ cannot be smaller than $\dim(\calO_{S,s})-\dim(\calO_{S,s'})$
because any chain $s'\succ s_1\succ\dots\succ s$ can be lifted to a chain $x'\succ x_1\succ\dots\succ x$ by the
going down theorem \cite[9.5]{Mat} applied to the homomorphism
$\calO_{S,s}/m_{s'}\to\calO_{X,x}/m_{s'}\calO_{X,x}$. Therefore, if $x\in(X/S)^0$ and $x'$ is any its
generalization then the two dimension drops are equal and $\dim(\calO_{X,x'})=\dim(\calO_{S,s'})$. Thus,
$x'\in(X/S)^0$, as claimed.
\end{proof}

The Lemma implies that the set $U=X_\eta\cup(X/S)^0$ is closed under generalizations. By \cite[$\rm IV_4$,
17.5.1]{ega} $\phi$ is smooth at $(X/S)^0$, hence $X$ is normal at the points of $(X/S)^0$ by \cite[$\rm IV_2$,
6.8.3(i)]{ega}, and therefore $f$ is an isomorphism over $(X/S)^0$. Since $f$ is an $\eta$-modification, it is
an isomorphism over the whole $U$. Let $z$ be a point of $Z=X\setminus U$, $s=\phi(z)$ and $Y=X_s$ the $s$-fiber
of $X$. Then, $\prof(\calO_{X,z})=\prof(\calO_{Y,z})+\prof(\calO_{S,s})$ by \cite[$\rm IV_2$, 6.3.1]{ega} (set
$M=A=\calO_{S,s}$ and $N=B=\calO_{X,z}$ in loc.cit.). Since $z\notin Y^0$ and $Y$ has no embedded components,
$\prof(\calO_{Y,z})\ge 1$. Also, $\prof(\calO_{S,s})\ge 1$ because $s\neq\eta$ and $S$ is integral. Hence
$\prof(\calO_{X,z})\ge 2$, and \cite[$\rm IV_2$, 5.10.4]{ega} implies that $\calO_X$ is $Z$-closed. Recall that
the latter means that $\calO_X\toisom\calH^0_{X/Z}=\injlim (\pi_i)_*(\calO_{U_i})$, where $\{U_i\}$ is the
family of open neighborhoods of $U$ and $\pi_i:U_i\to X$ are the open immersions. It remains to notice that
$f_*(\calO_{X'})\subset\calH^0_{X/Z}$ because $f$ is an isomorphism near each $u\in U$. So,
$f_*(\calO_{X'})=\calO_X$, and therefore $f$ is an isomorphism.
\end{proof}

Proposition \ref{redfibprop} gives a new insight on the reduced fiber theorem. Note that $\eta$-normality does
not have to be preserved by base changes $f:S'\to S$ with normal $S'$, but the Proposition implies that if an
$\eta$-normal $S$-scheme $X$ has reduced geometric fibers then its base changes with normal $S'$'s are
$\eta$-normal. Thus, an $\eta$-normal $X$ with reduced geometric $S$-fibers can be viewed as {\em stably
$\eta$-normal} with respect to $S$. Then the reduced fiber theorem can be interpreted as a stabilization theorem
which states that if $X$ is finitely presented over $S$ and has geometrically reduced generic fiber, then it can
be made $\eta$-normal by a generically \'etale alteration of the base $S$ and subsequent $\eta$-normalization of
the base change of $X$. This also explains why the heart of the proof is the finite presentation result of
Grauert-Remmert (see Step 2 in the proof of Theorem \ref{redfibth}) -- we have to assure that stabilization can
be achieved after a reasonably small base change (e.g. after an alteration of the base).

\begin{cor}
\label{redfibcor} Let $S$ be an integral qcqs scheme, $\phi:X\to S$ be a flat finitely presented morphism and
$f_i:X_i\to X$, $i=1,2$, be two finite $\eta$-modifications. Assume that $X_1$ and $X_2$ are $S$-flat with
geometrically reduced fibers. Then there exists a finite modification $S'\to S$ such that $X_1\times_S S'$ and
$X_2\times_SS'$ are $X\times_S S'$-isomorphic and such an isomorphism is unique.
\end{cor}
\begin{proof}
Let $\tilS$ be the normalization of $S$. Then each $\tilX_i=X_i\times_S\tilS$ is a finite $\eta$-modification of
$\tilX=X\times_S\tilS$, which is $\eta$-normal by Proposition \ref{redfibprop}. So, each $\tilX_i$ is the
$\eta$-normalization of $\tilX$, i.e. $\tilX_1\toisom\tilX_2$. By approximation, this isomorphism descends to an
isomorphism $X_1\times_S S'\toisom X_2\times_S S'$ for a finite modification $S'\to S$.
\end{proof}

Probably, $X_i$ are already $X$-isomorphic, but proving this will require a new argument (similarly to the
situation with stable modifications, see Remark \ref{uniqrem}).

\begin{theor}
\label{redfibth} Let $X\to S$ be a dominant finitely presented morphism with a geometrically reduced generic
fiber. Then there exists a generically \'etale alteration $S'\to S$, and a finite $\eta$-modification $X'\to
X\times_S S'$ such that $X'$ is flat, finitely presented and has reduced geometric fibers over $S'$.
\end{theor}
\begin{proof}
Step 0. {\sl Flattening. There exists a modification $S'\to S$ and a finite $\eta$-modification $X'\to X\times_S
S'$ such that $X'$ is $S'$-flat.} It is a particular case of the flattening by blow ups theorem of
Raynaud-Gruson, see \cite[5.2.2]{RG}. Thus, we can assume that $X$ is $S$-flat.

Step 1. {\sl Localization. We can assume that $S=\Spec(R)$ and $X=\Spec(A)$ are affine.} Find a finite affine
covering $\{S_i\}$ of $S$ and finite affine coverings $\{X_{ij}\}$ of $\phi^{-1}(S_i)$. If the affine case is
established then we can find generically \'etale alterations $S'_{ij}\to S_i$ such that
$X_{ij}\times_{S_i}S'_{ij}$ admit finitely presented $\eta$-modifications $X'_{ij}$ with geometrically reduced
fibers. Notice that the same properties hold for any further alteration $S''_{ij}\to S'_{ij}$, hence by
Proposition \ref{oetamodprop} we can alter $S'_{ij}$ so that they are open subschemes of a generically \'etale
alteration $S'$ of $S$. Note that $\RZ_{K^s}(S'_{ij})=\RZ_{K^s}(S_i)$ form an open covering of
$\RZ_{K^s}(S)=\RZ_{K^s}(S')$, hence the schemes $S'_{ij}$ form a covering of $S'$. We did not rule out the
possibility that $X'_{ij}$ do not agree on intersections (i.e. that the preimages of $X_{ij}\cap X_{kl}$ in
$X'_{ij}$ and $X'_{kl}$ are not isomorphic), but we know from Corollary \ref{redfibcor} that $X'_{ij}$ do agree
after an additional finite modification $S''\to S'$ of the base. Then $X''_{ij}=X'_{ij}\times_{S'} S''$ glue to
a required $\eta$-modification $X''\to X\times_S S''$. This completes the step, and in the sequel we assume that
$S=\Spec(R)$ and $X=\Spec(A)$.

Step 2. {\sl Analytic input. Grauert-Remmert finiteness theorem.} We make ultimate use of the following fact,
which is a consequence of Grauert-Remmert theorem, see \cite[\S6.4 and 6.4.1/4]{BGR}: assume that $K$ is an
algebraically closed complete field of height one, $\calA$ is a reduced affinoid $K$-algebra and
$A\subset\calAcirc$ is a topologically finitely generated $\Kcirc$-subalgebra with
$A\otimes_{\Kcirc}K\toisom\calA$, then $\calAcirc$ is finite over $A$ (where $\calAcirc\subset\calA$ denotes the
subalgebra of power-bounded elements).

Step 3. {\sl Decompletion. The Theorem holds when $K$ is a separably closed valued field of height one and
$R=\Kcirc$.} Set $A_K=A\otimes_R K$, and let $\Acirc=\Nor_{A_K}(A)$ denote the integral closure of $A$ in $A_K$.
It suffices to show that $\Acirc$ is finitely presented over $R$ and has geometrically reduced fibers (actually
we know from Proposition \ref{redfibprop} that this is the only way the Theorem can hold for $S=\Spec(R)$ since
such $S$ has no non-trivial separable alterations). Choose a non-zero $\pi\in m_R$ and provide $A$ and $R$ with
the $\pi$-adic topology. We can assume that $\pi^{-1}\notin A$ because otherwise $A_K=A$ and there is nothing to
prove. Let $\hatA$ and $\hatR$ be the completions, then $\hatA$ is topologically finitely presented and flat
over $\hatR$, $\hatK=\hatR[\pi^{-1}]$ is the completion of $K$, and the $\hatK$-affinoid algebra
$\calA=\hatA[\pi^{-1}]$ is reduced by the following argument. It suffices to prove reducedness of the formal
completion of $\calA$ along a maximal ideal $m$. Set $A_\hatK=A\otimes_R\hatK$, then $\calM(\calA)$ is an
affinoid domain in the analytification of the $\hatK$-variety $\Spec(A_\hatK)$ given by $|f_i(x)|\le 1$ where
$f_1\. f_n$ is a set of generators of $A$ over $R$. Since $m$ corresponds to a Zariski closed (or rigid) point
of $\calM(\calA)$, the ideal $p=m\cap A_\hatK$ is a maximal ideal and $\wh{(A_\hatK)}_p\toisom\hatcalA_m$. But
the excellent ring $A_\hatK$ is reduced by our assumptions, hence so is any its completion along a maximal
ideal.

By Step 2, $\calAcirc$ is finite over $\hatA$, and since $\calAcirc\subseteq\calA=(\hatA)_\pi$ we obtain that
$\calAcirc=\hatA[a_1/\pi_1\dots a_n/\pi_n]$ for some choice of $a_i\in\hatA$ and $\pi_i\in R$. We can replace
$a_i$ with any element $a'_i$ with $a_i-a'_i\in\pi_i\hatA$. Since $A$ is dense in $\hatA$, we can achieve that
$a_i\in A$. Now, we will prove that $\Acirc$ is finitely presented over $R$ by showing that it coincides with
$B=A[a_1/\pi_1\. a_n/\pi_n]$, which is finitely presented by Lemma \ref{valringlem}. Note that for any finitely
presented and flat $R$-algebra $C$ and an element $\omega\in R$ we have that $\omega\hatC\cap C=\omega C$.
Indeed, if a sequence $\omega x_j$ of elements of $C$ converges to $x\in C$ then $x-\omega x_j$ is divisible by
$\omega$ for sufficiently large $j$, hence $x$ is divisible by $\omega$ as well. Now fix $1\le i\le n$. Since
$a_i/\pi_i$ is integral over $\hatA$, there exist $m\in\bfN$ and $b_j\in\hatA$ such that
$x:=a_i^m+b_1a_i^{m-1}\pi_i+\dots +b_{m-1}a_i\pi_i^{m-1}\in\pi_i^m\hatA$. The inclusion survives when we move
$b_j$'s slightly, hence we can achieve that $b_j\in A$. It then follows that $x\in\pi_i^m\hatA\cap A=\pi_i^m A$,
and therefore $a_i/\pi_i$ is integral over $A$. We obtain that $B\subseteq\Acirc$, so $B$ is integral over $A$.
Since $B$ is finitely generated over $A$ it is finite. Moreover, since $B$ is finitely presented over $R$, it is
finitely presented over $A$ as an algebra, and then it is finitely presented over $A$ as a module. The latter
implies that $\hatB=\hatA[a_1/\pi_1\dots a_n/\pi_n]=\calAcirc$. Let $b\in B$ and $\omega\in R$ be such that
$b/\omega\in\Acirc$, then $b/\omega\in\hatB$ and, as we saw earlier (with $B=C$), this implies that $b\in\omega
B$. Thus, $b/\omega\in B$, and we have proved that $B=\Acirc$, as required.

It remains to show that $\Acirc\otimes_R\tilK$ is geometrically reduced. Since $\tilK$ is algebraically closed
we have to prove that any non-zero element $\tila\in\Acirc\otimes_R\tilK$ is not nilpotent. If it is not so then
there exists an element $a\in\Acirc\setminus m_R\Acirc$ such that $a^n$ is divisible by an element $x\in m_R$.
Since $|K^\times|$ is divisible, we can replace $x$ with $y^n$ such that $|x|=|y|^n$. Then $a/y$ is in $\Acirc$
because $\Acirc$ is integrally closed in $\Acirc[\pi^{-1}]$, and we obtain a contradiction to the assumption
that $a\notin m_R\Acirc$.

Step 4. {\sl Composition and induction on height. The Theorem holds when $K$ is a separably closed valued field
of finite height and $R=\Kcirc$.} We will use Lemma \ref{composlem} to carry out induction on the height of $K$.
(Note that this process respects the separably closed hypothesis on fraction fields of valuation rings of finite
height under consideration.) An easy inductive argument on height proves that the group $|K^\times|$ is
divisible, and then the same proof as in the end of the previous step shows that $\Acirc\otimes_R\tilK$ is
reduced, where $\Acirc=\Nor_{A_K}(A)$, as earlier. Thus, we have only to prove that $\Acirc$ is finitely
presented over $R$, and by Lemma \ref{valringlem} it is enough to check that $\Acirc$ is finitely generated.
While proving the latter we can obviously replace $A$ with any finitely generated ring $B$ with $A\subseteq
B\subseteq\Acirc$. If $m$ is the minimal non-zero prime ideal of $R$ then $R_1:=R_m$ is the localization of $R$
of height one and its maximal ideal $mR_m$ coincides with $m$. Furthermore, $\tilR=R/m$ is a valuation ring with
the separably closed fraction field $L:=R_1/m$, and the valuation induced by $R$ is composed from those induced
by $R_1$ and $\tilR$.

Since the heights of $\tilR$ and $R_1$ are smaller than the height of $R$, we can assume that the Theorem holds
for $\tilR$ and $R_1$. So, for $A_1:=A\otimes_R R_1$ the algebra $\Acirc_1:=\Nor_{A_K}(A_1)$ is $R_1$-finitely
presented with a geometrically reduced special fiber $\tilA_1=\Acirc_1/m\Acirc_1$. Set $T=R\setminus m$, then
$R_1=R_T$ and $A_1=A_T$. Note also that $\Acirc_1=(\Acirc)_T$ because normalization is compatible with
localization. Let $a_1\. a_k$ be $R_1$-generators of $\Acirc_1$. Multiplying $a_i$'s by elements from $T$ we can
achieve that $a_i\in\Acirc$. Then we replace $A$ with $A[a_1\. a_k]$ achieving that $A_1=\Acirc_1$. Now, $A$ and
$\Acirc$ coincide after inverting $T$ and next we will study the situation modulo $m$. Since $m$ is
$T$-divisible due to the structure of $\Spec(R)$, $m\Acirc_1$ equals to both $mA$ and $m\Acirc$, and hence
$\tilA:=A/mA$ is embedded into $\tilAcirc:=\Acirc/m\Acirc$. Note that it is enough to prove that
$\tilAcirc=\tilA[\tilb_1\.\tilb_l]$ because it would follow immediately that $\Acirc=A[b_1\. b_l]$ for any
choice of liftings $b_i\in\Acirc$ of $\tilb_i$.

Note that $\tilAcirc$ is integral over $\tilA$ because $\Acirc$ is integral over $A$. Note also that applying
$\otimes_\tilR L$ to both rings we obtain the geometrically reduced $L$-algebra $\tilA_1$. It follows, that
$\tilAcirc$ is contained in $\Nr_{\tilA_1}(\tilA)$, which is finite over $\tilA$ by the induction assumption
(applied to $\tilR$). We claim that actually, $\tilAcirc=\Nr_{\tilA_1}(\tilA)$ and proving this will finish Step
4. Any element of $\Nr_{\tilA_1}(\tilA)$ is of the form $\tila/\tilpi$ where $a\in A$, $\pi\in T$ and
$\tila\in\tilA$, $\tilpi\in\tilR\setminus\{0\}$ are their reductions modulo $m$. Furthermore, there exist
$\tilb_1\.\tilb_n\in\tilA$ such that $\tilb_n=1$ and $\sum_{i=1}^n \tilb_i\tila^i\tilpi^{n-i}\in\tilpi^n\tilA$.
Lifting $\tilb_i$'s to some elements $b_i\in A$ with $b_n=1$ and using that $m\subset\pi^n A$ we obtain that
$\sum_{i=1}^nb_ia^i\pi^{n-i}\in\pi^nA$ and hence $a/\pi\in\Acirc$. Thus, the reduction $\tila/\tilpi$ lies
already in $\tilAcirc$ and we are done.

Step 5. {\sl A limit argument. The Theorem holds in general.} Since any valuation ring coincides with the union
of all its valuation subrings of finite height by Lemma \ref{finhtlem}(ii), the Theorem holds when $K$ is an
arbitrary valued field and $R=\Kcirc$. Let us pass to the general case. The Riemann-Zariski space
$\gtS=\RZ_{K^s}(S)$ introduced in \S\ref{absRZsec} is homeomorphic to the projective limit of all generically
\'etale alterations of $S=\Spec(R)$. To give a point $\gtx\in\gtS$ is equivalent to give a valuation ring
$\calO_{\gtS,\gtx}$ of $K^s$ which contains $R$. For any point $\gtx\in\gtS$ set
$A_\gtx=A\otimes_R\calO_{\gtS,\gtx}$. We know from the previous step that the $\calO_{\gtS,\gtx}$-algebra
$A'_\gtx=\Nor_{A_K}(A_\gtx)$ is finitely presented and has geometrically reduced fibers. It follows from Lemma
\ref{approxlem} by approximation that the morphism $\Spec(\calO_{\gtS,\gtx})\to S$ factors through a generically
\'etale quasi-alteration $S_\gtx=\Spec(R_\gtx)\to S$ satisfying the following condition: there exists a finite
$\eta$-modification $X'_\gtx\to X\times_S S_\gtx$ such that the geometric fibers of the morphism $X'_\gtx\to
S_\gtx$ are reduced.

By Corollary \ref{openembcor}, the Riemann-Zariski space $\gtS_\gtx=\RZ_{K^s}(S_\gtx)$ can be naturally
identified with an open subspace of $\gtS$ containing $\gtx$. Since $\gtS$ is quasi-compact by Proposition
\ref{quacomprop}, we can find finitely many points $\gtx_i\in\gtS$ such that the corresponding quasi-alterations
$S_{\gtx_i}$ are such that their Riemann-Zariski spaces cover $\gtS$. Now we act exactly as in Step 1. By
Proposition \ref{oetamodprop}, replacing $S_{\gtx_i}$ with generically \'etale alterations, we can achieve that
they glue to a generically \'etale alteration $S'\to S$. Then after an additional finite modification of the
base, the schemes $X'_{\gtx_i}$ glue to an $\eta$-modification $X'\to X\times_S S'$ which is as required.
\end{proof}

\section{Desingularization of curves over valuation rings} \label{valcurves}
Throughout \S\ref{valcurves} we assume that $S=\Spec(R)$ for a valuation ring $R$ of finite height and with
separably closed fraction field $K$. In particular, $|K^\times|$ is divisible and each point $s\in S$ has
separably closed residue field $k(s)$, which is even algebraically closed if $s\neq\eta$ (this is easily seen
for analytic fields, and the general case follows by decompletion and induction on height). For such $S$ we will
prove Theorems \ref{minstabtheor} and \ref{stabmodtheor}. The first proof is easy, and the second one runs in
two main stages. The first stage is standard, but slightly technical: we prove that any semi-stable modification
of $(C,D)$ can be blown down successively until a stable modification is obtained. The heart of the second stage
is Proposition \ref{locunifprop} which is an analog of local uniformization. It asserts that locally along a
valuation (which is interpreted as a point in a Riemann-Zariski space) $(C,D)$ admits a semi-stable
quasi-modification. We deduce this Proposition from uniformization of valued fields established in Theorem
\ref{valunif}. In the sequel, $(C,D)$ is a multipointed $S$-curve with reduced $C$ and $D$ and structure
morphism $(\phi:C\to S,\phi_D:D\to S)$. Other multipointed $S$-curves will be denoted as $(C',D')$, $(\oC,\oD)$,
etc.

\begin{remsect}\label{etarem}
Note that any $\eta$-modification $f:C'\to C$ extends uniquely to an $\eta$-modification $(C',D')\to(C,D)$ by
taking $D'$ to be the schematical closure of $D_\eta$ in $C'_\eta\toisom C_\eta$ (we use that $D'$ is finitely
presented by Lemma \ref{valringlem}). For this reason, we will often denote an $\eta$-modification
$(f_C,f_D):(C',D')\to(C,D)$ only by use of the $\eta$-modification $f=f_C:C'\to C$.
\end{remsect}

\subsection{Reduction to the case of a smooth generic fiber} \label{smfibsec}
Since $K$ is separably closed, it follows that any \'etale morphism $S'\to S$ is Zariski locally an isomorphism.
Also, normalization $\oC$ of $C$ is finitely presented over $S$ by Theorem \ref{redfibth} (and Proposition
\ref{redfibprop}), and so it is an $S$-curve. The semi-stable generic fiber $C_\eta$ can be obtained from its
normalization $\oC_\eta$ by gluing together pairs of points $(x_1,y_1)\.(x_n,y_n)$ where $z=\{z_1\.
z_n\}=(C_\eta)_\sing$ and $\oz=x\cup y$ is the preimage of $z$. In other words, $C_\eta$ is the pushout of the
diagram $\oC_\eta\la\oz\to z$ with the second map taking $x_i$ and $y_i$ to $z_i$. Let $X,X_i,Y,Y_i$ denote the
closures of $x,x_i,y,y_i$ in $\oC$, and define $Z\into C$ similarly (everything is reduced, so we can use the
usual Zariski closure instead of the schematical one). Also, let $\pi:\oC\to C$ be the projection and
$\oD=\pi^{-1}(D)\cup X\cup Y$. Note that if $(C',D')\to(C,D)$ is a semi-stable $\eta$-modification then $D'$ is
disjoint from $Z'$ (which is the closure of $z$ in $C'$) and hence the normalization $\oC'$ of $C'$ underlies a
semi-stable modification $\of:(\oC',\oD')\to(\oC,\oD)$. We will show that the converse is true when $\of$ is
projective, that is for any projective semi-stable modification $(\oC',\oD')\to(\oC,\oD)$ there exists a unique
semi-stable $\eta$-modification $f:(C',D')\to(C,D)$ such that $\oC'$ is the normalization of $C'$.

Let $\oX_i$ and $\oY_i$ be the closures of $x_i$ and $y_i$ in $\oC'$. Obviously, the only possible way for
defining $C'$ is to glue the closed subschemes $\oX=\sqcup\oX_i$ and $\oY=\sqcup\oY_i$ in $\oC'$ (disjointness
follows from semi-stability of $(\oC',\oD'$)). First, we note that indeed each $\oX_i$ is isomorphic to $\oY_i$
because they are $S$-\'etale and proper over $Z_i$. So, set $\oZ=\oX\sqcup\oY$ and let $\oZ\to Z$ be the
morphism identifying $\oX$ and $\oY$. Our aim now is to paste $\oX$ and $\oY$, that is to define a scheme $C'$
as the pushout of $\oC'\la\oZ\to Z$. (We will not need this, but it is well known that such pushout $C'$ always
exists in the category of algebraic spaces. In our case, $C'$ will be a scheme due to projectivity of
$\oC'\to\oC$.) We can work locally over $C$, so assume that $C$, and hence, $\oC$ are affine. Furthermore, it
suffices to construct the pushout $\pi':\oC'\to C'$ on a neighborhood of $\oZ$ in $\oC'$. By our assumption,
$\oC'$ is projective over the affine scheme $\oC$, hence the finite set $\oZ_0$ of closed points of $\oZ$
possesses an affine neighborhood. Since $\oZ$ is a semi-local scheme, we can replace $\oC'$ with this affine
neighborhood. We will no longer need the projectivity, so our problem reduces to the case when $\oC'=\Spec(A)$,
and then the affine pushout is defined as $C'=\Spec(B)$, where $B$ consists of all elements $h\in A$ such that
the restrictions of $h$ on the closed subschemes $\oX$ and $\oY$ are compatible with the isomorphism
$\oX\toisom\oY$. It is well known that $C'$ is the pushout in the category of schemes, so we omit this check. We
claim that $(C',D')$ is a semi-stable $\eta$-modification of $(C,D)$, where $D'=\pi'(\oD'\setminus\oZ)$. Indeed,
gluing along $S$-flat closed subschemes commutes with base changes, so we can check semi-stability on the
$S$-fibers, and then the claim is trivial.

Since any stable modification morphism is projective by Theorem \ref{projtheor}, it now suffices to prove
Theorems \ref{minstabtheor} and \ref{stabmodtheor} for the multipointed curve $(\oC,\oD)$ and the case of
$(C,D)$ will follow by the above pushout construction (we leave it to interested reader to check on $S$-fibers
that the pushout of $\oC_\st$ produces $C_\st$ and the normalization of $C_\st$ is $\oC_\st$). Thus, we reduced
the problem to the case of curves with smooth generic fiber.

\subsection{Modifications of curves over valuation rings}
Until the end of \S\ref{valcurves} we assume that $C_\eta$ is smooth. This assumption will slightly simplify the
terminology because it implies that any modification of $C$ is an $\eta$-modification. Nevertheless, all
intermediate results of \S\ref{valcurves} can be easily generalized at cost of replacing $\RZ_{k(C)}(C)$,
normality, modifications, etc., with $\RZ_{C_\eta}(C)$, $\eta$-normality, $\eta$-modifications, etc. In this
section we also assume that $D$ is empty. We can restrict ourselves to the case of a connected $C$, and then
$C_\eta$ is also connected by $S$-flatness of $C$. Thus, $C$ is irreducible, and we denote by $L$ the field of
rational functions of $C$. Note that any modification $C'$ of $C$ is an $S$-curve because $\calO_{C'}$ has no
$R$-torsion (and so $C'$ is $S$-flat and hence finitely presented).

Let $\gtC=\RZ_L(C)$ be the Riemann-Zariski space of $C$ as defined in \S\ref{absRZsec}. By Proposition
\ref{quacomprop} $\gtC$ is qcqs. The space $\gtC$ is provided with a sheaf of rings $\calO_\gtC$ whose stalks
are valuation rings of $L$. To give a point $\gtx\in\gtC$ is equivalent to give a valuation ring
$\calO_{\gtC,\gtx}$ containing $R$ and a morphism $\psi_\gtx:\Spec(\calO_{\gtC,\gtx})\to C$ respecting $L$.

Next, we attach to $C$ the set $(C/S)^0=\cup_{s\in S}C_s^0$, which is closed with respect to generalization by
Lemma \ref{genlem}. For any modification $C'\to C$, by $\Gamma(C')$ we denote the preimage of $(C'/S)^0$ in
$\gtC$. Note that $\Gamma(C')$ is contained in the subset $\gtC^0\subset\gtC$ which is defined as follows: a
point $\gtx\in\gtC$ is in $\gtC^0$ if the valuation ring $\calO=\calO_{\gtC,\gtx}$ satisfies
$\trdeg((\calO/m_\calO)/(R'/m_{R'}))=1$, where $R'=\calO\cap K$. In this case the valuation ring $\calO$ is
bounded over $R'$ in the sense of the definition above Theorem \ref{valunif}. We remark that analogs of $\gtC^0$
are the set of type $2$ points of a non-Archimedean analytic curve, or the set of divisorial valuations in the
Riemann-Zariski space of an algebraic surface. Let $C_0$ denote the set of closed points of $C$, and define
$\gtC^0_0$ and $\Gamma_0(C')$ to be the preimages of $C_0$ in $\gtC^0$ and $\Gamma(C')$, respectively. Although
we will not use that, we note that it is not difficult to show that $\gtC^0$ and $\gtC_0^0$ are the unions of
the sets $\Gamma(C')$ and $\Gamma_0(C')$, respectively, where $C'$ runs over all modifications of $C$.

Given a modification $f:Y\to X$, by the {\em modification locus} of $f$ we mean the minimal closed set $Z\subset
X$ such that $f$ is an isomorphism over $X\setminus Z$. Sometimes, we will treat $Z$ as a reduced closed
subscheme.

\begin{lem}
\label{normmodlem} Assume that $C$ is normal, $x\in(C/S)^0$ is a point and $f:C'\to C$ is a modification with
the modification locus $Z$. Then

(i) the set $f^{-1}(x)$ consists of a single point $x'$ and $\calO_{C,x}\toisom\calO_{C',x'}$;

(ii) $Z$ is quasi-finite over $S$, and so $|Z|$ is a finite set.
\end{lem}
\begin{proof}
Note that for any point $y\in(C/S)^0$, the fiber $f^{-1}(y)$ is a finite set of generic points of
$C'_{\phi(y)}$. Since $Y:=\Spec(\calO_{C,x})$ is contained in $(C/S)^0$ by Lemma \ref{genlem}, we obtain that
$Y'=Y\times_C C'$ has finite fibers over $Y$, i.e. the morphism $f':Y'\to Y$ is quasi-finite. By \cite[$\rm
IV_3$, 8.11.1]{ega}, the modification $f':Y'\to Y$ is finite. Since $Y$ is normal, $f'$ is an isomorphism. This
proves (i), and (ii) follows.
\end{proof}

The Lemma implies that any point $x\in(C/S)^0$ possesses a unique preimage $\gtx$ in $\gtC$ and
$\calO_{C,x}\toisom\calO_{\gtC,\gtx}$. In particular, $\calO_{C,x}$ is a valuation ring. The following obvious
observation will often be used in the sequel: for any neighborhood $U$ of $Z$ the modification $f$ is defined by
its restriction $f_U:f^{-1}(U)\to U$.

\begin{lem}\label{gamlem}
Assume that $C'$ is normal. Then $\Gamma(C')\toisom(C'/S)^0$, the set $\Gamma_0(C')$ is finite, and
$\Gamma(C')=gen(\Gamma_0(C'))\cup\Gamma(C)$, where $gen(\Gamma_0(C'))\subset\gtC^0$ is the set of all
generalizations of the points of $\Gamma_0(C')$.
\end{lem}
\begin{proof}
The bijection is explained above and obviously the sets $\Gamma(C')$ and $\Gamma_0(C')$ are finite. Furthermore,
the set $(C'/S)^0$ is closed under generalizations by Lemma \ref{genlem}, hence its preimage
$\Gamma(C')\subset\gtC$ is also closed under generalizations. In particular,
$gen(\Gamma_0(C'))\cup\Gamma(C)\subseteq\Gamma(C')$. To establish an equality it suffices to prove that if
$x\in(C'/S)^0$ is {\em contracted} in $C$ (i.e. is mapped to a closed point  in some $C_s$) then $x$ possesses a
specialization $y\in(C'/S)^0$ which is mapped to a closed point of $C$. The closure $X$ of $x$ has
one-dimensional $S$-fibers (i.e. any non-empty $S$-fiber is one-dimensional), but the image of $X$ in $C$ has
zero-dimensional $S$-fibers by Lemma \ref{genlem}. Let $s\in S$ be the closed point of the image of $X$, and
find an irreducible curve $Y\subset X_s$. Then $Y$ is closed in $X$, and it is contracted to a point in $C$
which is closed by properness of the morphism $C'\to C$. Thus, the generic point $y$ of $Y$ is as required.
\end{proof}

Suppose that $f_i:C_i\to C$, $i=1,2$, are two modifications. If $f_1$ factors through $C_2$ then we say that
$C_1$ {\em dominates} $C_2$. Since $f_i$ are modifications, the domination morphism $C_1\to C_2$ is unique.

\begin{prop}
\label{schmodprop} Let $C'$ and $C''$ be two modifications of $C$ with $\Gamma_0(C')\subseteq\Gamma_0(C'')$, and
assume that $C''$ is normal. Then $C''$ dominates $C'$.
\end{prop}
\begin{proof}
Lemma \ref{gamlem} implies that $\Gamma(C')\subseteq\Gamma(C'')$. Let $\oC$ be the schematical closure of
$\Spec(L)$ in $C'\times_C C''$; it is a modification of $C$ which dominates both $C'$ and $C''$. If
$\Gamma(\oC)=\Gamma(C'')$ then the modification $f':\oC\to C''$ is finite because its $S$-fibers $f'_s:\oC_s\to
C''_s$ do not contract components. By normality of $C''$, $f'$ is an isomorphism, hence $C''$ dominates $C'$, as
required.

It remains to prove that indeed $\Gamma(\oC)=\Gamma(C'')$. Suppose on the contrary that
$\gtx\in\Gamma(\oC)\setminus\Gamma(C'')$. Let $s$ and $x$ be the images of $\gtx$ in $S$ and $\oC$ then
obviously $x\in(\oC/S)^0$. From other side the images of $x$ in $C'_s$ and $C''_s$ are closed points because
$\gtx$ is not in $\Gamma(C'')$ and $\Gamma(C')$. Therefore $x$ has to be a closed point of $\oC_s$, and the
contradiction finishes the proof.
\end{proof}

\subsection{Local uniformization}\label{locunifsec}
For any affine scheme $\oS=\Spec(A)$ consider two examples of multipointed semi-stable $\oS$-curves:
$C'=\Spec(A[T])$ and $D'=\{T=0\}$, or $C'=\Spec(A[U,V]/(UV-a))$, $a\in A$ and $D'=\emptyset$. A multipointed
$S$-curve $C$ is called {\em strictly semi-stable}, if locally on $C$ and $\oS$, $C$ admits an \'etale morphism
$f:(C,D)\to(C',D')$ (i.e. both $f_C$ and $f_D$ are \'etale) to one of the above curves. Obviously, strict
semi-stability implies semi-stability. If $\oS=\Spec(k)$ for a separably closed field then the strictness
condition means that the irreducible components of $C$ are smooth.

\begin{prop}
\label{cofinprop} Assume that $C$ is a strictly semi-stable $S$-curve with smooth generic fiber $C_\eta$. Then
strictly semi-stable blow ups of $C$ (i.e. blow ups $C'\to C$ with $C'$ strictly semi-stable over $S$) are
cofinal in the family of all its modifications.
\end{prop}
\begin{proof}
We claim that it is enough to prove the following claim: for any element $\gtx\in\gtC^0_0$ there exists a
strictly semi-stable blow up $h:C'\to C$ such that $\gtx\in\Gamma_0(C')$. Indeed, it is well known that blow ups
are preserved under compositions (see, for example, \cite[1.2]{Con}), hence given any finite subset
$F\subset\gtC^0_0$ we can apply the claim iteratively to construct a strictly semi-stable blow up $C'\to C$ with
$F\subset\Gamma_0(C')$. Then Proposition \ref{cofinprop} would follow from Proposition \ref{schmodprop} and
normality of semi-stable $S$-curves with smooth generic fiber.

Now, let $x\in C$ be the center of $\gtx\in\gtC^0_0$. We will act in two stages. At first stage we will blow $C$
up so that it becomes $S$-smooth at the center of $\gtx$ and at the second stage we will make the center a
generic point of its $S$-fiber. Note that if at some stage $x$ is not a closed point of its fiber then it is
already a generic point of its $S$-fiber and we are done. We will build these blow ups so that $x$ is the only
closed point of the modification locus; this allows to work locally at $x$ since each such blow up extends
trivially from a neighborhood of $x$ to all of $C$. In particular, we can replace $R$ with its localization
beneath $x$, so now $x$ is in the special fiber.

At first stage we assume that $C$ is not $S$-smooth at $x$. Localizing we can assume that there exists an
\'etale morphism $C\to\oC$ with $\oC=\Spec(A)$, $A=R[u,v]/(uv-a)$, $a\in R$, $a\neq 0$ and such that the image
of $x$ is the point $\ox$ with $u(\ox)=v(\ox)=0$. Let $\ogtx$ be the image of $\gtx$ in $\RZ_{k(\oC)}(\oC)$,
that is $\calO_\ogtx=\calO_\gtx\cap k(\oC)$. It is enough to find a blow up $\of:\oC'\to\oC$ such that $\ogtx$
is centered on a smooth point $\ox'\in\oC'$ and $\ox$ is the only closed point of the modification of $\of$.
Indeed, if such $\of$ exists then $C':=\oC'\times_\oC C$ is a strictly semi-stable blow up of $C$ and the center
$x'\in\ C'$ of $\gtx$ is smooth since it is the preimage of $\ox'$ under the \'etale morphism $C'\to\oC'$. Thus,
replacing $C$ and $\gtx$ with $\oC$ and $\ogtx$, we can assume that $C=\Spec(A)$. Consider the valued fields $L$
and $K$ with the valuations induced by $\calO=\calO_{\gtC,\gtx}$ and $R$. Since
$\trdeg(L/K)=1=\trdeg(\tilL/\tilK)$, $E_{L/K}=0$ by the Abhyankar inequality. In particular, the group
$G=|L^\times|/|K^\times|$ is a torsion group. By our assumption $K$ is separably closed, hence $|K^\times|$ is
divisible, $G=1$ and $|R|=|\calO|$. In particular, we can find an element $\pi\in R$ such that $|u|=|\pi|$ in
$L$. Note that $|uv|=|a|$ and $|v|\le 1$, hence $|\pi|\ge |a|$ and so $\pi|a$ in $R$. Let $C'=\Bl_{(u,\pi)}(C)$
be the blow up of $C$ along the ideal $(u,\pi)$. Since $C'$ is covered by the charts
$\Spec(A[v,\frac{u}\pi]/(v\frac{u}\pi-\pi^{-1}a))$ and $\Spec(A[u,\frac\pi{u}]/(u\frac\pi{u}-\pi))$ (where the
fractions serve as indeterminants, as is customary for blow up formulas), we obtain that $C'$ is semi-stable.
Finally, $x$ is the only closed point of the modification locus of $C'\to C$, and the center $x'\in C'$ of
$\gtx$ is contained in the $S$-smooth open subscheme
$\Spec(A[\frac{u}\pi,\frac\pi{u}]/(\frac{u}\pi\frac\pi{u}-1))$ of $C'$ because $\frac\pi u\in\calO^\times$.

Now we assume that the center $x$ is smooth and we want to make it the generic point of an irreducible component
in the fiber. Let $s\in S$ be the image of $x$. Locally at $x$ there exists an \'etale morphism $C\to\Spec(A)$
with $A=R[T]$, and the same argument as in the first stage shows that we can actually assume that $C=\Spec(A)$.
Consider the valued fields $L$ and $K$ as earlier. Note that $\tilL$ is generated over $\tilK$ by the residues
of the elements $f(T)/\pi$, where $f$ is an irreducible monic polynomial and $\pi\in R$. Indeed, any element
$f(T)\in L\setminus R$ with $|f(T)|=1$ can be represented as $\prod (f_i(T_i)/\pi_i)^{\ve_1}$, where $f_i$ are
irreducible over $K$, $\pi_i\in R$, $\ve_i\in\{\pm 1\}$ and $|f_i/\pi_i|=1$. Moreover, we can take only
separable irreducible polynomials because the residues of $f(T)/\pi$ and $(f(T)+\omega T)/\pi$ coincide for any
$\omega\in\pi m_R$. Since any separable irreducible polynomial of $K[T]$ is linear, we can find $a\in K$ and
$\pi_a\in R$ such that the residue of $(T-a)/\pi_a$ is transcendental over $\tilK$. Note that $|\pi_a|<1$ as
otherwise $\pi_a\in R^\times$ and $\gtx$ would be centered on the generic point of $C_s=\Spec(k(s)[T])$,
contradicting the assumption that $x$ is closed. Also, if $\tila\in k(s)$ is such that $x$ is the point of $C_s$
given by $T=\tila$ and $a'\in R$ is a lifting of $\tila$, then $|T-a'|<1$ and hence $|a-a'|<1$. So, $a\in R$ and
we can consider the blow up $C'=\Bl_{(\pi_a,T-a)}(C)$ covered by the charts $\Spec(A[\frac{T-a}{\pi_a}])$ and
$\Spec(A[T,\frac{\pi_a}{T-a}]/((T-a)\frac{\pi_a}{T-a}-\pi_a))$. Now one checks straightforwardly that $C'$ is a
strictly semi-stable modification of $C$, $x$ is the closed point given by $T-a=0$ and it is the only closed
point of the modification locus of $h:C'\to C$, $Z=h^{-1}(x)$ is a $\bfP^1_{k(s)}$-component of $C'$ and $\gtx$
is centered at the generic point of $Z$, as required.
\end{proof}

We will also need the following Lemma, where a nonempty $D$ is allowed.

\begin{lem}\label{divlem}
Assume that for a multipointed $S$-curve $(C,D)$ with a semi-stable generic fiber $(C_\eta,D_\eta)$ the
$S$-curve $C$ is strictly semi-stable with smooth generic fiber $C_\eta$. Then there exists a blow up $C'\to C$
underlying a strictly semi-stable modification $(C',D')\to(C,D)$.
\end{lem}
\begin{proof}
We will construct the required blow up by composing few intermediate blow ups. The assumptions that $K$ is
separably closed and $D_\eta$ is $K$-smooth imply that $D_\eta$ is a union of few smooth $K$-points. It is well
known that there exists a modification $f:C'\to C$ which separates the irreducible components $D_i$ of $D$. For
example, since $D'$ is the strict transform of $D$ it follows from \cite[1.4]{Con} that one can separate each
pair $D_{i_1}$ and $D_{i_2}$ by blowing up $C$ along $D_{i_1}\times_C D_{i_2}$. Applying this procedure few
times we can separate all components. By Proposition \ref{cofinprop} $C'$ admits a strictly semi-stable
modification $C''$ which is a blow up of $C$, hence we can replace $(C,D)$ with a modification $(C'',D'')$
achieving that $C$ is strictly semi-stable over $S$ and the irreducible components of $D$ are disjoint. Using
Lemma \ref{valringlem} to show that $D\to S$ is finitely presented, we obtain that $D$ is Zariski locally
isomorphic to $S$. Thus, the only problem can arise if $D$ intersects the singular locus of $\phi:C\to S$.

We will show that there exists a semi-stable blow up $(C',D')\to(C,D)$ whose center is contained in $D$. Note
that it suffices to prove this stronger claim locally at a closed point $d\in D\cap C_\sing$. In particular,
similarly to the argument in Proposition \ref{cofinprop} it suffices to study the following model situation:
$C=\Spec(R[u,v]/(uv-a))$ and $D\toisom S$ is given by $u=\pi$ for an element $\pi\in R$ with $|a|<|\pi|<1$. Then
(again, similarly to the proof of Proposition \ref{cofinprop}) the blow up $C'=\Bl_{(\pi,u)}(C)$ is a strictly
semi-stable modification of $C$ such that the pair $(C',D')$ is as required because $D'$ is contained in
$\Spec(R[\frac{u}\pi,\frac\pi{u}]/(\frac{u}\pi\frac\pi{u}-1))$ and hence is disjoint form the singular locus of
$C'$.
\end{proof}

Finally, we establish a local uniformization of an $S$-curve along a valuation.

\begin{prop}
\label{locunifprop} Let $C\to S$ be a curve with smooth $C_\eta$. Then there exist quasi-modifications $C_1\to
C\. C_m\to C$ such that $C_i$ are strictly semi-stable $S$-curves and any point of $\gtC$ is centered on some
$C_i$.
\end{prop}
\begin{proof}
We can assume that $C=\Spec(A)$ is affine. Since $\gtC$ is quasi-compact and $\RZ_L(C_i)\into\gtC$ is an open
space embedding by Corollary \ref{openembcor}, it suffices to show that any point $\gtx\in\gtC$ can be centered
on a strictly semi-stable quasi-modification of $C$. To simplify notation we set $\calO=\calO_{\gtC,\gtx}$.
Also, we provide $L=k(C)$ with a valuation induced by $\calO$.

By Theorem \ref{valunif} the one-dimensional valued $K$-field $L$ is uniformizable because $K\calO$ is centered
on the smooth $K$-curve $C_\eta$. Therefore there exists an element $T\in L$ such that $L$ is unramified over
$k(T)$, and so $\calO$ is \'etale over $\calO\cap k(T)$ by Lemma \ref{finhtlem}. Replacing $T$ with $T^{-1}$, if
necessary, we can assume that $T\in\calO$. Furthermore, replacing $A$ with $A[T]\subset L$ and $C$ with its
quasi-modification $\Spec(A[T])$ we achieve that $T\in A$ while $\calO$ is still centered on $C$. In particular,
we obtain a morphism $T:C\to\oC=\Spec(\oA)$, where $\oA=k[T]$. Let $\ogtC=\RZ_{k(T)}(\oC)$ be the
Riemann-Zariski space of $\oC$, then a natural map $\gtC\to\ogtC$ arises, and we denote the image of $\gtx$ by
$\ogtx$. Notice that $\ocalO=\calO_{\ogtC,\ogtx}$ coincides with $\calO\cap k(T)$, so $\calO$ is \'etale over
$\ocalO$. In the sequel, we will few times replace $C$ and $\oC$ with their affine quasi-modifications whose
preimages in $\gtC$ and $\ogtC$ contain $\gtx$ and $\ogtx$, respectively. For simplicity, the new curves will be
also denoted $C=\Spec(A)$ and $\oC=\Spec(\oA)$.

Let $T_1\. T_l$ be $\ocalO$-generators of $\calO$. Replacing $C$ with a quasi-modification if necessary, we can
assume that $T_i\in A$. By flattening theorem, there exists a modification $\oC'\to\oC$ such that the
schematical closure of $\Spec(L)$ in $C\times_\oC\oC'$ is flat over $\oC'$. Thus, replacing $C$ and $\oC$ with
their quasi-modifications, we can achieve that $A$ is $\oA$-flat. It follows that $A\otimes_\oA\ocalO$ is a
subalgebra of $L$ containing $T_i$'s, hence $A\otimes_\oA\ocalO\toisom\calO$. Let $\{\oC_i\}_{i\in I}$ denote
the family of modifications of $\oC$, and let $\ox_i$ denote the centers of $\ogtx$ on $\oC_i$, then $\cup_{i\in
I}\calO_{\oC_i,\ox_i}\toisom\ocalO$. By \cite[$\rm IV_4$, 17.7.8]{ega}, $A\otimes_\oA\calO_{\oC_i,\ox_i}$ is
\'etale over $\calO_{\oC_i,\ox_i}$ for some $i\in I$. Moreover, by Proposition \ref{cofinprop} enlarging $i$ we
can achieve that $\oC_i$ is strictly semi-stable. Then the center of $\gtx$ on $C_i=C\times_{\oC}\oC_i$ is
contained in the strictly semi-stable locus of $C_i$. So, this locus is a required quasi-modification of $C$.
\end{proof}

\subsection{Blowing down to a stable modification} \label{blowdownsec}
In this section we will prove that any semi-stable modification can be blown down to a stable modification. In
addition, we will prove Theorem \ref{minstabtheor} (in the case of $S=\Spec(\Kcirc)$). Similarly to
Castelnuovo's contraction of exceptional curves on surfaces, this involves some cohomological technique. We will
need the notion of $\bfP_k^1$-trees introduced in Appendix \ref{B}.

\begin{lem} \label{compaclem}
Let $C$ be an $S$-curve with geometrically reduced $C_\eta$.

(i) $\Nr_{C_\eta}(C)$ is an $S$-curve (i.e. it is finitely presented over $S$) with geometrically reduced
fibers, and $\Nr_{C_\eta}(C)=\Nr(C)$ if $C_\eta$ is smooth.

(ii) If $C$ is normal and affine then it can be embedded into a projective $S$-curve $\oC$ with geometrically
reduced $S$-fibers.
\end{lem}
\begin{proof}
We will prove (ii) and the proof of (i) is similar (using Proposition \ref{redfibprop}). Since a reduced
$K$-scheme is geometrically reduced if and only if it is generically so, we can easily find an $S$-projective
compactification $C\into C'$ with geometrically reduced $C'_\eta$. Recall that any separable alteration of $S$
is an isomorphism, hence by Theorem \ref{redfibth} there exists a finite $\eta$-modification $\oC\to C'$ that
has geometrically reduced $S$-fibers. This map is an isomorphism over $C$, so $\oC$ is as required.
\end{proof}

\begin{lem}
\label{stabcritlem} Let $f:(C',D')\to(C,D)$ be a modification of a multipointed $S$-curve $(C,D)$, $x\in C$ be a
point, $Z=f^{-1}(x)$ and $s=\phi(x)$. Assume that $C$ is normal and $Z$ is a curve contained in the semi-stable
locus of $C'$. Then $(C,D)$ is semi-stable at $x$ if and only if $Z$ is an exceptional $\bfP^1_{k(s)}$-tree in
the sense of Appendix \ref{B} (i.e. $|\partial_{C'_s}(Z)|+|D'_s\cap Z|\le 2$).
\end{lem}
\begin{proof}
The question is local in $x$. Obviously, $x$ is closed in its fiber $C_s$. Localizing $R$ we can assume that $s$
is the closed point of $S$, and then $x$ is closed in $C$. Shrinking $C$ we can assume that $C$ is connected,
normal and affine. The modification locus $V$ of $f$ is quasi-finite over $S$, hence by Lemma \ref{finhtlem}
$V_x:=\Spec(\calO_{V,x})$ is an open neighborhood of $x$ in $V$. Thus, shrinking $C$ again, we can achieve that
$x$ is the only closed point of $V$. By Lemma \ref{compaclem}, we can embed $C$ into a connected normal
$S$-projective curve $\oC$. Define $\oD$ as the closure of $D$ in $\oC$. Since $V$ is closed in $\oC$, we can
extend $f$ trivially outside of $C$ obtaining a modification $f':\oC'\to\oC$. So, it suffices to solve our
problem for the projective multipointed $S$-curves $\oC$ and $\oC'$.

Now, we can assume that $C$ is $S$-projective. By Lemma \ref{compaclem}(i), we can replace $C'$ with
$\Nr_{C'_\eta}(C')$ (no change near $Z$ by Proposition \ref{redfibprop}) achieving that $C'$ has geometrically
reduced $S$-fibers. Note that any finite, connected and dominant $S$-scheme maps bijectively onto $S$. Applying
the Zariski connectedness theorem, we obtain that the $S$-fibers of $C'$ are connected (usually, Zariski
connectedness theorem is formulated for a noetherian base scheme $S$, e.g. \cite[$\rm III_1$, 4.3.2]{ega}, but
$C'$ comes from a curve defined over a noetherian base). It follows that $h^0(C'_t)=1$ for any $t\in S$. Since
the Euler-Poincare characteristic of the fibers is constant on $S$ (again, \cite[$\rm III_2$, 7.9.4]{ega} is
formulated for noetherian schemes, but the general case follows by noetherian approximation), we see that the
arithmetic genus $h^1(C'_t)$ of the fibers is constant on $S$. Hence $h^1(C'_s)=h^1(C_\eta)=h^1(C_s)$, and the
case of connected $Z$ follows from Corollary \ref{stabdowncor}. Finally, $Z$ is connected because otherwise $C$
is not normal at $x$ by Zariski connectedness theorem.
\end{proof}

\begin{cor} \label{sepcloscor}
Assume that $C_\eta$ is smooth and $(C,D)$ has a stable modification $(C_\st,D_\st)$ and a semi-stable
modification $(C',D')$.

(i) If both modifications are dominated by a semi-stable modification $(\oC,\oD)$ then $C'$ dominates $C_\st$.

(ii) The assumption of (i) is satisfied when $C$ is a modification of a strictly semi-stable $S$-curve $C_1$.
\end{cor}
\begin{proof}
For (i), we note that Lemmas \ref{stabcritlem} and \ref{stabblowdown} imply for any $s\in S$ the following
claim: if an irreducible component $Z\subset\oC_s$ is contracted in $C'_s$ then it is contracted in $(C_\st)_s$
too. Therefore, $C'$ dominates $C_\st$ by Proposition \ref{schmodprop}. To prove (ii) we note that both $C'$ and
$C_\st$ are modifications of $C_1$ and so they are dominated by a strictly semi-stable modification $\tilC$ by
Proposition \ref{cofinprop}. Let $\tilD$ be the Zariski closure of $D_\eta$ in $\tilC$. Then the multipointed
curve $(\tilC,\tilD)$ possesses a semi-stable modification $(\oC,\oD)$ by Lemma \ref{divlem}. Obviously, the
latter dominates both $(C_\st,D_\st)$ and $(C',D')$.
\end{proof}

Let us assume that $(C',D')$ is a semi-stable modification of $(C,D)$. We will show that by successive blowing
down exceptional components of $S$-fibers of $C'$ one can construct a stable modification $(C_\st,D_\st)$ of
$(C,D)$. Let $E=E(C',D')$ be the set of exceptional components of the fibers $C'_s$. We identify $E$ with a
subset of $\Gamma(C')$ and set $E_0=E_0(C',D')=E\cap\Gamma_0(C')$.

\begin{lem}\label{emptylem}
Assume that $C_\eta$ is smooth. A semi-stable modification $f:(C',D')\to(C,D)$ is stable if and only if $E_0$ is
empty.
\end{lem}
\begin{proof}
Only the converse implication needs a proof. Any semi-stable modification factors through the normalization of
$C$, hence we can replace $C$ with its normalization (and update $D$ accordingly). Assume that $Z\subset C'_s$
is an exceptional component. Acting as in the proof of Lemma \ref{stabcritlem} one establishes the following
claim: if $t\in S$ is a specialization of $s$ then any irreducible component $Z'\subset C'_t$ lying in the
Zariski closure of $Z$ is exceptional. Indeed, $Z'$ lies in the modification locus of $f$, hence it is
contracted to a point $x\in C$ by normality of $C$. Now similarly to the mentioned proof we shrink $C$ about
$x$, compactify it and compute genera. It remains to notice that the generic point of $Z$ has a specialization
$z\in\Gamma_0(C')$, and therefore $z$ is the generic point of an exceptional component belonging to $E_0$. The
contradiction shows that $E$ is empty, i.e. the modification $f$ is stable.
\end{proof}

In the sequel, by {\em exceptional blow down} of $(C',D')$ we mean a normal modification $(C'',D'')\to(C,D)$
such that $C'$ dominates $C''$, the morphism $C'\to C''$ contracts exactly one point of $\Gamma_0(C')$ and that
point is in $E=E(C',D')$ (so it is even in $E_0$).

\begin{lem}
If $(C'',D'')$ is an exceptional blow down of $(C',D')$ then $(C'',D'')$ is semi-stable.
\end{lem}
\begin{proof}
Semi-stability is an open condition by Lemma \ref{openloclem}(i) below (no circular reasoning occurs here),
hence it suffices to prove that $(C'',D'')$ is semi-stable at closed points. Let $Z\in E_0$ be the component
contracted in $C''$ and $x\in C''$ be its image, then $x$ is the only closed point of the modification locus of
$C'\to C''$. It remains to note that $(C'',D'')$ is semi-stable at $x$ by Lemma \ref{stabcritlem}.
\end{proof}

\begin{prop}
\label{blowdownprop} Assume that $C_\eta$ is smooth. If $(C,D)$ admits a semi-stable modification $(C',D')$ then
it admits a stable modification $(C_\st,D_\st)$.
\end{prop}
\begin{proof}
Assume that $(C',D')$ is not stable. Let $Z\in E_0$ be an exceptional component of $C'_s$; it exists by Lemma
\ref{emptylem}. It suffices to find an exceptional blow down $C'\to C''$ which contracts $Z$. Indeed:
$(C'',D'')$ is semi-stable by the above Lemma, and if $E_0(C'',D'')$ is not empty then we can blow down $C''$
further, etc. This process must stop because we can perform at most $|\Gamma_0(C')|$ contractions.

Now we will reduce to the case of $S$-projective $C$ as in the proof of Lemma \ref{stabcritlem}. Since $f:C'\to
C$ factors through $\Nor(C)$, we can assume that $C$ is normal. Let $x$ be the image of $Z$ in $C$ and
$s=\phi(x)$. The problem is local in $x$, hence we can localize $S$ and shrink $C$, so that $s$ is closed, $C$
is connected and affine and $x$ is the only closed point of the modification locus $V$ of $f$ ($V$ is
quasi-finite over $S$ and we use Lemma \ref{finhtlem}). Embed $C$ into a connected normal $S$-projective curve
$\oC$ (and take $\oD$ to be the closure of $D$) and define $(\oC',\oD')\to(\oC,\oD)$ as the trivial extension of
$(C',D')\to(C,D)$. It now suffices to find an exceptional blow down $\oC'\to\oC''$ which contracts $Z$. To
simplify the notation we replace $\oC'\to\oC$ with $C'\to C$ achieving that $C$ is $S$-projective.

Let $Z_1\. Z_n$ be other irreducible components of $C'_s$. For any $1\le i\le n$, find a point $P_i\in C'_\eta$
such that the $s$-fiber of its closure $\oP_i$ is a smooth point $p_i\in Z_i$, and so $\oP_i$ lies in the smooth
locus of $C'$. In particular, $\cup\oP_i$ is a Cartier divisor. Note that $Y:= C'_s=(\cup Z_i)\cup Z$,
$Z\toisom\bfP^1_{k(s)}$ intersects $\cup Z_i$ transversally, and the intersection is exactly
$\partial_{C'_s}(Z)$, so it contains at most two points. Using that $h^1(Z,\calO_Z)=0$ it follows easily that
$H^1(\calO_Y(m\sum p_i))=0$ for sufficiently large $m$.

Consider the $R$-flat sheaf $\calL=\calO_{C'}(m\sum\oP_i)$. Then $h^1(\calL_s)=0$ and therefore
$h^1(\calL_\eta)=0$ by semi-continuity. Since the Euler-Poincare characteristic of $\calL$ is constant,
$h^0(\calL_\eta)=h^0(\calL_s)$. Applying the theorem of Grauert and Grothendieck on base changes and direct
images, see \cite[III.12.9]{Har}, or \cite[$\rm III_2$, 7.6.9 and 7.7.5]{ega}, we obtain that the homomorphism
$H^0(\calL)\to H^0(\calL_s)$ is onto (the cited results are formulated in noetherian setting, so we use
noetherian approximation).

For sufficiently large $m$, there exists a section $h_s\in H^0(\calL_s)$ which does not vanish at $p_i$'s. Find
a lifting $h\in H^0(\calL)$ and consider it as a meromorphic function on $C'$. The pole divisor of $h$ is at
most $\oP=m\sum\oP_i$, and the zero divisor $V$ does not pass through the points $p_i$ because $h_s$ has poles
of order exactly $m$ at $p_i$. Thus, $V\cap\oP\cap C'_s=\emptyset$, and therefore the intersection of $V$ and
$\oP$ is empty. It follows that $h$ defines a morphism $C'\to\bfP^1_S$, the induced morphism $\oh:C'\to\bfP^1_C$
contracts $Z$ to a point and $Z$ is the only component of $C'_s$ contracted by $\oh$. Hence, one can take $C''$
to be the normalization of $\oh(C')$ and define $D''$ accordingly (with semi-stability of $(C',D')$ following
from Lemma \ref{stabcritlem}).
\end{proof}

\subsection{Gluing local models}\label{gluesec}

\begin{prop}
\label{sepclosprop} Theorems \ref{minstabtheor} and \ref{stabmodtheor} hold for $S=\Spec(\Kcirc)$ where $K$ is a
separably closed valued field of a finite height.
\end{prop}
\begin{proof}
We showed in \S\ref{smfibsec} that the case of smooth $C_\eta$ implies the general case by a pushout procedure.
So, we assume in the sequel that $C_\eta$ is smooth. Note that if $C$ is a modification of a strictly
semi-stable $S$-curve then Theorem \ref{minstabtheor} holds true for $(C,D)$ by Corollary \ref{sepcloscor}.
Next, let us prove Theorem \ref{stabmodtheor}. It suffices to prove that $(C,D)$ possesses a semi-stable
modification because then we can construct the stable modification by Proposition \ref{blowdownprop}. By
Proposition \ref{locunifprop} there exist strictly semi-stable quasi-modifications $C_1\. C_m$ of $C$ such that
any element of $\gtC$ is centered on some $C_i$. Let $D_i$ be the closure of $D_\eta\cap (C_i)_\eta$ in $C_i$.

By Proposition \ref{oetamodprop} there exists a modification $C'\to C$ and open subschemes $C'_i\subset C'$ such
that $C'_i$ are $C$-isomorphic to modifications of $C_i$. Let $D'_i\into C'_i$ be the closure of $(D_i)_\eta$ in
$C'_i$. By our assumptions, $\gtC$ is covered by the subspaces $\RZ_L(C_i)$, where $L=k(C)$. Hence $C'_i$'s
cover $C'$ by the valuative criterion of properness and surjectivity of the projection $\gtC\to C'$. Consider
the curve $(C_i,D_i)$ with the modification $(C'_i,D'_i)$. Applying to them first Proposition \ref{cofinprop}
and then Lemma \ref{divlem} we can construct a semi-stable modification of $(C'_i,D'_i)$. Then Proposition
\ref{blowdownprop} implies that each $(C'_i,D'_i)$ admits a stable modification
$f_i:(C''_i,D''_i)\to(C'_i,D'_i)$. In this case, stable modification is unique by Corollary \ref{sepcloscor}
(with $C_i$ being a ground strictly semi-stable curve), hence $f_i$'s agree over the intersections $C'_i\cap
C'_j$, and so they glue to a stable modification $(C'',D'')\to(C',D')$. Thus, $(C'',D'')$ is a semi-stable
modification of $(C,D)$, and it remains to use Proposition \ref{blowdownprop} once again to construct a stable
blow down of $(C'',D'')\to(C,D)$.

Finally, let us prove Theorem \ref{minstabtheor} in general. Given a stable modification $(C_\st,D_\st)$ and a
semi-stable modification $(C',D')$ of $(C,D)$, find a modification $(C'',D'')$ which dominates them. By Theorem
\ref{stabmodtheor} there exists a stable modification $(\oC,\oD)$ of $(C'',D'')$ which obviously dominates both
$(C',D')$ and $(C_\st,D_\st)$. Hence $(C',D')$ dominates $(C_\st,D_\st)$ by Corollary \ref{sepcloscor}, and we
are done.
\end{proof}

We will not need the following aside remark, so we omit a detailed proof of its assertion.

\begin{rem}
It follows from Proposition \ref{sepclosprop} that if $\bfP$ denotes semi-stability then for a multipointed
relative $S$-curve $(C,D)$ with a $\bfP$ $\eta$-fiber the family of its $\bfP$ modifications is cofinal in the
family of all its modifications. (Actually, we established particular cases of this result while proving the
Proposition.) Using pushout technique of \S\ref{smfibsec} one can show that a similar cofinality result holds
also for strict semi-stability and for ordinary relative curves, where we say that $(C,D)$ is {\em ordinary}
over $S$ if the geometric fibers $(C_\os,D_\os)$ are ordinary in the sense of Appendix \ref{B}. We will not make
any use of ordinary curves, but it is described in Appendix \ref{A3} how they appear in some proofs of the
stable reduction theorem.
\end{rem}

\section{Proof of the main results} \label{mainsec}
In this section we only assume that $S$ is integral qcqs and with generic point $\eta$. Also, $K=k(\eta)$.

\begin{lemsect}\label{openloclem}
Let $(C',D')\to (C,D)$ be a modification of multipointed $S$-curves then:

(i) the set of points $x\in C'$ at which $(C',D')$ is semi-stable is open;

(ii) the set of points $s\in S$ for which $(C'_s,D'_s)$ is semi-stable and $(C'_\os,D'_\os)$ has no exceptional
components for a geometric point $\os$ lying over $s$ is constructible.
\end{lemsect}
\begin{proof}
By approximation we can assume that $S$ is of finite type over $\bfZ$. It is well known that the semi-stable
locus of $C'\to S$ is open (e.g. use the local description from \cite[2.23]{dJ}). It follows that the
semi-stable locus of $(C',D')$ is also open because \'etaleness of $D'\to S$ and disjointness of $D'$ from
$(C'/S)_\sing$ are open conditions. This proves (i), and we also obtain that $(C',D')$ is semi-stable over a
constructible set whose complement is the image of the not semi-stable locus of $(C',D')$.

It suffices now to prove that exceptional components show up in precisely the geometric fibers over a
constructible subset $T\subset S$ when $(C',D')$ is semi-stable. Let $Z$ be the union of all irreducible
components in the fibers $C'_s$ that are contracted in $C$ to a point. Thus, $Z$ has proper $S$-fibers and is
closed in $C'$ by \cite[$\rm IV_3$, 13.1.3]{ega}. To each geometric fiber $Z_\os$ we associate the combinatorial
data consisting of its incidence graph -- vertex per generic point and edge per self-intersection, and also we
provide each vertex with the weight equal to the arithmetic genus of its irreducible component. We claim that
the combinatorial data of the geometric fiber over a point $s\in S$ is a locally constructible function of $S$.
For the topological data without weights this follows from the results of \cite[$\rm IV_3$, \S9.7]{ega}, see
loc.cit. 9.7.9 and 9.7.12; and for the genera this follows from existence of a stratification of $S$ (by reduced
locally closed subschemes) which flattens the morphism $Z\to S$ and from the semi-continuity theorem \cite[$\rm
III_2$, 7.6.9]{ega} (to prove that flattening exists use that the flat locus is open by \cite[$\rm IV_3$,
11.1.1]{ega}). Since each exceptional component is an irreducible component in some $Z_\os$ which satisfies
obvious combinatorial properties, it follows that $T$ is constructible.
\end{proof}

Recall that the Riemann-Zariski space $\gtS=\RZ_{K^s}(S)$ is homeomorphic to the projective limit of all
generically \'etale alterations of $S$, and a point $\gtx\in\gtS$ is defined by a valuation ring
$\calO_{\gtS,\gtx}$ of $K^s$ and a morphism $\phi_\gtx:\gtS_\gtx=\Spec(\calO_{\gtS,\gtx})\to S$ which agrees
with $\Spec(K^s)\to S$. By $C_\gtx=(C_\gtx,D_\gtx)$ we will denote the multipointed $\gtS_\gtx$-curve
$(C,D)\times_S\gtS_\gtx$. We start with an analog of Proposition \ref{schmodprop}.

\begin{propsect}
\label{valcrit} Let $S$ be an integral noetherian scheme with generic point $\eta=\Spec(K)$, and let $X\to C$
and $Y\to C$ be $\eta$-modifications, where $C,X,Y$ are reduced flat $S$-schemes of finite type. Assume that $Y$
is $\eta$-normal, and for any $\gtx\in\gtS=\RZ_K(S)$ the modification $Y_\gtx\to C_\gtx$ factors through
$X_\gtx$. Then $Y$ dominates $X$.
\end{propsect}
\begin{proof}
By Lemma \ref{approxlem} and approximation, for any point $\gtx\in\gtS$ there exists a quasi-modification
$S_\gtx\to S$ such that $Y_\gtx=Y\times_SS_\gtx$ dominates $X_\gtx=X\times_SS_\gtx$, i.e. the modification
$Y_\gtx\to C_\gtx$ factors through $X_\gtx$. Since $\RZ_K(S)$ is quasi-compact and $\RZ_K(S_\gtx)\into\RZ_K(S)$
is a neighborhood of $\gtx$ by Corollary \ref{openembcor}, we can find a finite set of quasi-modifications
$S_i\to S$ such that $\RZ_K(S)=\cup_i\RZ_K(S_i)$ and $Y\times_SS_i\to C\times_SS_i$ factors through
$X\times_SS_i$. By Proposition \ref{oetamodprop}, replacing $S_i$ with their modifications we can achieve that
$S_i$ are open subschemes of a modification $S'$ of $S$. Set $C'=C\times_SS'$, $X'=X\times_SS'$ and
$Y'=Y\times_SS'$, then $Y'\to C'$ factors through $X'$.

We want now to pull down the $\eta$-modification $f':Y'\to X'$ to an $\eta$-modification $Y\to X$. For an
$S$-scheme $T$ let $(T/S)_0\subset T$ denote the set of points closed in the $S$-fibers of $T$. Choose a closed
point $y\in Y$ and let $s$ be its image in $S$, so $y$ is closed in the fiber $Y_s$. Since $Y$ is $\eta$-normal,
$Y'\to Y$ is its own Stein factorization and so the preimage $Y'_y$ of $y$ in $Y'$ is connected. Clearly
$(Y'/S')_0$ contains the preimage of $(Y/S)_0$ under the projection $Y'\to Y$, and similarly for $X'$ and $X$.
So $Y'_y\subset(Y'/S')_0$, hence $f'(Y'_y)\subset (X'/S')_0$, yet it is a closed set in $X_s\times S'_s$, so it
must be quasi-finite over $S'_s$. Thus its image in $X_s$ must be $s$-quasi-finite. By connectedness of $Y'_y$
this image must be a single point $x\in X_s$. Obviously, $\calO_{X,x}\subset\calO_{Y'}(Y'_y)$, where
$\calO_{Y'}(Y'_y)=\injlim_U\calO_{Y'}(U)$ for $U$ running through all open neighborhoods of $Y'_y$, and also
$\calO_{Y,y}=\calO_{Y'}(Y'_y)$ by Stein factorization. Thus, we have proved that for any closed point $y\in Y$
there exists a point $x\in X$ whose local ring is contained in the local ring of $y$. It follows easily that the
modification $Y\to C$ factors through $X$ (and $y$ is taken to $x$).
\end{proof}

\begin{proof}[Proof of Theorem \ref{minstabtheor}.]
First we reduce to the case of $S$ of finite type over $\bfZ$. Indeed, $(C,D)$, $(C_\st,D_\st)$ and $(C',D')$
are induced from multipointed curves $(\oC,\oD)$, $(\oC_\st,\oD_\st)$ and $(\oC',\oD')$ over a scheme $\oS$ of
finite type over $\bfZ$, and by Lemma \ref{openloclem}(ii) and Remark \ref{prem}(ii) one can also achieve that
$(\oC_\st,\oD_\st)$ is stable and $(\oC',\oD')$ is semi-stable. Now in order to show that $C'\to C$ factors
through $C_\st\to C$ it is enough to show that $\oC'\to\oC$ factors through $\oC_\st\to\oC$. So, we can replace
$S,C,D$, etc., with $\oS,\oC,\oD$, etc., achieving that $S$ is of finite type over $\bfZ$. In particular, any
valuation in $\gtS$ is of finite height by Abhyankar inequality \ref{Abhlem}. Since $C'$ is semi-stable over
$S$, it is $\eta$-normal by Proposition \ref{redfibprop}. For any $\gtx\in\gtS$, the modification
$((C_\st)_\gtx,(D_\st)_\gtx)\to(C_\gtx,D_\gtx)$ is stable and the $\gtS_\gtx$-curve $(C'_\gtx,D'_\gtx)$ is
semi-stable. By Corollary \ref{sepcloscor}, $C'_\gtx$ dominates $(C_\st)_\gtx$, hence $C'$ dominates $C_\st$ by
Proposition \ref{valcrit}.
\end{proof}

Next we deduce the Corollaries from the Introduction. To prove Corollary \ref{ccc} we note that if
$(C'_\st,D'_\st)$ is another stable modification then $C_\st$ and $C'_\st$ dominate each other. Hence
$C_\st\toisom C'_\st$, and we get (i). The semi-stable locus $U$ of $(C,D)\to S$ is an open subscheme of $C$ by
Lemma \ref{openloclem}(i), so $f^{-1}(U)$ is a stable modification of $U$ and $U_\st\toisom U$ by minimality of
the stable modification.

Corollary \ref{minstabcor} follows from Theorem \ref{minstabtheor} as follows. The normalization $S''$ of $S$ is
the projective limit of finite modifications. Hence by approximation, the domination morphism $\oC\times_S
S''\to C_\st\times_S S''$ can be defined already over a finite modification $S'$ of $S$.

\begin{remsect}
\label{uniqrem} It is not clear if the above modification $S'\to S$ is necessary at all. Simple examples, see
\cite[3.17]{AO}, show that extension of a stable curve $C_\eta$ to a stable $S$-curve $C\to S$ can be not unique
when $S$ is not normal. However, stable modification is a more subtle creature (for example, it exists when
$C\to S$ is not proper). The author does not know examples where stable modification is not unique.
\end{remsect}

\begin{proof}[Proof of Theorem \ref{stabmodtheor}.]
Since $(C,D)$ is induced from a multipointed curve over some $S_0$ of finite type over $\bfZ$, it suffices to
build a stable modification over $S_0$. So, it suffices to deal with $S$ of finite type over $\bfZ$, and we will
assume that this is the case. In particular, each valuation from $\gtS=\RZ_{K^s}(S_0)$ is now of finite height.
For any point $\gtx\in\gtS$, the relative multipointed curve $(C_\gtx,D_\gtx)\to\gtS_\gtx$ possesses a stable
modification $(C'_\gtx,D'_\gtx)$ by Proposition \ref{sepclosprop}. By Lemma \ref{approxlem} and approximation,
there exist a generically \'etale quasi-alteration $S'\to S$ and an $\eta$-modification $(C',D')\to(C,D)\times_S
S'$ such that $\gtx$ is centered on a point $x\in S'$ and $(C'_\gtx,D'_\gtx)\to\gtS_\gtx$ is the base change of
$(C',D')\to S'$. Moreover, by Lemma \ref{openloclem}(ii) and Remark \ref{prem}(ii) we can achieve that $(C',D')$
is a stable modification of $(C,D)\times_SS'$.

Now we can act similarly to the proof of Proposition \ref{valcrit}. Since $\RZ_{K^s}(S')\into\gtS$ is a
neighborhood of $\gtx$ and $\gtS$ is quasi-compact, we can find a finite set of generically \'etale
quasi-alterations $S_i\to S$ such that $\gtS=\cup_i\RZ_{K^s}(S_i)$ and the multipointed curves $(C,D)\times_S
S_i$ admit stable modifications $(C_i,D_i)$. Replacing $S_i$ with their generically \'etale alterations, we can
achieve that $S_i$ are open subschemes of a generically \'etale alteration $S'$ of $S$. Let $\oS$ be the
normalization of $S'$ and $\oS_i$ be the preimages of $S_i$ in $\oS$ then the curves $(C,D)\times_S\oS_i$ admit
stable modifications $(\oC_i,\oD_i)$ which agree over the intersections $\oS_i\cap\oS_j$ by Corollary \ref{ccc}
(i). So, they glue to a stable modification of the $\oS$-curve $(C,D)\times_S\oS$.
\end{proof}

Finally, if $S'$ is a normal alteration of $S$ and a stable modification $(C'_\st,D'_\st)\to (C,D)\times_S S'$
exists then it is unique, whence Corollary \ref{equivcor}.

\section{Uniformization of one-dimensional analytic fields}\label{ansec}
Throughout \S\ref{ansec} we fix an analytic ground field $k$ and set $p=\cha(\tilk)$. All analytic fields are
understood to be $k$-fields. Our general goal in \S\ref{ansec} is to establish analytic one-dimensional local
uniformization. In particular, we will fulfil our earlier promise to prove Theorem \ref{fieldunif1} (in the very
end of \S\ref{ansec}). We will see that the main difficulty is in treating immediate extensions, and we propose
a way to control them in \S\ref{immsec}.

\subsection{Immediate extensions of degree $p$} \label{immsec}
Our aim is to gain some control on immediate algebraic extensions of an analytic field $K$. For example, we
would like to obtain a criterion when $K$ is {\em stable}. Recall that stability means that any finite extension
$L/K$ is Cartesian, i.e. $e_{L/K}f_{L/K}=[L:K]$ or $d_{L/K}=1$ (see also \cite[\S3.6]{BGR}). In particular, if
$K$ is a DVR or $p=\cha(\tilK)=0$ then $K$ is stable. It easily follows from simple ramification theory that $K$
is not stable if and only if there exists a finite extension $L/K$ such that $L$ admits an immediate extension
of degree $p$ (it may happen, though, that $K$ itself does not admit finite immediate extensions but $K$ is not
stable). We will assume that $p>0$ until the end of \S\ref{immsec} (note, however, that many statements become
trivial but make sense if one takes $p$ to be the exponential characteristic).

If $L/K$ is Cartesian wildly ramified of degree $p$ then there exists $b\in K$ such that $\min|b+K^p|=|b|$ and
$\inf|b+L^p|<|b|$. Indeed, either $e_{L/K}=p$ and then we take $b$ such that
$|b|^{1/p}\in|L^\times|\setminus|K^\times|$, or $\tilL/\tilK$ is inseparable and then we take $b\in K$ such that
$|b|=1$ and $(-\tilb)^{1/p}\in\tilL\setminus\tilK$. This fact has the following analog.

\begin{prop}
\label{immprop0} Let $L/K$ be an immediate extension of degree $p$. Then there exist elements $a,b\in K$ and
$\alpha\in L$ such that the infimum $s=\inf_{x\in K}|x^p-ax+b|$ is not achieved on $K$, $|pb|<s$ and
$|\alpha^p-a\alpha+b|<s$. In addition, one can achieve that either $|a|=s^{\frac{p-1}{p}}$, or $a=0$.
\end{prop}

The role of the oddly looking condition $|pb|<s$ will be seen later; it allows to replace $b$ with $c^p-ac+b$
for any $c\in K$ such that $|pc^p|<s$. Before proving the Proposition we prefer to reformulate it in a less
intuitive form that is more convenient for applications.

Let $K$ be an analytic field of positive characteristic $p$. Given an element $a\in K$, the set
$S_a=S_a(K)=\{c^p-ac|c\in K\}$ is an additive group whose cosets $b+S_a$ contain some information about $K$. If
$K$ is of mixed characteristic then we have a group structure only approximately. One can remedy the problem by
switching to characteristic $p$ and studying the ring $\Kcirc/p\Kcirc$ and its modules $x\Kcirc/px\Kcirc$. We
prefer another approach which is nearly equivalent. Given an element $a\in K$ and a positive number $s$, the set
$S_{a,s}(K)=\{c^p-ac+d|\ c,d\in K,|pc^p|<s,|d|<s\}$ is an additive subgroup of $K$ because
$|(c_1+c_2)^p-c_1^p-c_2^p|<s$ as soon as $|pc_i|^p<s$. We will study its cosets $b+S_{a,s}(K)$, and we say that
a coset is {\em trivial} or {\em split} if it contains zero. Note that because of the $d$-term a coset is split
if and only if it contains an element $x$ with $|x|<s$. We say that a non-split coset $b+S_{a,s}(K)$ is {\em
critical} if $\inf|b+S_{a,s}(K)|=s>0$, that is the coset contains elements of absolute value arbitrary close to
the lowest possible bound. If in addition either $a=0$, or $a=1$ and $s=1$ then the coset is called {\em
special}. Note that for any non-split coset we can ignore the $d$-term when finding the infimum and so a coset
is critical if and only if $s=\inf_{c\in K,|pc^p|<s}|c^p-ac+b|$ (with $s>0$).

\begin{lem} \label{212lem}
Assume that the valuation on $K$ is not discrete.

(i) A coset $b+S_{a,s}(K)$ is critical if and only if $|pb|<s=\inf_{c\in K}|c^p-ac+b|$.

(ii) Any critical coset satisfies $|a|\le s^{\frac{p-1}{p}}$.
\end{lem}
\begin{proof}
First, we claim that given an element $b'$ of a critical coset $b+S_{a,s}(K)$ and an element $c\in K$ such that
$|pc^p|<s$ and $|c^p-ac+b'|<|b'|$ one necessarily has that $|b'|>|ac|$ and, in particular, $|b'|=|c^p|$. Without
loss of generality $b=b'$. Then to justify our claim it is enough to show that if, to the contrary, $a,b,c\in K$
satisfy $|c^p-ac+b|<|b|$ and $|b|\le |ac|$ then there exists a root $c_0$ of $f(T)=T^p-aT+b$ with $|c-c_0|<|c|$.
Indeed, $|pc_0^p|=|pc^p|\le s$ and so $b+S_{a,s}(K)$ is split, contradicting our assumption. Now let us prove
that $c_0$ exists. Note that we can re-scale $a,b,c$ by replacing them with $c'=uc$, $a'=u^{p-1}a$ and $b'=u^pb$
for any $u\in K^\times$ because $|c'^p-a'c'+b'|<|b'|$, $|b'|\le |a'c'|$ and any root $c'_0$ of $T^p-a'T+b'$
corresponds to a root $c_0=c'_0/u$ of $f(T)$. In particular, taking $u=c^{-1}$ we can make $c'=1$. To simplify
notation we will denote the new triple as $a,b,c$. Since $|1-a+b|<|b|$ and $|b|\le |a|$ one of the following is
true: (a) $|a|=1\ge |b|$, (b) $|a|=|b|>1$. In case (a) the residue polynomial $\tilf(T)=T^p+\tila T+\tilb$ is
separable with root $1$. In particular, it has a root $c_0$ by Hensel's lemma. In case (b) we consider the
non-monic polynomial $g(T)=-a^{-1}f(T)$. The reduction $\tilg(T)=T-\wt{(b/a)}$ has a simple root, hence $g(T)$
has a root by the non-monic version of Hensel's lemma (the usual proof with lifting a root works fine).

\begin{rem}
A more elegant proof avoiding re-scalings can be given by use of graded Hensel's lemma. In such case one
considers the graded reduction $\tilk\to\tilk_\gr=\oplus_{r\in\bfR_+^\times}\{x\in k|\ |x|\le r\}/\{x\in k|\
|x|<r\}$ as defined in \cite[\S1]{Tem1.5} and lifts homogeneous roots of the graded reductions of $f(T)$ (one
reduction per each value of $|T|$). The graded version of Hensel's lemma is proved very similarly to its
classical version, and we refer to \cite[1.9]{Duc} for details
\end{rem}

Now let us prove the Lemma. We start with (ii). Suppose to the contrary that $|a|>s^{\frac{p-1}{p}}$. Replacing
$b$ with another element of the coset we can assume that $s<|b|<|a|^{\frac{p}{p-1}}$. Then there exists $c$ with
$|c^p-ac+b|<|b|$ and we proved above that $|ac|<|b|=|c^p|$. In particular, $|a^p|<|b|^{p-1}$ and we obtain a
contradiction.

Next we prove the direct implication in (i). Assume that a coset $b+S_{a,s}(K)$ is critical. If $|pb|\ge s$ then
$p\neq 0$ and $|b|>s$. So, the norm of $b$ can be decreased by adding an element of the form $c^p-ac$ with
$|pc^p|<s$, and then $|c^p|<|p|^{-1}s\le|b|$. In particular, we must have $|ac|=|b|$ contradicting the proved
above claim that $|b|>|ac|$. This proves that $|pb|<s$. If $\inf_{c\in K}|c^p-ac+b|<s$ then there exists $c\in
K$ with $|pc^p|\ge s$ and $|c^p-ac+b|<s$. We proved in (ii) that $|a|\le s^{\frac{p-1}{p}}$ and it follows that
for any $c$ with $|c^p|>s$ we have that $|c^p|>|ac|$. In particular, $|c^p|=|b|$ and therefore $|pc^p|=|pb|<s$.
This contradiction establishes the second claim of the direct implication.

It remains to prove the converse implication in (i), so assume that $|pb|<s=\inf_{c\in K}|c^p-ac+b|$. We should
only prove that $s'=\inf_{c\in K,|pc^p|<s}|c^p-ac+b|$ equals to $s$. Assume to the contrary that $s'>s$, and let
$c\in K$ be such that $|pc^p|\ge s$ and $|c^p-ac+b|<s'\le |b|$. Certainly, $|c^p-ac|=|b|$, hence either
$|c^p|=|ac|>|b|$ or $|c^p|,|ac|\le|b|$ with at least one an equality. In the first case,
$$|a|=|c|^{p-1}>|b|^{(p-1)/p}\ge s'^{(p-1)/p}>s^{(p-1)/p}$$ which gives a contradiction with (ii). In the second
case, we cannot have that $|c^p|=|b|$ because $|pc^p|\ge s>|pb|$. Thus, the only remaining possibility is that
$|c^p|<|b|=|ac|$. But then $|a|>|c|^{p-1}$, and hence $s=|a|^{p/(p-1)}>|c|^p>|pc^p|$, contrary to our
assumption.
\end{proof}

Now we can give a more precise version of Proposition \ref{immprop0}. It is easily checked that the old
statement follows from the new one.

\begin{prop}
\label{immprop} Let $L/K$ be a wildly ramified extension of degree $p$.

(i) There exists a critical coset $b+S_{a,s}(K)$ which splits over $L$ (i.e. $b+S_{a,s}(L)$ is split) and such
that one of the following possibilities holds:

(a) $a=0$;

(b) $|a|=s^{\frac{p-1}{p}}$ and $s\notin|b+S_{a,s}(K)|$.

(ii) In case (ia) the coset is special by the definition and in case (ib) the wildly ramified extension
$L(a^{\frac{1}{p-1}})/K(a^{\frac{1}{p-1}})$ (obtained from $L/K$ by the moderately ramified base field extension
$K(a^{\frac{1}{p-1}})/K$) admits a special coset satisfying the conditions of (ib).

(iii) In case (ib) $L\toisom K[T]/(T^p-aT+b)$.

(iv) The invariants $d=d_{L/k}$, $e=e_{L/K}$ and $f=f_{L/K}$ satisfy $def=p$ and in the situation of (i) they
can be found as follows:

(a) The extension $L/K$ is immediate (i.e. $d=p$) if and only if $s\notin|b+S_{a,s}(K)|$. In particular, it is
the case in (ib).

(b) $e=p$ if and only if $s\in|b+S_{a,s}(K)|$ and $s\notin|K^\times|^p$. In this case $s^{1/p}$ generates
$|L^\times|$ over $|K^\times|$.

(c) $f=p$ if and only if $s\in|b+S_{a,s}(K)|$ and $s\in|K^\times|^p$. In this case for any $y\in b+S_{a,s}(K)$
and $c\in K$ such that $|y|=s$ and $|c|=s^{1/p}$ the element $(\wt{y/c^p})^{1/p}$ generates the purely
inseparable extension $\tilL/\tilK$.
\end{prop}

Before proving the Proposition we will make few aside remarks that will not be used later.

\begin{rem}
(i) The Cartesian cases of the Proposition are classical and very easy. Although we will not need them in the
sequel, they are included for the sake of comparison and completeness.

(ii) In the immediate case our method is to start with any $\alp\in L\setminus K$ and to partially orthogonalize
it with respect to $K$ by subtracting elements $c\in K$. When $|\alp-c|$ is close enough to its infimum, the
minimal polynomial of $\alp$ is essentially affected only by the terms $T^p$, $aT$ and $b$, so it can be used to
construct a critical coset over $K$ that is split over $L$. As was explained to me by F.-V. Kuhlmann, this
method was used already by Kaplansky, though the formulation of the Proposition seems to be new.

(iii) The Proposition admits various generalizations and complements which we plan to discuss elsewhere. Here we
only remark that using the method of our proof one can easily show that the case (ib) (which has no Cartesian
analog) happens exactly when the $L/K$ is almost unramified in the sense of Faltings (or $\Lcirc/\Kcirc$ is
almost \'etale in the sense of \cite[Ch. 6]{GR}). Other equivalent conditions for case (ib) are that $L/K$ has
zero different, or $\Omega^1_{\Lcirc/\Kcirc}=0$ (the latter a priori could be stronger than vanishing of the
different).
\end{rem}

\begin{proof}
If $L/K$ is Cartesian then there exists non-zero $\alp\in L$ orthogonal to $K$, i.e. $|\alp|\le |\alp-c|$ for
any $c\in K$. Obviously, $\alp^p$ is "orthogonal" to the set $K^p$ in the sense that $|\alp^p|\le|\alp^p-c^p|$
for $c\in K$. On the other hand, $\alp^p$ is not orthogonal to $K$ because either $e_{L/K}=p$ and then
$|\alp^p|\in|K^\times|$ and $\tilL=\tilK$, or $f_{L/K}=p$ and then $|\alp^p|=|c^p|$ for $c\in K$, and
$\wt{\alp^p/c^p}\in\tilK$. All in all, $|\alp^p-b|<|\alp^p|=s$ for some $b\in K$ and then $b$ is orthogonal to
$K^p$. In particular, $b+S_{0,s}(K)$ is a special coset split by $L$. This proves (i) for Cartesian extensions.
Also, we note that $def=[L:K]=p$. Hence there are three cases which exclude each other: $d=p$, $e=p$ and $f=p$,
and therefore it is enough to prove only converse implications in (iva), (ivb) and (ivc). The Cartesian cases
(b) and (c) are done similarly to the above argument, so we skip the details.

Next and main, we are going to establish the remaining case of (i), so assume that $L/K$ is immediate. If $L/K$
is inseparable then set $a=0$ and choose $b\in K$ such that $L=K(b^{1/p})$. Since the extension is immediate the
infimum $r=\inf_{c\in K}|c+b^{1/p}|$ is not achieved. Also, $r>0$ by completeness of $K$. Since $\cha(K)=p$ we
have that $s:=\inf_{c\in K}|c^p+b|=r^p>0$, hence $b+S_{0,s}(K)$ is a special coset. Finally, this coset is split
by $L$ because $-b=(-b^{1/p})^p\in S_{0,s}(L)$.

In the sequel, we assume that $L/K$ is separable. Choose an element $\alpha\in L\setminus K$ and set
$r=\inf_{x\in K}|\alpha-x|$ and $r_0=|\alpha|$. As earlier, the infimum $r$ is not attained on $K$ and $r>0$.
Let $f(T)=T^p+\sum_{i=1}^{p-1}a_iT^i+b$ be the minimal monic polynomial of $\alpha$. Recall that the {\em closed
disc} $E_{\hatK^a}(\alpha,r)$ of radius $r$ and with center at $\alpha$ in the Berkovich affine line
$\bfA^1_{\hatK^a}$ with a fixed coordinate $T$ is the affinoid domain given by the condition $|T-\alpha|\le r$
(see \cite[1.4.4]{Ber1}). By the {\em $K$-disc} $E$ of radius $r$ and with center at $\alpha$ we mean the image
of $E_{\hatK^a}(\alpha,r)$ under the morphism $\bfA^1_{\hatK^a}\to\bfA^1_K$, so the preimage $E'$ of $E$ in
$\bfA^1_{\hatK^a}$ is the union of discs of radii $r$ with centers at the conjugates of $\alp$ (i.e. the roots
of $f(T)$). It is well known that for any polynomial $f(T)$ the Weierstrass domain
$W_s=\bfA^1_{\hatK^a}\{s^{-1}f(T)\}$ is the union of closed discs with centers at the roots of $f$, and by the
symmetry in our case all these discs are of equal radius which monotonically depends on $s$. Thus, $E'=W_s$ for
a certain $s$, and hence $E=\bfA^1_k\{s^{-1}f(T)\}$. (Note that $K$-discs are also introduced in
\cite[3.6]{Ber2}, but the radius there is defined to be $s^{1/\deg(f(T)}$, and in general it does not equal to
the radius in our sense.) We say that a $K$-disc is {\em split} (or $K$-split) if it has a $K$-point. Note that
by our assumption on $K$ and $\alpha$ the disc $E=E_K(\alpha,r)$ is not split, but any larger disc
$E_K(\alpha,r+\ve)$ is split.

For any point $x\in E$, the infimum $\inf_{z\in K}|(T-z)(x)|$ equals to $r$ and is not achieved. If $x$ is
Zariski closed then we obtain that the finite extension $\calH(x)/K$ is not Cartesian, in particular,
$[\calH(x):K]\ge p$. On the other side, the derivative $f'(T)$ does not vanish identically and is of degree
smaller than $p$, hence it is invertible on $E$. It follows that $f'$ is invertible on a disc
$E_0=E(c,r_1)\supset E$ for $c\in K$ and $r_1>r$. Replacing $T$ with $T-c$, $\alpha$ with $\alpha-c$ and $r_0$
with $|\alpha-c|$, we can assume that $E_0=E(0,r_0)$ with $r_0>r$. By the same reasoning we can assume that if
$1\le i<p$ and the higher derivative $f^{(i)}$ is not identically zero then $f^{(i)}$ is invertible on $E_0$.

The condition on the derivatives of $f$ implies the following important property: if $c\in K$ satisfies $|c|\le
r_0$ and $f(T)=\sum_{i=0}^p a'_i(T-c)^i$ then $|a_i-a'_i|<|a_i|$ for $i<p$. Indeed, for a $K$-splits disc
$\calM(k\{s^{-1}T\})$ a function $f(T)=\sum_{i=0}^\infty b_iT^i$ is invertible if and only if $|b_0|>|b_i|s^i$
for $i>0$, and hence $|b_0-f(x)|<|b_0|$ for any point $x$ in that disc. In our situation, $a'_i=f^{(i)}(c)/i!$
and $a_i$ is the constant coefficient of the polynomial $f^{(i)}(T)/i!$ which is invertible on the $K$-split
disc $E_0$.

Thus, the values of $|a_i|$ are fixed when we shrink $E_0$ and change the coordinate accordingly. The condition
on $f'$ implies that $|ia_i|r_0^i=|a_i|r_0^i<|a_1|r_0$ for $1<i<p$. Since $f(T)$ is irreducible, we have that
$|a_1\alpha|\le|\alpha^p|$ by \cite[3.2.4/3]{BGR}, hence $|a_1|\le r_0^{p-1}$. Shrinking $E_0$ we can make $r_0$
arbitrary close to $r$ while $|a_i|$'s remain fixed. Therefore, we can assume that $|a_1|\le r^{p-1}$ and
$|a_i|<\frac{|a_1|}{r_0^{i-1}}<r^{p-i}$ for $i>1$. By additional shrinking of $E_0$ we can, furthermore, achieve
that $|a_i|<\frac{r^p}{r_0^i}$ for $i>1$ and $|p|r_0^p<r^p$.

Set $a=-a_1$, then $g(T)=T^p-aT+b$ is obtained from $f(T)$ by removing the $a_iT^i$ terms for $1<i<p$. We will
show that $a,b$ and $s:=r^p$ are as required. Notice that $\alpha$ is a root of $f$ and
$|a_i\alpha^i|<\frac{r^p}{r_0^i}r_0^i=s$ for any $1<i<p$, hence $|g(\alpha)|<s$. It follows that $-b\in
S_{a,s}(L)$. Moreover, if $|a_1|<r^{p-1}$ then we remove the $-aT$ term using the same argument, achieving that
either (a) or (b) in (i) is satisfied. Notice that $|f(T)-g(T)|<s$ on $E$, $\max_{x\in E}|f(x)|\ge |b|=r_0^p>s$
and $f$ has a root in $E$. It follows that $|g(T)|$ is not a constant on $E$ and, therefore, it is not constant
on the disc $E':=E_{\hatK^a}(\alpha,r)$ (which is a connected component of the preimage of $E$ in
$\bfA^1_{\hatK^a}$). It follows that $g(T)$ has a zero on $E'$, i.e. it has a root $\beta$ with
$|\alpha-\beta|\le r$. The latter inequality and the definition of $r$ imply that $\inf|\beta-K|=r$ and the
infimum is not attained. In particular, $K(\beta)/K$ is not Cartesian, hence its degree is divided by $p$. Since
$\beta$ is annihilated by the polynomial $g(T)$ of degree $p$, we obtain that $g(T)$ is the minimal polynomial
of $\beta$ and, in particular, $g(T)$ is irreducible. The distance between $\beta_1=\beta$ and other roots
$\beta_2\.\beta_p$ of $g(T)$ does not exceed $\inf|\beta-K|=r$ by Krasner's lemma, hence for any $c\in K$ the
numbers $|c-\beta_i|$ are equal for $1\le i\le p$. In particular, $|g(c)|=|c-\beta|^p$, and we obtain that
$\inf_{c\in K}|g(c)|=(\inf_{c\in K}|c-\beta|)^p=r^p=s$, and the infimum is not achieved. Finally, the inequality
$|pb|=|p|r_0^p<s$ implies that $b+S_{a,s}(K)$ is critical by Lemma \ref{212lem}(i). This finishes the proof of
(i).

The claim of (ii) for $a=0$ follows from the definition, and for $a\neq 0$ set $b'=b/a^{\frac{p}{p-1}}$ and
$K'=K(a^{\frac{1}{p-1}})$ and observe that the coset $b'+S_{1,1}(K')$ is a required special coset. In (iii) the
inequality $|g(\alpha)|<s$ implies that $\inf|b+S_{a,s}(L)|<s=|a|^{\frac{p}{p-1}}$. As we saw in the proof of
Lemma \ref{212lem}, it then follows from Hensel's lemma that $0\in b+S_{a,s}(L)$, i.e. $g(T)$ has a root in $L$.
Therefore $K(\beta)\toisom L$ as claimed. It remains to prove (iv) and we earlier reduced this to proving that
if the infimum $s$ is not attained on the coset $b+S_{a,s}(K)$ then the extension is immediate. Set
$f(T)=T^p-aT+b$, then the disc $E:=\bfA^1_K\{s^{-1}f(T)\}$ is a non-split $K$-disc such that any larger disc
$\bfA^1_K\{(s+\ve)^{-1}f(T)\}$ is split. Therefore, for any element $\beta\in K^a$ which belongs to $E$ (i.e.
satisfies $|f(\beta)|\le s$) the infimum $\inf_{c\in K}|c-\beta|$ equals to the radius of $E$ and is not
achieved. Since the coset is split over $L$, the disc $E$ contains an $L$-point $\beta$. Since $\inf|\beta-K|$
is not achieved the extension $L/K$ is not Cartesian, and so $d_{L/K}=p$ as claimed.
\end{proof}

\begin{cor}
\label{immcor} An analytic field $K$ is stable if and only if either $p=\cha(\tilK)$ is zero or for any finite
Cartesian extension $K'/K$ any special coset $b+S_{a,s}(K')$ contains an element of minimal absolute value.
\end{cor}
\begin{proof}
Since any finite extension is Cartesian when $p=0$, we have only to deal with the case of non-zero $p$. Set
$f(T)=T^p-aT+b\in K'[T]$ and assume that $s=\inf_{z\in K'} |f(z)|$ is not achieved. Then, as we observed in the
end of proof of Proposition \ref{immprop}, $E:=\bfA^1_{K'}\{s^{-1}f(T)\}$ is a non-split $K'$-disc such that any
larger disc is split, and for any $\beta\in K^a$ contained in $E$ the extension $K'(\beta)/K'$ is not Cartesian.
In particular, $K$ is not stable then.

Conversely, assume that $K$ is not stable. We claim that some its Cartesian extension $F$ admits an immediate
extension of degree $p$. Indeed, if $L/K$ is not Cartesian then for sufficiently large moderately ramified
extension $K_1/K$ the extension $LK_1/K_1$ splits into a tower of $p$-extensions $K_1\subset K_2\subset\dots
\subset LK_1$ (we use that $\Gal(K^s/K^\mr)$ is a pro-$p$-group and hence any (maybe inseparable) extension of
$K^\mr$ splits into a tower of $p$-extensions). Then there exists $i$ such that $K_i/K$ is Cartesian and
$K_{i+1}/K_i$ is not Cartesian. Since the latter is of degree $p$, it is immediate and we can take $F=K_i$. By
Proposition \ref{immprop}(ii), a field $K'$ of the form $F(a^{\frac{1}{p-1}})$ possesses a special coset without
minimal element. It remains to note that $K'/F$ is Cartesian because $[K':F]<p$, and hence $K'/K$ is Cartesian.
\end{proof}

\subsection{Analytic fields topologically generated by an element}\label{anensec}
An analytic $k$-field $K$ is {\em topologically generated} by an element $T$ if $K$ coincides with the
topological closure $\ol{k(T)}$ of the subfield $k(T)$ in $K$. If $K$ is finite over a subfield of the form
$\ol{k(T)}$ and $T\notin\whka$ then we say that $K$ is {\em one-dimensional}. For example, if $x$ is a point on
a $k$-analytic curve $C$ and the preimages of $x$ in $C\wtimes_k\whka$ are not Zariski closed then the analytic
$k$-field $\calH(x)$ is one-dimensional. We claim that the sum of $F=F_{K/k}=\trdeg(\tilK/\tilk)$ and
$E=E_{K/k}=\dim_\bfQ\left((|K^\times|/|k^\times|)\otimes_\bfZ\bfQ\right)$ does not exceed one. Indeed, this
follows from Abhyankar's inequality (see Lemma \ref{Abhlem}) applied to $k(T)$ and the fact that the numbers
$E_{K/k}$ and $F_{K/k}$ are preserved by replacing $K$ with the completion or a finite extension. Thus,
similarly to \cite[1.4.4]{Ber1} we divide one-dimensional fields to three types as follows: $K$ is of type $2$
(resp. $3$, resp. $4$) if $E=0,F=1$ (resp. $E=1,F=0$, resp. $E=F=0$). (Type $1$ points from \cite{Ber1}
correspond to subfields of $\whka$.) Note that the type of a one-dimensional field is preserved by passing to a
finite extension. We say that $K$ is {\em $k$-split} if for any $T\in K$ we have that $\inf|T-k|=\inf|T-k^a|$,
where the right hand side makes sense due to uniqueness (up to an automorphism) of the isometric embedding
$\whka\into\hatK^a$. The next two sections are devoted to uniformization of one-dimensional analytic fields of
the following form.

\begin{assum}\label{assum}
Assume that $K$ is a one-dimensional analytic $k$-field and one of the following conditions is satisfied:

(i) $k=k^a$;

(ii) $K$ is of type $4$ and $k$-split, $k=k^\mr$, and either $p=0$, or $p>0$ and $\kcirc=p\kcirc+(\kcirc)^p$. If
the last condition is satisfied then we say that $k$ is {\em deeply ramified}.
\end{assum}

\begin{rem}\label{rem0}
Note that the assumption implies that $\tilk$ is algebraically closed and $|k^\times|$ is divisible. Indeed,
$\tilk$ is separably closed and $|k^\times|$ is divisible by any prime $l$ with $(l,p)=1$ because $k=k^\mr$. On
the other side, the deep ramification condition implies that $\tilk$ is perfect and $|k^\times|$ is
$p$-divisible.
\end{rem}

\begin{rem}\label{rem}
We will need only case (i) in this paper, but case (ii) does not require any extra-work and it will play a
central role in a subsequent work \cite{Tem3} on inseparable local uniformization. Our definition of deeply
ramified fields agrees with its analog in \cite[6.6.1]{GR} by \cite[6.6.6]{GR}. It is proved there that deeply
ramified valued fields can be characterized by many other equivalent properties. For example, if $k$ is not
discrete valued then two other equivalent properties are that $\Omega^1_{(k^s)^\circ/\kcirc}=0$, or any
separable algebraic extension of $k$ is almost unramified (i.e. has trivial different).
\end{rem}

In this section we always assume in addition to \ref{assum} that $p>0$ and $K$ is topologically generated by an
element, say $K=\ol{k(z)}$. Equivalently, $K\toisom\calH(x)$ for a not Zariski closed point
$x\in\bfA^1_k=\calM(k[z])$. It follows from the classification of points, see \cite[1.4.4]{Ber1} and
\cite[3.6]{Ber2} for details, that the type of $x$ is the type of $K$ as defined above. So, if $x$ is of type
$2$ or $3$ then $k=k^a$ by the assumption, and $x$ is the generic point of a disc of rational or irrational
radius, respectively (i.e. the radius is or is not contained in $|k^\times|=\sqrt{|k^\times|}$). Then
$K\toisom\wh{\Frac(k\{T\})}$ or $K\toisom k\{r^{-1}T,rT^{-1}\}$ for some $r\notin\sqrt{|k^\times|}$,
respectively, where the valuation used to complete $\Frac(k\{T\})$ is the extension of the Gauss (or spectral)
norm on $k\{T\}$. In the situation of \ref{assum}(ii), $x$ is the intersection of a decreasing sequence
$E_0\supsetneq E_1\supsetneq E_2\supsetneq\dots$ of closed discs in $\bfA^1_k$. If $E_i=E(\alp_i,r_i)$ then
$r:=\lim_{i\to\infty} r_i$ equals to $\inf|z-k^a|$ (where the absolute value is computed in $\wh{K^a}$), and the
infimum is not attained on $k^a$. If some $E_i$ is not split then $r_i\le\inf|\alp_i-k|$ and for each $j>i$ we
already have that $r_j<\inf|\alp_j-k|$. In particular, $|(z-\alp_j)(x)|\le r_j<\inf|\alp_j-k|$, and we obtain
that $K=\calH(x)$ is not $k$-split. The contradiction proves that all discs $E_i$ are $k$-split, hence we can
re-choose the centers so that $\alp_i\in k$. Now it is also clear that
$r=\lim_{i\to\infty}|z-\alpha_i|=\inf|z-k|$ and the infimum is not achieved since it is not attained even on
$k^a$.

\begin{prop}
\label{modramprop} Assume that $p>0$ and $K=\ol{k(z)}$ is as in \ref{assum}, and let $L$ be a finite moderately
ramified extension of $K$. If $K$ is of type $2$ or $3$ then any critical coset in $L$ contains an element of
minimal absolute value. If $K$ is of type $4$ then any special coset in $L$ contains either an element of $k$ or
a topological generator of $L$ over $k$.
\end{prop}

It is for the sake of simplicity that in the case of type $4$ fields we consider only special cosets. To prove
the Proposition (in the end of \S\ref{anensec}) will need a good explicit description of analytic $k$-fields
topologically generated by an element. Recall that a subset $B$ of a normed $k$-vector space $V$ is called
orthogonal (resp. orthonormal) Schauder basis if any element $v\in V$ admits a unique representation of the form
$v=\sum_{b\in B}a_b b$ and $\|v\|=\max_{b\in B}|a_b|\|b\|$ (resp. $\|v\|=\max_{b\in B}|a_b|$). It is easy to see
that any analytic field $\ol{k(z)}$ of type $2$ or $3$ admits a Schauder basis over $k$ (e.g. for
$K=\wh{\Frac(k\{T\})}$ one can take the union of the sets $B_\infty=\{T^i\}_{i\ge 0}$ and
$B_\tila=\{(T-a)^{-i}\}_{i\ge 1}$ where $\tila$ runs over $\tilk$ and $a$ is a fixed lifting of $\tila$ to $k$),
but we will need a Schauder basis of a special form.

\begin{prop}
\label{schaudprop} Let $k,K$ and $L$ be as in Proposition \ref{modramprop} and assume that $K$ is of type $2$ or
$3$. Then there exists a set $U\subset L$ such that the set $B=\{1\}\sqcup U\sqcup U^p\sqcup U^{p^2}\dots$ is an
orthogonal Schauder basis of $L$ over $k$ and any element $u\in\Span_k(U)$ is orthogonal to $L^p$, i.e.
$|u+c^p|\ge|u|$ for any $c\in L$.
\end{prop}
\begin{proof}
If $K$ is of type $3$ then $K\toisom k\{r^{-1}S,rS^{-1}\}$ for some $r\notin|k^\times|$. In particular,
$\tilK=\tilk$ is algebraically closed and hence $L$ is totally ramified. Note also that
$\sqrt{|K^\times|}/|K^\times|\toisom r^\bfQ/r^\bfZ$. A well known description of moderately ramified extensions
(see for example \cite[3.4.4(iii)]{Ber2}) implies that $K$ possesses a unique totally ramified extension of
degree $n:=[L:K]$. Clearly $L$ and $K(S^{1/n})$ are two such extensions, hence $L\toisom K(S^{1/n})\toisom
k\{s^{-1}T,sT^{-1}\}$ for $s={r^{1/n}}$ and $T=S^{1/n}$. Therefore, we can take $U=T^{\bfZ\setminus p\bfZ}$
(i.e. all integral powers of $T$ with exponent prime to $p$) and then $B=T^\bfZ$. Assume now that
$K=\wh{\Frac(k\{T\})}$ (i.e. $T$ topologically generates $K$ over $k$ and the induced norm on $k[T]$ is the
Gauss norm); then $|K^\times|=|k^\times|$ is divisible, hence $L$ is unramified over $K$. We will need the
following Lemma.

\begin{lem}
\label{basislem} Let $E$ be a perfect field of characteristic $p>0$ and $F$ be a finitely generated extension of
$E$ such that $E$ is algebraically closed in $F$. Then there exists a set $U\subset F\setminus F^p$ such that
$\{1\}\sqcup U\sqcup U^p\sqcup\dots$ is a basis of $F$ over $E$ and $\Span_E(U)\cap F^p=0$.
\end{lem}
\begin{proof}
A naive attempt would be to take any basis $\tilU$ of $F/F^p$ and to lift it to a subset $U\subset F$
arbitrarily. The set $B=\{1\}\sqcup U\sqcup U^p\dots$ is indeed linearly independent, and $F=F^{p^n}+\Span_E(B)$
for any $n$. Nevertheless, this is not enough for $B$ being a basis. In particular, we must take finite
generatedness of $F$ into account to exclude the possibility that $\cap_{n=1}^\infty F^{p^n}$ is strictly larger
than $E$. For this reason, we refine the above strategy by introducing a norm on $F$ related to finite
generatedness of $F/E$ and lifting $\tilU$ by use of an orthogonal complement procedure.

Choose a proper normal model $X$ of $F$ over $E$. For any $f\in F$, let $\|f\|$ be the maximal order of poles of
$f$ on the points of $X$ of codimension one. Then $\|\ \|$ induces on the $E$-vector space $F$ a non-Archimedean
semi-norm whose kernel coincides with $E$. The residue semi-norm on $V=F/F^p$ is actually a norm. Indeed, the
kernel consists of the images of elements $f\in F$ such that $\|f-g^p\|=0$, but then $c=f-g^p$ is a constant and
by perfectness of $E$ we obtain that $c=b^p$, whence $f\in F^p$ and its image in $V$ is zero. The residue norm
induces an increasing exhausting filtration $V_1\subset V_2\subset\dots$ on $V$ by balls of radii less than $n$
for $n=1,2,\dots$. Find a subset $U=U_1\sqcup U_2\sqcup\dots$ of $F\setminus F^p$ such that each $U_n$ consists
of elements of norm $n$ and the image of $\sqcup_{i=1}^n U_i$ in $V$ is a basis of $V_n$. We will see that $U$
is as required. For any $f\in F$ there exists a unique element $f_0\in\Span_E(U)$ such that $f-f_0\in F^p$. In
particular, $\Span_E(U)\cap F^p=0$. Moreover, it follows from the construction that $\|f_0\|\le \|f\|$. Next,
there exists a unique $f_1\in\Span_E(U^p)$ such that $f-f_0-f_1\in F^{p^2}$, and then $\|f_1\|\le\|f\|$, etc. It
remains to notice that at some stage we obtain $f_i\in F^{p^i}$ which satisfies $\|f_i\|\le\|f\|<p^i$, and then
$f_i$ is necessarily a constant. So, $\{1\}\sqcup U\sqcup U^p\sqcup\dots$ is indeed a basis of $F$.
\end{proof}

We use Lemma \ref{basislem} to find $\tilU\in\tilL\setminus\tilL^p$ such that $1$ and $p^n$-th powers of $\tilU$
for $n\ge 1$ form a basis $\tilB$ of $\tilL$ over $\tilk$. Though one could expect that it suffices to lift
$\tilU$ to $L$ in an arbitrary way, some care should be exercised at this point. We will use in the sequel that
any analytic field $F$ is henselian, and so, as we noted in \S\ref{basicsec}, for any finite separable extension
$E^\sim/\tilF$ there exists a unique unramified extension $E/F$ with $\tilE=E^\sim$. Furthermore, by \cite[$\rm
IV_4$, 18.8.4]{ega} for any analytic $F$-field $E'$ with an $\tilF$-embedding $\tili:\tilE\into\tilE'$ there
exists a lifting $i:E\into E'$ to an embedding of analytic $F$-fields.

Recall that $k=k^a$ by \ref{assum}(i). Set $\pi=p$ in the mixed characteristic case, and choose any non-zero
$\pi\in\kcirccirc$ in the equal characteristic case. It is easy to find a discrete valued field $k_0\subset k$
such that $\tilk_0=\tilk$ and $\pi$ is a uniformizer. For example, let $\tilS=\{\tilS_i\}_{i\in I}$ be a
transcendence basis of $\tilk/\bfF_p$ with a lifting $S\subset k$. Then $k$ contains a discrete valued field
$k_1$ topologically generated over the prime field by $\pi$ and the elements $S_i^{1/p^n}$ for $n\in\bfN$ and
$i\in I$ and we have that $\tilk_1=\bfF_p(S^{1/p^\infty})$. Since $\tilk=\tilk_1^s$, the extension
$\tilk/\tilk_1$ lifts to an unramified extension $k_0/k_1$ with $k_0\subset k$, and the field $k_0$ is as
required.

Let $K_0$ be the closure of $k_0(T)$ in $K$, so $\tilK_0=\tilk(\tilT)=\tilK$, and let $L_0$ be the subfield of
$L$ such that $L_0/K_0$ is the unramified extension corresponding to $\tilL/\tilK$. By \cite{BGR}, 2.7.3/2 and
2.7.5/2, any lifting of a $\tilk$-basis of $\tilL=\tilL_0$ gives an orthonormal Schauder basis of $L_0$ over
$k_0$. So, we take $U\subset L_0$ to be any lifting of $\tilU$, obtaining an orthonormal Schauder basis
$B=\{1\}\sqcup U\sqcup U^p\sqcup U^{p^2}\sqcup\dots$ of $L_0$ lifting the basis $\tilB$. Note that
$K_0\toisom\wh{\Frac(k_0\{T\})}$ because $K$ induces the Gauss norm on $k_0[T]$, hence $K_0\wtimes_{k_0}k\toisom
K$ and so $L_0\wtimes_{k_0}k\toisom L$. The latter isomorphism implies that $B$ is also an orthonormal Schauder
basis of $L$ over $k$. It remains to check that $\Span_k(U)$ is orthogonal to $L^p$. Assume to the contrary that
there exists $u=\sum a_iu_i$ with $a_i\in k^\times$, $u_i\in U$ and $v\in L$ such that $|u-v^p|<|u|$. Since
$|L^\times|=|k^\times|$ is divisible, we can re-scale this using an element from $k^p$ so that $|u|=|v|=1$, and
then by the orthonormality $|a_i|\le 1$. It follows that $\sum\tila_i\tilu_i=\tilv^p$, which is an absurd since
$\Span_\tilk(\tilU)\cap\tilL^p=0$.
\end{proof}

Next, we assume that $K=\ol{k(z)}$ satisfies \ref{assum}(ii). We have observed earlier that $r=\inf_{\alpha\in
k}|z-\alpha|=\inf_{\alpha\in k^a}|z-\alpha|$ and neither infimum is achieved. Consider a sequence
$0=\alpha_0,\alpha_1,\dots$ of elements of $k$ such that the sequence $r_i=|z_i|$, where $z_i=z-\alpha_i$,
monotonically decreases and tends to $r$. Clearly, the discs $E_i=E_k(\alpha_i,r_i)\subset\bfA_k^1$ have a
unique common point $x$ and $\calH(x)\toisom K$. The field $K$ contains a dense subfield $\kappa(x)=\cup
k\{r_i^{-1}z_i\}$, in particular, $K=\ol{k[z]}$. Caution: the elements $1,z_i,z_i^2\dots$ do not form a Schauder
basis, even worse, they do not form a topological generating system in the sense of \cite[2.7.2]{BGR}. Any
non-zero element $b\in k[z]$ is invertible in a neighborhood of $x$, hence replacing $z$ with some $z_j$ we can
achieve that $b=\sum_{i=0}^m b_iz^i$ and the element $b(T)=\sum_{i=0}^mb_iT^i\in k\{|z|^{-1}T\}$ is invertible.
We will need the following Lemma, which implies in particular that $|b|\ge|b_m|r^m$.

\begin{lem}
\label{type4lem} Let $L$ be an analytic $k$-field, $b=\sum_{i=0}^m b_iz^i$ be an element of $L$ such that
$b_i\in k$ and $z\in L$ satisfies $\inf|z-k^a|=r>0$, where the distance is measured in $\wh{L^a}$. Then
$|b|\ge|b_m|r^m$ and the inequality is strict if $m>0$ and $|z-a|>r$ for any $a\in k^a$.
\end{lem}
\begin{proof}
Since $r>0$, we have that $\inf|z-k^a|$ is achieved if and only if $\inf|z-\whka|$ is achieved. So, we can
replace $k$ and $L$ with $\whka$ and $\wh{k^aL}$, achieving that $k$ is algebraically closed. Note that when
replacing $z$ with $z-a$ for $a\in k$ we do not change $b_m$. If there exists $a\in k$ with $|z-a|=r$ then the
induced norm on $k[z-a]$ is the Gauss norm of radius $r$, hence $|b|\ge|b_m|r^m$. If the infimum $\inf_{a\in
k}|z-a|$ is not achieved then $\ol{k(z)}$ is of type $4$ over $k$, and we saw before the Lemma that replacing
$z$ with some $z-a$ we can achieve that $\sum b_iT^i$ is invertible in $k\{|z|^{-1}T\}$. Then
$|b|=|b_0|>|b_iz^i|$ for any $i>0$. In particular, $|b|>|b_mz^m|>|b_m|r^m$ for $m>0$.
\end{proof}

A type $4$ field $K$ is not Cartesian over $k$ (and if it is split then it contains finite dimensional subspaces
which have no orthogonal bases). Sometimes it is convenient to enlarge the ground field so that
orthogonalization becomes possible. Let $l$ be an analytic $k$-field with an isometric embedding $\phi:K\to l$,
and let $\oz=\phi(z)$. Set $L=K\wtimes_k l$ and $w=z-\oz$ (i.e. $w=z\otimes 1-1\otimes\oz$), then we claim that
$L\toisom l\{r^{-1}w\}$. Recall that $K=\ol{k[z]}$ and the norm on $k[z]$ is the infimum of the Gauss norms on
$E_k(\alp_i,r_i)$, hence $L=\ol{l[z]}$ and the norm on $l[z]$ is the infimum of the Gauss norms on
$E_l(\alp_i,r_i)$. But the latter discs have common Zariski closed point corresponding to $\oz$, hence
$E_l(\alp_i,r_i)=E_l(\oz,r_i)$ and the infimum norm is the Gauss norm of the closed disc $E_l(\oz,r)$. (On the
geometric side, we showed that the preimage of the point $x=\cap_{i=1}^\infty E_k(\alp_i,r_i)$ under the
projection $\bfA^1_l\to\bfA^1_k$ is the closed disc $E(\oz,r)$.) In particular, we see that the set $w^\bfN$ is
an orthogonal Schauder basis of $L$ over $l$. We remark that $L$ is an $l$-affinoid algebra with a
multiplicative norm, and it admits an isometric embedding into the $l$-field $\wh{\Frac(L)}$ of type $2$ or $3$,
depending on whether $r\in\sqrt{|l^\times|}$ or not.

\begin{lem}
\label{lemrank1} Keep the above notation and assume that $|pz|<\inf|z-k|$. Then $z^{1/p}\notin K$.
\end{lem}
\begin{proof}
Assume, on the contrary, that $z^{1/p}\in K$. Choose an embedding $\phi:K\to l$ with an algebraically closed $l$
and set $\oz=\phi(z)$. As earlier, let $w=z-\oz\in L=K\wtimes_k l$, and let also $u=z^{1/p}+(-\oz)^{1/p}$. Using
that $\binom{p}{i}$ is divided by $p$ for $0<i<p$ one easily obtains the inequality $|u^p-w|\le|pz|<r=|w|$ in
$L\toisom l\{r^{-1}w\}$. The latter is impossible. Indeed, if $r\notin|l^\times|$ then $r=|w|$ is not a $p$-th
power in $|L^\times|=r^\bfN|l^\times|$, and if $r\in|l^\times|$ then after re-scaling by an element of $l^p=l$
we can achieve that $r=1$ and then $\tilw$ is not a $p$-th power in $\tilL=\till[\tilw]$.
\end{proof}

The following unpleasant Lemma plays a key role in our treating of type $4$ fields.

\begin{lem}
\label{dirtylem} Keep the above notation, and let $b+S_{a,s}(K)$ be a special coset. Assume that $b=\sum_{i=0}^m
b_iz^i\in k[z]$ and $(p,m)=1$. If $a=1$ then we assume in addition that $|pb_iz^i|<1$ for $i\ge 1$. Then
$|b_m|r^m\le s$, and the inequality is strict when $m>1$.
\end{lem}
\begin{proof}
Recall that $r=\inf|z-k^a|=\inf|z-k|$. Shifting $z$ by an element of $k$ we can achieve that $|pz|<r$. We will
later need the following two claims, which follow easily from Lemma \ref{easylem} below: (*) for any $n$ with
$p^n\le m$ one has that $\inf|z^{1/p^n}-k^a|=r^{1/p^n}$, (**) if $k$ is as in \ref{assum}(ii) then
$\inf|k-c^{1/p^n}|\le|pc|^{1/p^n}$ for any $c\in k$.

\begin{lem}\label{easylem}
Given elements $x_1\. x_m\in\whka$ one has that

(i) $|(\sum_{i=1}^m x_i)^{p^n}-\sum_{i=1}^m x_i^{p^n}|\le|p|\max_{1\le i\le n}|x_i|^{p^n}$;

(ii) $|(\sum_{i=1}^m x_i)^{1/p^n}-\sum_{i=1}^m x_i^{1/p^n}|\le |p|^{1/p^n}\max_{1\le i\le n}|x_i|^{1/p^n}$,
where the choice of the root is not important since $\xi_1^{p^n}=\xi_2^{p^n}=1$ implies that
$|\xi_1-\xi_2|<|p|^{1/p^n}$.
\end{lem}
\begin{proof}
(i) is clear and to prove (ii) we estimate the $p^n$-th power of the difference $(\sum_{i=1}^n
x_i)^{1/p^n}-\sum_{i=1}^n x_i^{1/p^n}$ by use of (i).
\end{proof}

Choose an embedding $\phi:K\to l$, where $l$ is an algebraically closed analytic $k$-field, and set
$\oz=\phi(z)$, $L=K\wtimes_k l$ and $w=z-\oz$. Recall that $|w|=r$ and $L\toisom l\{r^{-1}w\}$, and so $w^\bfN$
is an orthogonal Schauder basis. Note that $b=\sum_{i=0}^m \ob_iw^i$, where
\begin{equation}
\label{ob} \ob_i=\ob_i(\oz)=\sum_{j=i}^m\binom{j}{i}b_j\oz^{j-i}\in\phi(K)\subset l
\end{equation}
Recall that $|pb|<s$ by Lemma \ref{212lem}(i), so each term of $b$ satisfies $|p\ob_iw^i|<s$. Similarly to the
case when $L$ is a field, the set $S_{a,s}(L):=\{c^p-ac+d|\ c,d\in L,|pc^p|<s,|d|<s\}$ is an additive group, and
obviously $s':=\inf|b+S_{a,s}(L)|\le\inf|b+S_{a,s}(K)|=s$. To ease the exposition we separate the cases (i)
$a=0$, and (ii) $a=1$ and $s=1$.

We start with (i) since it is easier. We claim that in this case $|\ob_1w|\le s$. Indeed, for any $c=\sum
c_iw^i\in L$ with $|pc^p|<s$ the first term of $b+c^p=\sum c'_iw^i$ is $c'_1w=(\ob_1+pc_1c_0^{p-1})w$. So, if
$|\ob_1w|>s$ then there exists $c$ as above with $|\ob_1w+pc_1c_0^{p-1}w|<|\ob_1w|$ and the latter would imply
that $|pc^p|\ge|pc_1c_0^{p-1}w|=|\ob_1w|>s$. The contradiction proves the claim, and hence $|\ob_1|\le
s/|w|=s/r$. On the other hand we can estimate $|\ob_1|$ by applying Lemma \ref{type4lem} to $\phi(K)$. Since
$|m|=1$ in $L$ and $\ob_1=mb_m\oz^{m-1}+\dots$ by (\ref{ob}), the Lemma implies that $|b_m|r^{m-1}\le|\ob_1|$
and the inequality is strict when $m>1$. Thus, $|b_m|r^{m-1}\le s/r$ and the first case is established.

Now let us assume that $a=1$ and $s=1$. Let $\ob_{pi}w^{pi}$ be the non-zero term of $b$ with largest $i$. Since
$|p\ob_{pi}w^{pi}|<1$ and $l$ is algebraically closed, $c:=\ob_{pi}^{1/p}w^i$ is an element of $L$ satisfying
$|pc^p|<1$, and replacing $b$ with $b-c^p+c$ we find another element of $b+S_{1,1}(L)$ with smaller value of
$i$. Iterating this procedure we obtain an element $B=\ob w+\sum_{i>1,(i,p)=1}x_iw^i\in b+S_{1,1}(L)$ where
$\ob=\ob_1+\ob_p^{1/p}+\dots+\ob_{p^N}^{1/p^N}$ and $N=[\log_p(m)]$. Similarly to the case (i), we will now
prove that $|\ob w|\le 1$. Assume that the inequality fails. Then there exists $c=\sum c_iw^i$ with $|pc^p|<1$
and $|B+c^p-c|<|\ob w|$. Note that $c^p$ coincides with $C:=\sum c_i^pw^{pi}$ up to terms of absolute value
smaller than $1$, hence the inequality $|B+c^p-c|<|\ob w|$ can hold only when $|B+C-c|<|\ob w|=|c_1w|$. Choose
the maximal $i$ with $|c_iw^i|\ge|\ob w|$, then the $pi$-th term of $B+C-c$ is $(c_i^p-c_{pi})w^{pi}$. Since
$|c_i^pw^{pi}|\ge|\ob w|^p>|\ob w|>|c_{pi}w^{pi}|$, we obtain that $|B+C-c|\ge|\ob w|^p>|\ob w|$. The
contradiction proves that $|\ob w|\le 1$. Now, it suffices to show that $|b_m|r^{m-1}\le|\ob|$ and the
inequality is strict when $m>1$. This will be done below to accomplish the proof.

As in (i), we would like to apply Lemma \ref{type4lem} to $\ob$. Since $\ob$ does not have to be in $\phi(K)$,
we will also have to approximate it with an element $\ob'\in\phi(K)(\oz^{1/p^N})$. It will be convenient to
assume that $|z|^{-1}<|b_m|r^{m-1}$, and this is harmless because otherwise $1\ge|b_mz|r^{m-1}>|b_m|r^m$, which
is the assertion of the Lemma. Recall that in the mixed characteristic case $|b_jz^j|<|p^{-1}|$ for $j\ge 1$ (by
assumption of the Lemma). Therefore, in the formula (\ref{ob}) for $\ob_{p^n}$ the absolute value of the
summands is strictly smaller than $R:=|p|^{-1}|z|^{-p^n}$. Using Lemma \ref{easylem}(i) to approximate the
$p^n$-th root we obtain that
\begin{equation}
\label{obb} (\ob_{p^n})^{1/p^n}=W+\sum_{j=p^n}^m{\binom{j}{p^n}}^{1/p^n}b_j^{1/p^n}\oz^{(j-p^n)/p^n}
\end{equation}
where the error term $W$ satisfies $|W|<|p|^{1/p^n}R^{1/p^n}=|z|^{-1}<|b_m|r^{m-1}$. Recall that
$\inf|k-b_j^{1/p^n}|\le|pb_j|^{1/p^n}$ by claim (**). In particular, each $b_j^{1/p^n}$ in (\ref{obb}) can be
replaced by an element of $k$ without increase in the error term (use that $|b_j|<|p|^{-1}|z|^{-j}\le R$ for
$j\ge p^n$). Summing up such approximations of the elements $(\ob_{p^n})^{1/p^n}$ for $1\le n\le N$ (without the
error terms) we obtain an approximation $\ob'$ of $\ob$. More specifically, we obtain an element $\ob'\in
k[\oz^{1/p^N}]$ such that $|\ob'-\ob|<|b_m|r^{m-1}$ and $\ob'$ is a polynomial of degree $(m-1)p^N$ in $\oz$ and
with the highest degree term $mb_m\oz^{m-1}$, which is the contribution of $\ob_1$. Recall that
$\inf|z^{1/p^N}-k^a|=r^{1/p^N}$ by claim (*), and hence $\inf|\oz^{1/p^N}-k^a|=r^{1/p^N}$. Applying Lemma
\ref{type4lem} to the field $\ol{k(\oz^{1/p^N})}$, we obtain that
$|\ob'|\ge|b_m|(r^{1/p^N})^{(m-1)p^N}=|b_m|r^{m-1}$ and the inequality is strict when $m>1$. It follows that
$\ob$ satisfies the same inequalities and we are done.
\end{proof}

\begin{proof}[Proof of Proposition \ref{modramprop}]
Let $b+S_{a,s}(L)$ be a critical coset. The strategy of the proof is very simple: we gradually improve $b$
moving it inside $b+S_{a,s}(L)$ until either $|b|=s$ or $b$ is a topological generator. If $K$ is of type $2$ or
$3$ (so $k=k^a$) then $L$ possesses a Schauder basis $B=1\sqcup U\sqcup U^p\sqcup\dots$ as in Proposition
\ref{schaudprop}. Moving $b$ a little, we can assume that the representation $b=\sum_{v_i\in B} a_iv_i$ involves
finitely many non-zero terms. Let $b=a_0+\sum_{i=1}^N\sum_{u_j\in U}a_{ij}u_j^{p^i}$. If $N>0$ then subtracting
from $b$ elements of the form $a_{ij}u_j^{p^i}-aa_{ij}^{1/p}u_j^{p^{i-1}}$ we can achieve that it is contained
in $\Span(1\sqcup\dots\sqcup U^{p^{N-1}})$.  (One has also to check that $|pa_{ij}u_j^{p^i}|<s$, but this is
obvious because $|pb|<s$ by Lemma \ref{212lem}(i)). Iterating the process we achieve that $b=a_0+\sum_{u_j\in
U}a_ju_j$. Since $k$ is algebraically closed, $a_0=c^p-ac$ for some $c\in k$, and so we can remove $a_0$ as
well. Now, we claim that the absolute value of $b$ cannot be reduced further by adding elements $c^p-ac$.
Indeed, if $|b+c^p-ac|<|b|$ then $|b|>s$, so $|c^p-ac|=|b|>s$. Since $s^{\frac{p-1}{p}}\ge |a|$, the latter is
possible only when $|c^p|>|ac|$. Therefore $|b+c^p|<|b|$, contradicting the property that $b\in\Span(U)$ is
orthogonal to $L^p$. This settles the Proposition for types $2$ and $3$.

In the case of \ref{assum}(ii), $L=K$ because $\tilK=\tilk$ is algebraically closed and $|K^\times|=|k^\times|$
is divisible by Remark \ref{rem0}. Set $r=\inf|z-k|$, as usually. Replacing $b$ with a sufficiently close
element we can achieve that $b\in k[z]$, say $b=\sum_{i=0}^m b_iz^i$. Using linear change of the coordinate $z$
we can achieve that $\sum b_iT^i$ is invertible in $k\{|z|^{-1}T\}$, and then $|pb_iz^i|\le|pb|<s$. Now, if
$(p,m)=1$ and $m>1$ then Lemma \ref{dirtylem} implies that $|b_m|r^m<s$. Choose $z_0\in k$ such that
$d=b_m(z-z_0)^m$ satisfies $|d|<s$. Subtracting $d$ from $b$, we decrease the degree of $b$. If $p|m$ and
$\cha(k)=p$ then we subtract $c^p-ac$ with $c=b_m^{1/p}z^{m/p}$ from $b$ decreasing the degree of $b$, and then
use another linear coordinate change to restore the condition $|pb_iz^i|<s$.  In the mixed characteristic case
with $p|m$, $b_m^{1/p}$ does not have to be in $k$ but by \ref{assum}(ii) there exists $c_0\in k$ so that
$|c_0^p-b_m|\le|pb_m|$. Since $|pb_mz^m|<s$, for $c=c_0z^{m/p}$ we have that $|c^p-b_mz^m|<s$. Thus, we can
replace $b$ with $b-c^p+ac$ so that the absolute value of the $m$-th term of $b$ drops below $s$ and then we can
safely remove it. All in all, we can move $b$ inside the coset until its degree $m$ drops below two. If $m=1$
then $b$ is a topological generator of $K$, and if $m=0$ then $b\in k$.
\end{proof}

\subsection{One-dimensional analytic fields} \label{onedimsec}
Let $z\in K\setminus k$ be an element. We say that it is an {\em unramified generator} (resp. {\em moderately
ramified generator}) if $K/\ol{k(z)}$ is a finite unramified extension (resp. finite moderately ramified
extension). In this section, generator always means topological generator. The main result of this section
states that any one-dimensional analytic $k$-field possesses an unramified generator.

\begin{theor}
\label{fieldunif} Let $L$ be a one-dimensional analytic $k$-field satisfying the conditions of \ref{assum},
then:

(i) $L$ possesses an unramified generator over $k$;

(ii) if $p>0$ and $k=k^a$ then any special coset in $L$ contains a moderately ramified generator of $L$;

(iii) if $L$ is of type $2$ or $3$ then it is stable.
\end{theor}

We start with few simple lemmas. Until the end of Corollary \ref{fincor} $k$ is an arbitrary analytic field.

\begin{lem}
\label{epsclose} Let $K=\ol{k(z)}$ be a one-dimensional analytic $k$-field and $L$ be an analytic $K$-field. Set
$r=|z-k^a|>0$ (computed in $\wh{K^a}$) and assume that $z'\in L$ satisfies $|z-z'|<\ve r$ for some $\ve<1$. Then
there exists an isometric embedding $\phi:K\into L$ over $k$ such that $\phi(z)=z'$ and $|x-\phi(x)|<\ve|x|$ for
any $x\in K^\times$.
\end{lem}
\begin{proof}
For any linear polynomial $f(T)$ over $k$ we have that $|f(z)-f(z')|<\ve|f(z)|$. Assume given $x,x',y,y'\in L$
with $|x-x'|<\ve|x|$ and $|y-y'|<\ve|y|$ (this assumption is symmetric in $x,x'$ and $y,y'$, and it forces that
$x,x',y,y'\in L^\times$). It is easy to check that $|xy-x'y'|<\ve|xy|$ and
$|\frac{x}{y}-\frac{x'}{y'}|<\ve|\frac{x}{y}|$. Hence the above inequality holds for any rational function
$f(T)\in k(T)$ (by factoring $f(z)$ in $\wh{k^aL}$). From this the Lemma follows by continuity.
\end{proof}

\begin{lem}
\label{fin} If $\alpha,\beta:K\to L$ are embeddings of analytic fields and $\ve<1$ is a number such that
$|\alpha(x)-\beta(x)|<\ve|x|$ for any $x\in K$ then $[L:\alpha(K)]=[L:\beta(K)]$.
\end{lem}
\begin{proof}
It suffices to prove that if $[L:\alpha(K)]=n<\infty$ then $[L:\beta(K)]\le n$. Fix a number $\ve<r<1$ and pick
up an $r$-orthogonal basis $x_1\. x_n$ of $L$ over $\alpha(K)$; we recall (see \cite[2.6.1/3]{BGR}) that this
means that for any choice of $a_1\. a_n\in\alpha(K)$ one has that $|\sum a_ix_i|\ge r\max_i(|a_ix_i|)$. Let
$x\in L$ be arbitrary. Express it as $x=\sum\alpha(a_i)x_i$ for some elements $a_i\in K$ and set
$x'=\sum\beta(a_i)x_i$, then $|x-x'|<\ve\max_i(|a_i||x_i|)\le\frac{\ve}{r}|x|$. In a similar way we can
approximate $x-x'$ by an element from the vector space $L':=\sum x_i\beta(K)$, etc., obtaining in the end
arbitrary good approximations of $x$ by elements from $L'$. By completeness, $L'=L$ and obviously
$[L':\beta(K)]\le n$.
\end{proof}

\begin{cor}
\label{fincor} A one-dimensional $k$-field $L$ is finite over any one-dimen\-sional subfield $K$.
\end{cor}
\begin{proof}
We may decrease $K$, so replace it with a subfield $\ol{k(x)}$ where $x\notin\whka$. Let $y\in L$ be such that
$L$ is finite over $L_0=\ol{k(y)}$ and let $L_1$ be the separable closure of $L_0$ in $L$. Replacing $x$ with
some $x^{p^n}$ and decreasing $K$ accordingly we achieve that $x\in L_1$, and then it suffices to show that
$[L_1:K]<\infty$. So, replacing $L$ with $L_1$ we can assume that $L/L_0$ is separable and then by Krasner's
lemma $L=\oll$ for a finite separable extension $l/k(y)$. Pick up an element $x'\in l$ such that
$|x-x'|<|x-k^a|$ and set $K'=\ol{k(x')}$. The above two Lemmas imply that $[L:K]=[L:K']$, and it remains to note
that $[L:K']\le[l:k(x')]<\infty$ because the degree of an extension can only drop by passing to completions.
\end{proof}

\begin{rem}
The caution exercised in the proof of the above Corollary is explained by the following rather surprising fact.
One can naturally define the topological transcendence degree of the extensions, but the latter does not have to
be additive in towers. In particular, in some cases there exist non-surjective analytic $k$-endomorphisms of
fields $\wh{k(x)^a}$. However, this non-additivity can only occur in towers containing deeply ramified
extensions (i.e. extensions with infinite different). All this is not needed in this paper, so topological
transcendence degree will be studied elsewhere.
\end{rem}

\begin{proof}[Proof of Theorem \ref{fieldunif}]
The case of $p=0$ is obvious, so we assume that $p>0$. We start with the type $2$ or $3$ case. Since
$\tilk=\tilk^a$ and $|k^\times|$ is divisible, there exists $z\in L$ such that $\tilL/\tilk(\tilz)$ is a finite
separable extension in the type $2$ case and $|L^\times|=|z^{\bfZ}k^\times|$ in the type $3$ case. Note that
$L/\ol{k(z)}$ is finite by Corollary \ref{fincor}. Let $K$ be the unramified closure of $\ol{k(z)}$ in $L$, then
$K$ admits an unramified generator and the extension $L/K$ is finite and immediate. So, $L/K$ is totally wildly
ramified and $[L:K]=p^n$. The extension $LK^\mr/K^\mr$ splits to a tower of extensions of degree $p$ because
$G=\Gal(K^\mr)$ is a pro-$p$-group (hence any subgroup $H\subset G$ of index $p^n$ possesses a tower of larger
subgroups $H=H_0\subset H_1\subset\dots\subset H_n=G$ with $\#(H_{i+1}/H_i)=p$) and any purely inseparable
extension splits into such a tower as well. It follows that there exists a finite moderately ramified extension
$K'/K$ such that the extension $K'L/K'$ splits to a tower of $p$-extensions. In our situation, $K'$ admits a
moderately ramified generator and is of type $2$ or $3$, hence by Propositions \ref{modramprop} and
\ref{immprop} $K'$ has no immediate extensions of degree $p$. Thus $K'L=K'$, and we obtain that $L\subseteq K'$
is moderately ramified over $K$. So, $L=K$ and we have proved (i). Next we note that for any finite extension
$L'/L$ the one-dimensional $k$-field $L'$ possesses a moderately ramified generator by applying (i) to $L'$.
Propositions \ref{modramprop} and \ref{immprop} imply that $L'$ has no immediate extensions of degree $p$, hence
$L$ is stable. In particular, this gives (iii).

For types $2$ and $3$ it remains only to prove (ii). Let $b+S_{a,s}(L)$ be a special coset. By Proposition
\ref{modramprop}, the value of $|c^p-ac+b|$ accepts its minimum $s$ for some $c\in L$. Set $z=c^p-ac+b$. If
$s=|z|\notin|L^\times|^p$ then $L$ is of type $3$, $L/\ol{k(z)}$ is Cartesian by (iii) and
$|L^\times|/|k(z)^\times|$ is prime to $p$, hence we obtain that $L/\ol{k(z)}$ is moderately ramified, i.e. $z$
is a moderately ramified generator. If $s\in |L^\times|^p$ but $s\notin|k^\times|$, then necessarily $a=0$ and
$L$ is of type $3$. Since $\tilK=\tilk=\tilk^a$, there exists $c\in L$ with $|z-c^p|<|z|$, and we obtain a
contradiction with the minimality of $|z|$. It remains to deal with the case when $|z|\in|k^\times|$. Assume
first that $a=0$. Since $k=k^a$ and $z$ is orthogonal to $L^p$ (i.e. $|z+c^p|\ge |z|$ for $c\in L$), the same is
true for any element $tz$ with $t\in k=k^p$. Taking $t$ such that $|tz|=1$ we obtain from orthogonality that
$\wt{tz}\notin\tilL^p$. In particular, $\tilL\neq\tilk$ and so $L$ is of type 2. The extension $L/\ol{k(z)}$ is
Cartesian by (iii), $|L^\times|=|k^\times|$ and $\tilL$ is separable over $\tilk(\wt{tz})$ and hence over the
residue field of $\ol{k(z)}$. Thus, $L/\ol{k(z)}$ is an unramified extension and $z$ is an unramified generator.
Finally, we are left with the case when $a=1$ and $|z|=s=1$. By our choice of $z$, the equation $x^p-x+z=0$ has
no solutions in $\tilL$. Since $\tilk=\tilk^a$ we can use Lemma \ref{basislem} to prove that there exists an
element $\tilc\in\tilL$ such that $\tilc^p-\tilc+\tilz\notin\tilL^p$. Taking a lifting $c$ of $\tilc$ and
replacing $z$ with $c^p-c+z$ we achieve the situation when $\tilz\notin\tilL^p$. Then the same argument as above
proves that $z$ is an unramified generator.

Assume now that $L$ is of type $4$ in \ref{assum}(ii). First, let us prove (ii) assuming (i). In view of (i) and
Proposition \ref{modramprop}, we have only to rule out the possibility that a critical coset $b'+S_{a,s}(L)$
contains an element $b\in k$. Since $k$ is algebraically closed, the equation $b=c^p-ac$ has a solution $c\in k$
which implies that already $b+S_{a,s}(k)$ is split. The contradiction shows that (i) implies (ii), so we should
only establish (i). Since $\tilL=\tilk$ is algebraically closed and $|L^\times|=|k^\times|$ is divisible, $L$
admits only immediate algebraic extensions. In particular, any finite extension of $L$ splits to a tower of
$p$-extensions. Thus we have only to prove that if $K=\ol{k(z)}$ and $[L:K]=p$ then $L$ possesses a generator.
Since $K$ has no non-trivial moderately ramified extensions, by Proposition \ref{immprop}(ii) we can find a
special coset $b+S_{a,s}(K)$ as in \ref{immprop}(i). In particular, the coset $b+S_{a,s}(L)$ is split. By
Proposition \ref{modramprop} we can achieve furthermore that either $b$ is a generator of $K$ or $b\in k$. If
$b\in k$ then there exists $\alp\in L$ with $|\alp^p-a\alp+b|<s$ and this easily implies that $\alp$
approximates a root of $T^p-aT+b=0$ better than any element from $k$. The latter would contradict our assumption
that $L$ is $k$-split, hence $b$ is a generator of $K$.

If $a=1$ then by Proposition \ref{immprop}(iii) the polynomial $T^p-aT+b$ has a root $\alp\in L$ and we claim
that $\alp$ generates $L$. Indeed, $\alp$ generates $L$ over $K$ and $b=-\alpha^p+\alpha$ generates $K$ over
$k$. If $a=0$ we take $\alp\in L$ with $|\alp^p+b|<s$ and we will see that this $\alp$ is a generator. Set
$b'=-\alpha^p$ and $K'=\ol{k(b')}$. By the choice of $\alpha$ we have that $|b-b'|<s=\inf|b+K^p|$. We claim that
the latter infimum does not exceed $\inf|b-k|$. Indeed, if $|b+c|<s$ for some $c\in k$ then by \ref{assum}(ii)
there exists $d\in k$ with $|d^p-c|$ arbitrary close to $|pc|$. Since $|pc|=|pb|<s$ by Lemma \ref{212lem}(i), we
can achieve that $|d^p-c|<s$, which implies that $|b+d^p|<s$. The contradiction proves that
$|b-b'|<\inf|b+K^p|\le\inf|b-k|$, and the latter infimum equals to $\inf|b-k^a|$ because $K$ is $k$-split. By
Lemmas \ref{epsclose} and \ref{fin} we obtain that $[L:K']=[L:K]=p$. It remains to note that
$|pb'|<s\le\inf|b-k|=\inf|b'-k|=\inf|-b'-k|$, hence $\alpha=(-b')^{1/p}\notin K'$ by Lemma \ref{lemrank1}. Thus,
$[\ol{k(\alpha)}:K']=p$, and therefore $\ol{k(\alpha)}=L$.
\end{proof}

As a corollary we prove a slightly generalized version of the stability theorem of Grauert-Remmert, see
\cite[5.3.2/1]{BGR}.

\begin{cor}
If $k$ is a stable analytic field and $K$ is of type $2$ or $3$ over $k$ then $K$ is stable.
\end{cor}
\begin{proof}
Let $z\in K$ be such that either $r:=|z|$ is not in $\sqrt{|k^\times|}$ or $r=1$ and $\tilz\notin k^a$. By
Corollary \ref{fincor} $K$ is finite over $\ol{k(z)}$, so it suffices to establish stability of the latter
field. Thus we can assume that $K=\ol{k(z)}=\wh{\Frac(k\{r^{-1}z\})}$. Suppose to the contrary that a finite
extension $L/K$ is not Cartesian. Since $K^s$ is dense in $K^a$, it suffices to consider the case of a separable
extension $L/K$ (if $s=\inf|\alp-K|$ is not achieved for some $\alp\in K^a$ then there exists $\alp'\in K^s$
with $|\alp'-\alp|<s$, so the separable extension $k(\alp')/k$ is not Cartesian). In particular, we can assume
that $L$ is separable over $k$. Given a finite extension $k'/k$, we will use the notation $L'=k'L$ and $K'=k'K$.
Note that $K'=\wh{\Frac(k'\{r^{-1}z\})}$ and one checks straightforwardly that the extension $K'/K$ is Cartesian
since $k'/k$ is Cartesian. We claim that for sufficiently large $k'$ the extension $L'/K'$ is Cartesian. Indeed,
the extension $\wh{k^aL}/\wh{k^aK}$ is Cartesian by Theorem \ref{fieldunif}(iii), hence it admits an orthogonal
basis $\{a_1\. a_n\}\in\wh{k^aL}$. Moving $a_i$'s slightly does not spoil orthogonality in the non-Archimedean
world, so we can achieve that $a_i\in k^aL$, and then $a_i\in k'L$ for a suitable $k'$. Clearly, $\{a_i\}$ is an
orthogonal basis of $L'/K'$, i.e. this extension is Cartesian. Finally, the extension $L'/K$ splits to a tower
of Cartesian extensions $K\subset K'\subset L'$. Hence $L'/K$ is Cartesian and its subextension $L/K$ is
Cartesian too.
\end{proof}

\begin{rem}
\label{GRrem} Stability theorem for type $2$ fields is the main ingredient in the proof of the Grauert-Remmert
finiteness theorem, see \cite[\S6.4]{BGR}. We saw in this section that stability theorem (both for types $2$ and
$3$) is essentially equivalent to uniformization of one-dimensional analytic fields.
\end{rem}

\begin{proof}[Proof of Theorem \ref{fieldunif1}]
We simply have to combine already proved results. Part (i) of the Theorem is exactly Theorem \ref{fieldunif}(i).
To prove (ii) we note that by Lemma \ref{epsclose}, if $\veps$ is small enough and $|x-x'|\le\veps$ then there
exists an analytic $k$-isomorphism $\phi$ between the fields $L=\ol{k(x)}$ and $L'=\ol{k(x')}$ such that
$|y-\phi(y)|\le|y|/2$ for any $y\in L$. Obviously, $\tilL$ coincides with $\tilL'$ as a subfield of $\tilK$, and
also we have that $[K:L]=[K:L']$ by Lemma \ref{fin}. So, $K/L'$ is unramified because $K/L$ is unramified.
Finally, in (iii) we necessarily have that $K$ is of type $2$. Hence $L=\ol{k(x)}$ is stable by Theorem
\ref{fieldunif}(iii) and it follows that $[K:L]=[\tilK:\tilL]$. In particular, $K/L$ is unramified if and only
if $\tilK/\tilL$ is separable.
\end{proof}

\appendix
\section{Stable modification and desingularization of surfaces} \label{A}
This appendix is an attempt to systemize known results and methods in the theories of semi-stable curves and
desingularization of surfaces. It seems to be impossible to give credits to all mathematicians that have
contributed to these theories, but I try to do my best. One of the aims of this systematization is to stress the
analogy between the two theories, to compare them and to describe an interplay between them. For the sake of
simplicity, all relative curves in this appendix are automatically assumed to have smooth geometrically
connected generic fiber.

\subsection{Two contexts where semi-stable families of curves appear}
The semi-stable families of curves naturally appear in two different contexts: the context of moduli spaces and
the context of desingularization of relative curves. Originally, a systematic study of families of semi-stable
curves was motivated by the theory of moduli spaces of curves. The foundational work in that direction is
\cite{DM}. This is the work where the stable reduction theorem over a discrete valuation ring $R$ appeared for
the first time. It is worth to mention that though it may sound surprisingly today, the theorem was very
surprising when discovered \footnote{The author is grateful to M. Raynaud for a consultation on this issue.}.
Not only it was not expected in the positive/mixed characteristic case, but even in the characteristic zero case
its assertion was strikingly new despite the fact that this case follows easily from desingularization of
surfaces of characteristic zero proved by Zariski in 1930s. The stable reduction theorem was applied in
\cite{DM} to prove that the moduli stack of stable curves is proper, and then the general stable extension
theorem (see the Introduction) follows easily, including, as a particular case, the stable reduction theorem
over non-discrete valuation rings.

On the other hand, relative semi-stable curves can be considered as a relative analog of the notion of a smooth
curve. Indeed, if one starts with a relative curve $\phi:C\to S$ and tries to improve the singularities of
$\phi$ by reasonable (e.g. proper surjective) base changes and modifications of $C$, then the mildest
singularities one can hope to obtain are those of semi-stable curves. (Note that one has to allow base changes
which are not modifications in order to obtain relative curves with reduced fibers. For example, to get rid of
nilpotents in the central fiber of $\phi:\Spec(\bfQ[\pi,x,y]/(\pi x-y^2)\to\Spec(\bfQ[\pi]))$ one has to adjoin
$\sqrt\pi$ to the base.) Thus, de Jong's semi-stable modification theorem \cite[2.4]{dJ} can be considered as a
relative desingularization theorem. Our work in this paper has a clear flavor of desingularization approach,
and, as we will see below, our stable modification theorem is an analog of the minimal desingularization of
surfaces.

\subsection{Desingularization of surfaces}\label{A2}
There are two main theorems concerning smooth models of surfaces: the minimal model theorem and the minimal
desingularization theorem. The first one is an absolute result that is close in nature to the theory of moduli
spaces. It states that if $k$ is an algebraically closed field and $K$ is finitely generated over $k$ with
$\trdeg(K/k)=2$ and is sufficiently generic (namely, $K$ is not of the form $L(T)$ for a subfield $L\subset K$
containing $k$), then $K$ admits a minimal $k$-smooth proper model. The minimal desingularization theorem states
that any integral quasi-excellent two-dimensional scheme $X$ admits a minimal desingularization, i.e. a
modification $X'\to X$ with regular source such that any other modification $X''\to X$ with regular source
factors through $X'$). Unlike the minimal model theorem, this second theorem applies to any integral
quasi-excellent surface. Moreover, it admits generalizations which treat divisors and finite group actions.

Usually, a proof of the minimal model/desingularization theorem is not direct and goes in three steps: (1) find
some regular model/modification, (2) prove that the family of such models/modifications is filtered by
domination, (3) given any regular model/modification construct a minimal regular contraction and establish its
uniqueness. The last two stages are easy and rather standard. Step (2) follows from the following two facts:
(2a) if $X$ is a regular surface then the family of all its modifications that can be obtained by successive
blowing up closed points is cofinal in the family of all modifications of $X$; and (2b) any modification of
surfaces $X'\to X$ with regular $X$ and $X'$ can be obtained by successive blowing up closed points. Step (3) is
done by successive contraction of exceptional $\bfP^1$'s (these are $\bfP^1$'s with self-intersection equal to
$-1$, i.e. with the normal bundle isomorphic to $\calO(1)$), and by certain combinatorial computations with the
intersection form. The heart of the proof is in the first step which we call desingularization of surfaces. We
present two approaches to desingularization of surfaces due to Zariski and Lipman.

Zariski was first to establish desingularization of surfaces over fields of characteristic zero. His approach
was to first desingularize a surface $X$ along a valuation ring. Zariski proved the following local
uniformization theorem which can be considered as a local (on the Riemann-Zariski space of a variety) solution
of the desingularization problem: if $X$ is integral and of finite type over an algebraically closed field $k$
of characteristic zero, $\calO$ is a valuation ring of the field of rational functions $k(X)$ with
$k\subset\calO$ and a $k$-morphism $\Spec(\calO)\to X$ extending the isomorphism of generic points, then there
exists a modification $X'\to X$ such that the lifting $\Spec(\calO)\to X'$, which exists by the valuative
criterion, lands in the smooth locus of $X'$. Local uniformization implies desingularization of surfaces because
one can glue the local solutions using Steps (2a) and (2b) above. Using a much more involved gluing method,
Zariski was also able to obtain desingularization of threefolds in characteristic zero via local uniformization.

Lipman proposed in \cite{Lip} another method of desingularization of a quasi-excellent integral two-dimensional
scheme $X$. At the first step, a modification $X'\to X$ is constructed so that $X'$ is normal and has only
rational singularities, i.e. singular points that (a posteriori) are resolved by trees of $\bfP^1$'s (with a
negative definite intersection form). One easily sees that $X'$ has rational singularities if and only if the
arithmetic genus $p_a(X')=h^1(X',\calO_{X'})$ is minimal in the set $\{X_i\}$ of all modifications of $X$, and
the first step is established by proving that arithmetic genus of $X_i$'s is bounded. At the second step, each
rational singularity point is resolved by $\bfP^1$-trees rather explicitly.

\subsection{Desingularization of relative curves}\label{A3}
The theory of semi-stable modifications of relative curves is analogous in many aspects to the theory of
desingularization of surfaces. Its two main results are the stable extension theorem (mentioned in the
introduction) and the stable modification theorem which are clear analogs of the two main results on
desingularization of surfaces. Note also that de Jong's semi-stable modification theorem is an analog of
(non-minimal) desingularization of surfaces. Moreover, we will see that the theory of relative curves is
slightly easier because some arguments are easier and some results can be proved in a stronger form. In
addition, desingularization can often be used to construct semi-stable modifications, while it is much harder
(though sometimes possible) to go in the opposite direction.

Localizing the base (in the Riemann-Zariski sense) one obtains a very important particular case of the above
theorems: the (semi-)stable reduction theorem. I know two published direct proofs of this theorem: the proof of
Bosch-L\"utkebohmert in \cite{BL1} and the proof of van der Put in \cite{Put} (other proofs at least use
desingularization of surfaces, and we will discuss them in \S\ref{stabsec}). The two proofs are close in spirit
and have many common features with the method of Lipman. Both are rigid-analytic and apply to a formal curve
over any complete valuation ring of height one, and then the algebraic version over any valuation ring of height
one is an easy consequence, see \cite[p. 377]{BL1}. Similarly to Lipman's method, both proofs run in two stages:
first one studies the arithmetic genus of the closed fibers of modifications to prove that there exists an
ordinary modification (i.e. a modification whose closed fiber has only ordinary singularities), then ordinary
singularities are resolved rather explicitly by trees of $\bfP^1$'s.

The new proof of the stable reduction theorem presented in this paper is a close analog of Zariski's approach.
The main ingredient of the proof is uniformization of one-dimensional valued fields in Theorem \ref{valunif}.
Local uniformization of a relative curve along a valuation was easily deduced in Proposition \ref{locunifprop}.
The latter statement is a clear analog of Zariski's local uniformization, and similarly to local uniformization,
which is still unknown in positive characteristic and large dimensions, the case of positive characteristic was
much more difficult. Indeed, the main effort in the proof of \ref{valunif} was in struggling with the effects of
wild ramification, in particular in controlling extensions with defect in \S\ref{immsec}. Gluing local
desingularizations to a global one is, again, rather similar to the surface case described in \ref{A2}. Analogs
of Steps (2a) and (3) from \ref{A2} are Propositions \ref{cofinprop} and \ref{blowdownprop}. The only subtle
point is that we managed to avoid factorization of modifications from Step (2b); see \S\ref{gluesec} for
details.

Finally, we would like to say few words about the history of the closely related and nearly equivalent problems
of uniformization of one-dimensional (analytic) valued fields and local description (or uniformization) of
non-Archimedean curves. Despite the fact that the formulation of Theorem \ref{valunif} seems to be new, it was
clear for experts that similar statements can be deduced from the stable reduction theorem. For example,
Berkovich deduced a local description of analytic curves from the stable reduction theorem, see
\cite[3.6.1]{Ber2}. As for direct valuation-theoretic proofs of uniformization of one-dimensional valued fields,
the author knows about the following works: in the analytic case stability theorem of Grauert-Remmert covered
type 2 case (and, probably, the case of type 3 fields was known to experts), and M. Matignon established the
case of analytic type 4 fields in an unpublished work; in the general case F.-V. Kuhlmann proved generalized
stability theorem for types 2 and 3 in \cite{Kuh1} and uniformization of type 4 fields will be worked out in
\cite{Kuh2}. It seems that in all these works the argument is more computational than ours because one studies
$p$-extensions of valued fields by use of Kummer and Artin-Shreier theories, while we struggled with defect by
use of Proposition \ref{immprop} which covered all cases in a uniform manner. Also, it seems that currently only
our method covers some cases when the ground field is not algebraically closed (see \ref{assum}). We note also
that probably a classical valuation-theoretic work \cite{Epp} of Epp was initially motivated by a hope to obtain
uniformization of valued fields algebraically and to then deduce the stable reduction theorem.

\subsection{Comparison of the two theories}
In the two previous sections we described two parallel desingularization theories. We summarize the analogies
between them in the following table.

\begin{tabular}{|l|l|}
\hline Surfaces:& Relative curves:\\ & \\
\hline Modification of the surface & Alteration of the base and \\
& modification of the curve\\
\hline Desingularization of surfaces& Semi-stable modification\\
\hline Minimal model theorem & Stable extension theorem \\
\hline Minimal desingularization theorem &  Stable modification theorem\\
\hline No analog (no localization & Stable reduction theorem \\
on the base)&\\
\hline Local uniformization of surfaces & Uniformization of one-dimensional\\
 &  valued fields \\
\hline Surfaces with rational singularities & Ordinary relative curves\\
\hline Lipman's method & The methods of Bosch-L\"utkebohmert's\\
 & and van der Put \\
\hline Zariski's method & The method of this paper \\
\hline
\end{tabular}

\subsection{The link between the two theories}

The two parallel theories we have described meet together in the following very important particular case. If
$X\to S$ is a relative curve and the base $S$ is a curve then $X$ is a surface. In particular, one can wonder
what is the connection between the desingularizations of $X$ of two kinds when $S$ is a regular curve. Until the
end of this section we assume that $S$ is a local regular curve, that is $S=\Spec(R)$ for a discrete valuation
ring $R$.

The minimal surface desingularization $X_\sm\to X$ does not need to be semi-stable over $S$ because its special
fiber can be non-reduced. In such situations, one is guaranteed that a non-trivial alteration of the base (i.e.
replacing of $R$ with its integral closure in a finite separable extension $K'/K$ of its field of fractions) is
required in order to construct a semi-stable modification. Conversely, a semi-stable modification which involves
a non-trivial alteration of the base does not help (at least at first glance) to desingularize $X$. If a
semi-stable modification is possible already over $S$ then $X_\sm$ and the stable modification $X_\st$ are
tightly connected (e.g. $X_\sm$ is semi-stable) and one can easily use either of them to construct another one.
Thus, one can expect that the Galois group $G_K$ of $K$ is essentially responsible for the gap between the two
theories, and, indeed, we will see that a good control on the Galois group sometimes makes it possible to pass
from desingularizations to semi-stable models and vice versa.

If the residue field $k=R/m_R$ is algebraically closed of characteristic zero then the Galois group $G_K$ has a
simple structure because it coincides with the tame inertia group. In this case the link between the two
theories is so tight that it even extends to higher dimensions. Given $X_\sm$ one can easily predict what is the
minimal extension $K'$ over which stable modification exists ($K'/K$ is totally ramified and its degree is the
minimal common multiple of the multiplicities of the irreducible components in the special fiber of $X_\sm$).
Moreover, the same argument was used in \cite{KKMS} to deduce a higher dimensional semi-stable reduction theorem
from the desingularization theorem of Hironaka. In opposite direction, de Jong and Abramovich proved in
\cite{AdJ} that the quotient $X_\st/G_{K'/K}$ has very mild toric singularities which can be easily resolved
(thus, giving a link from $X_\st$ to $X_\sm$). Moreover, their argument applies to any base $S$, so they deduce
weak desingularization of higher dimensional algebraic varieties in characteristic zero.

The situation with $k$ of positive characteristic is far more complicated. No general way is known to go in the
difficult direction $X_\st\mapsto X_\sm$ even when $S$ is a curve. The main problem here is to control the
properties of the quotient by a wildly ramified Galois group. The easier link $X_\sm\mapsto X_\st$ can be
established at least in the case of curves. The main idea here is to control the Galois group through its action
on the $l$-adic cohomology group $H^1(X_\eta,\bfQ_l)$ or another invariant of close nature (e.g. Jacobian's
$l$-torsion). In particular, it turns out that the Galois group of $K'$ acts unipotently on $H^1(X_\eta,\bfQ_l)$
(via the embedding $G_{K'}\into G_K$) if and only if $X_\eta$ admits a stable model after the base change
corresponding to the extension $K'/K$. This underlies the proofs of the stable reduction theorem by
Deligne-Mumford (using Grothendieck's semi-stable reduction of abelian varieties), Artin-Winters and Saito.

\subsection{Proofs of the stable reduction theorem}
\label{stabsec} The stable reduction theorem is a fundamental result which has been proved in many ways, though
no easy self-contained proof is known. The author knows about six published proofs of the stable reduction
theorem \cite{DM}, \cite{AW}, \cite{Gi}, \cite{BL1}, \cite{Put} and \cite{Sa}, and a new proof is presented in
the paper. It seems natural to systemize different proofs and we try to do this below.

All proofs are naturally divided to three types. The proof of Gieseker in \cite{Gi} is the only proof of the
first type. It is based on the geometric invariant theory. One constructs moduli spaces of stable curves by
global projective methods then the stable reduction theorem is obtained as a by-product.

Three direct proofs perform the main work in the framework of non-Archimedean analytic geometry. They apply to
any complete valuation ring of height one and construct a semi-stable modification similarly to
desingularization of a surface. The proofs of Bosch-L\"utkebohmert, \cite{BL1}, and van der Put, \cite{Put}, are
close to Lipman's desingularization of surfaces. Arithmetic genus plays an important role in these proofs. Our
proof is an analog of Zariski's desingularization of surfaces. It is of rather valuation-theoretic nature, and
the arithmetic genus (and sheaves $R^1f_*(\calO_X)$) shows up only when we want to contract a semi-stable
modification to the stable one.

The proofs of the third type apply to discrete valuation rings. One uses desingularization of surfaces as a
(non-trivial!) starting point. If it is known that $X'=X\times_S S'$ admits a stable modification, where
$S'=\Spec(R')$ and $R'$ is the integral closure of $R$ in a finite separable extension $K'/K$, then such a
desingularization of $X'$ is a required semi-stable modification. The extension $K'/K$ is specified via the
action of the Galois group $G_K$ on an appropriate invariant of $X$. It is the group of $l$-torsion points,
unless $l=2$, of the relative generalized Jacobian $J_{X/S}$ in the Deligne-Mumford-Grothendieck or
Artin-Winters approaches ($K'$ is chosen so that it splits the $l$-torsion of the Jacobian), or the \'etale
cohomology group $H^1(X_\oeta,\bfQ_l)$ in Saito's approach ($K'$ is chosen so that $G_{K'}$ acts unipotently on
this cohomology group).

\section{Curves over separably closed fields} \label{B}

The material of this section is rather standard, so we give sketched proofs only. We assume that $k$ is a
separably closed field and $S=\Spec(k)$. Let $C$ be a proper connected geometrically reduced $S$-curve and
$\pi:\tilC\to C$ be its normalization. For any point $x\in C$ we define a number $g(x)$ as follows: if $x\in
C^0$, where $C^0$ is the set of generic points of $C$, then $g(x)$ is the geometric genus $h^1(\calO)$ of its
irreducible component; if $x\in C\setminus C^0$ then $g(x)$ is the dimension of the $k$-vector space
$\calO_{\tilC,\tilx}/\calO_{C,x}$, where $\tilx=\pi^{-1}(x)$. In particular, $g(x)=0$ if and only if either $x$
is a regular closed point or $x$ is the generic point of a rational irreducible component.

A point $x\in C$ is called an {\em ordinary $n$-fold point} if the completed local ring $\hatcalO_{C,x}$ is
isomorphic to $k[[T_1\. T_n]]/(\{T_iT_j\}_{i\neq j})$, and for $n=2$, $x$ is called an {\em ordinary double
point}. An ordinary $n$-fold point $x$ is a $k$-point and $\tilx=\{\tilx_1\. \tilx_n\}$ where each $\tilx_i$ is
a smooth $k$-point. Furthermore, Zariski locally at $x$ the curve $C$ is obtained from $\tilC$ by gluing the
points $\tilx_i$ to a single point, i.e. $\calO_{C,x}$ is the subring of $\calO_{\tilC,\tilx}$ whose elements
satisfy $f(\tilx_1)=\dots=f(\tilx_n)$. It follows easily from this description that if $x$ is the only
non-regular point in $C$ then $C$ is the pushout of the diagram $\tilC\la\tilx\to x$. Alternatively, one can
describe ordinary points as follows: any $k$-point $x$ satisfies $g(x)\ge|\tilx|-1$ (where $|\tilx|$ is the
cardinality of $\tilx$) and the equality holds if and only if $x$ is ordinary. Indeed,
$\calO_{\tilC,\tilx}/\calO_{C,x}$ admits a natural surjective $k$-linear map $\phi$ onto $k(\tilx)/k(x)$ (where
$k(\tilx)$ is the quotient of the semi-local ring $\calO_{\tilC,\tilx}$ by its radical), hence the inequality is
satisfied. If $x$ is ordinary then $\phi$ is bijective and $k(\tilx)\toisom k^n$, hence the exact equality
holds. Conversely, if $g(x)=|\tilx|-1$ then $k(\tilx)/k(x)$ is of dimension at most $|\tilx|-1$, hence $\tilx$
consists of $k$-points. By the pushout property, the normalization $\tilC\to C$ factors through a curve $C'$
obtained from $\tilC$ by gluing all points of $\tilx$ to a single point $x'$. We proved that $g(x')=|\tilx|-1$,
hence $g(x')=g(x)$ and it follows that the subrings $\calO_{C',x'}$ and $\calO_{C,x}$ of $\calO_{\tilC,\tilx}$
coincide. Since $x'$ is ordinary, $x$ is ordinary too. Note also that this characterization of ordinary points
implies that a $k$-point $x$ is ordinary if and only if its preimage $x^a\in C\otimes_k k^a$ is ordinary. (For
the sake of completeness we remark also that a $k$-point $x$ is ordinary if and only if $C$ is semi-normal at
$x$, i.e. $\Spec(\calO_{C,x})$ does not admit non-trivial bijective finite modifications.)

We say that a curve $C$ is {\em ordinary} (resp. {\em semi-stable}), if all its non-smooth points are ordinary
(resp. ordinary double points). We claim that $C$ is ordinary if and only if $C^a=C\otimes_k k^a$ is ordinary.
The direct implication is obvious and to prove the converse one it suffices to show that if the preimage $x^a\in
C^a$ of a point $x$ is an ordinary singular point then $x$ is such a point too. Note that $g(x^a)\ge g(x)$ and
the equality holds if and only if $\tilC\otimes_k k^a$ is normal over $x$. Also, $|\tilx|$ is the number of
valuation rings of $k(C)$ centered on $\calO_{C,x}$ and any such ring extends uniquely through the purely
inseparable extension $k^a/k$. Hence $|\wt{x^a}|=|\tilx|$ and we obtain that $|\tilx|-1=|\wt{x^a}|-1=g(x^a)\ge
g(x)$, and hence $\tilx$ is ordinary. A multipointed curve $(C,D)$ is called {\em ordinary} (resp. {\em
semi-stable}) if $C$ is ordinary (resp. semi-stable) and $D$ is a union of smooth $k$-points. As usual,
$p_a=1-h^0(\calO_C)+h^1(\calO_C)=h^1(\calO_C)$ denotes the arithmetic genus of $C$.

\begin{lemsect}\label{lastlem}
The equality $p_a(C)=1-|C^0|+\sum_{x\in C}g(x)$ holds.
\end{lemsect}
\begin{proof}
The normalization morphism $\pi$ is affine, hence we have an isomorphism $H^i(C,\pi_*\calO_\tilC)\toisom
H^i(\tilC,\calO_\tilC)$. Also, the sheaf $\calF=\pi_*\calO_\tilC/\calO_C$ is a skyscraper because $\pi$ is an
isomorphism over non-closed points. As a consequence, we obtain the following exact sequence
$$
0\to H^0(\calO_C)\to H^0(\calO_\tilC)\to H^0(\calF)\to H^1(\calO_C)\to H^1(\calO_\tilC) \to 0.
$$
Since $k=k^s$ and $C$ is geometrically reduced, $k$ is algebraically closed in the function field of each
irreducible component of $C$. In particular, $H^0(\calO_\tilC)$ is the direct sum of $|C^0|$ copies of $k$. The
Lemma follows now by computing the dimensions of the cohomology groups.
\end{proof}

The Lemma has the following Corollary which will serve us in applications. Let $Z$ be a connected proper
semi-stable $k$-curve. If $p_a(Z)=0$ then we say that $Z$ is a {\em $\bfP^1_k$-tree}. This condition is
equivalent to requiring that the irreducible components of $Z$ are isomorphic to $\bfP^1_k$, and the incidence
graph of $Z$ is a tree (vertex per generic point and edge per double point). If $C$ is a curve then by the
boundary $\partial(Z)=\partial_C(Z)$ of a closed subscheme $Z$ we mean the intersection of $Z$ with the Zariski
closure of its complement.

\begin{corsect}
\label{stabdowncor} Let $f:C'\to C$ be a morphism of proper geometrically reduced $k$-curves such that
$p_a(C')=p_a(C)$. Assume that $x\in C$ is a point such that $Z=f^{-1}(x)$ is a connected curve and
$f:C'\setminus Z\to C\setminus x$ is an isomorphism. Let us assume also that $C'$ is semi-stable at all points
of $Z$. Then $x$ is an ordinary $n$-fold point if and only if $Z$ is a $\bfP^1_k$-tree and $\partial(Z)$
contains exactly $n$ points.
\end{corsect}
\begin{proof}
Apply the genus formula of Lemma \ref{lastlem} to compute $p_a(C')$ and $p_a(C)$. Since the genera are equal and
$C'\setminus Z\toisom C\setminus\{x\}$, we obtain that $g(x)=\sum_{z\in Z}g_{C'}(z)-|Z^0|$. Set
$n=|\partial_{C'}(Z)|$, then the number of double points on $Z$ equals to $n+l$ where $l$ is the number of
non-boundary double points of $Z$. If $\Gamma$ is the incidence graph of $Z$ then $h^1(\Gamma)=l-|Z^0|+1$. So,
$g(x)=\sum_{z\in Z}g_{C'}(z)-|Z^0|=n+l+\sum_{z\in Z^0}g_{C'}(z)-|Z^0|=n-1+h^1(\Gamma)+\sum_{z\in
Z^0}g_{C'}(z)\ge n-1$, and the equality holds if and only if $h^1(\Gamma)=\sum_{z\in Z^0}g_{C'}(z)=0$. The
latter means that $\Gamma$ is a tree and the genera are all zero, i.e. $Z$ is a $\bfP^1_k$-tree. Thus,
$g(x)=n-1$ if and only if $Z$ is a $\bfP^1_k$-tree. It only remains to show that $|\tilx|$ equals to
$n=|\partial_{C'}(Z)|$. We have an obvious embedding of normalizations $\tilC\into\tilC'$, hence a morphism
$h:\tilC\to C'$ arises. Clearly, $h$ maps $\tilx$ surjectively onto $|\partial_{C'}(Z)|$, hence it is enough to
show that $h$ is injective on $\tilx$. If $y\in C'$ is the image of two different points of $\tilx$ then $y$ is
necessarily a double point of $C'$ and hence no component of $Z$ can pass through $y$. But then $\{y\}$ is a
connected component of $Z$, that contradicts $Z$ being a connected curve.
\end{proof}

In the sequel, by $(f,f_D):(C',D')\to(C,D)$ we denote a proper surjective morphism of connected geometrically
reduced multipointed $k$-curves. We say that an irreducible component $Z$ of $C'$ is {\em exceptional} if $Z$
lies in the semi-stable locus of $(C',D')$, is isomorphic to $\bfP^1_k$, is contracted in $C$, and contains at
most two points $x$ from $D'\cup C'_{\sing}$ (i.e. either $x$ is on the divisor or $x$ is an ordinary double
point). Note that the pushout $C''$ of the diagram $\ol{C'\setminus Z}\la\partial_{C'}(Z)\to\Spec(k)$ is also
the pushout of the diagram $C'\la Z\to\Spec(k)$. Let $D''$ be the image of $D'$ in $C''$, then $D'\toisom D''$,
$f$ factors through $C''$ and the image of $Z$ is a point contained in the semi-stable locus of $(C'',D'')$. We
say that $(C'',D'')$ is obtained from $(C',D')$ by {\em contracting} $Z$. Contracting exceptional components
successively, we construct a surjective proper morphism $(\oC,\oD)\to(C,D)$ which has no exceptional components.
We call such $(\oC,\oD)$ a {\em stable blow down of $(f,f_D)$}.

\begin{lemsect}
\label{stabblowdown} A stable blow down of $(f,f_D)$ is unique up to unique isomorphism.
\end{lemsect}
\begin{proof}
Let $Y\subset C'$ be a $\bfP^1_k$-tree lying in the semi-stable locus of $(C',D')$. We say that $Y$ is
exceptional if it is contracted to a point of $C$ and $|Y\cap D'|+|\partial_{C'}(Y)|\le 2$. We claim that given
two exceptional $\bfP^1_k$-trees $Y_1$ and $Y_2$ with non-empty intersection, their union is an exceptional
$\bfP^1_k$-tree too. Indeed, if none of them is contained in the other then there exists a double point $x$ such
that one irreducible component through $x$ is in $Y_1$ but not in $Y_2$, and the other one is in $Y_2$ but not
in $Y_1$. So, $x$ is contained in $\partial_{C'}(Y_i)$ for $i=1,2$ but not in $\partial_{C'}(Y_1\cup Y_2)$ and
therefore $|\partial_{C'}(Y_1\cup Y_2)|\le |\partial_{C'}(Y_1)|+|\partial_{C'}(Y_2)|-2$ and the claim follows.
Thus, $C'$ contains disjoint exceptional $\bfP^1_k$-trees $Y_1\. Y_n$ such that any exceptional $\bfP^1_k$-tree
is contained in some $Y_i$. Now it is clear that the stable blow down of $f$ is determined up to an isomorphism
by the property that it contracts each $Y_i$ to a point. Indeed, by Corollary \ref{stabdowncor} any semi-stable
blow down $C'\to\oC'$ contracts few disjoint exceptional $\bfP^1_k$-trees, hence the stable blow down $C'\to\oC$
factors through $\oC'$.
\end{proof}

For the sake of completeness, we note that analogous stabilization lemma holds in the absolute situation (i.e.
we contract a multipointed $k$-curve $(C',D')$ without specified base curve $(C,D)$) when one of the following
cases holds: $p_a(C')\ge 2$, or $p_a(C')=1$ and $|D'|\ge 1$, or $p_a(C')=0$ and $|D'|\ge 3$. Note also that one
can similarly construct an ordinary blow down of $C'$, but it is not unique in general. In particular, there is
no minimal ordinary modification of relative curves.

\end{document}